\documentclass[hyperref]{ctexart}

\usepackage[english]{babel}

\usepackage[letterpaper,top=2cm,bottom=2cm,left=3cm,right=3cm,marginparwidth=1.75cm]{geometry}

\usepackage{helvet}
\usepackage{booktabs}
\usepackage{amsthm}
\usepackage[titletoc]{appendix}
\usepackage{float}
\usepackage{amsmath}
\allowdisplaybreaks
\usepackage{graphicx}
\usepackage{amssymb}
\usepackage{subcaption}
\usepackage{multirow}
\usepackage{algorithmic}
\usepackage{tikz}
\usepackage{authblk}
\usepackage{placeins} 
\usepackage{afterpage} 

\usepackage[ruled,vlined]{algorithm2e}  % 算法包
\newtheorem{thm}{\textbf{Theorem}}
\newtheorem{lem}{\textbf{Lemma}}

\newtheorem{rmk}{\textbf{Remark}}
\newtheorem{assumption}{Assumption}
\usepackage[colorlinks=true, allcolors=blue]{hyperref}

\title{Bridging Scales in Chemotaxis: Scale-Uniform Forward Stability for Run-and-Tumble Kernel Estimation}
\author[1]{Jos\'e A. Carrillo}
\author[2]{Jiangjun Ma}
\author[3]{Min Tang}

\affil[1]{Mathematical Institute, University of Oxford, Woodstock Road, Oxford, OX2 6GG, UK. (jose.carrillo@maths.ox.ac.uk)}
\affil[2]{School of Mathematical Sciences, Institute of Natural Sciences, Shanghai Jiao
Tong University, Shanghai, P.R. China. (dreamheaven@sjtu.edu.cn).
}
\affil[3]{School of Mathematical Sciences, Institute of Natural Sciences and MOE-LSC, Shanghai Jiao
Tong University, Shanghai, P.R. China. (tangmin@sjtu.edu.cn).
}

\begin{document}
\maketitle

\begin{abstract}
Chemotactic motion is described by run-and-tumble kinetic models at microscopic scales and by Keller--Segel equations at macroscopic scales. We develop variational loss functionals for estimating the two components \(T_0(x)\) and \(T_1(x)\) of a turning kernel \(T_{\epsilon}=T_0+\epsilon T_1\), where $T_0$ determines the leading-order turning rate and diffusion, while $T_1$ governs the macroscopic chemotactic drift. Under suitable regularity and data-informativeness assumptions, we establish conditional scale-uniform forward-stability estimates showing that a small loss leads to a small discrepancy between the forward solutions generated by the true and estimated kernels across the kinetic and diffusive regimes. Combined with sparse inversion, the method accurately recovers smooth, nonsmooth, and strongly heterogeneous kernels and remains robust under measurement noise.
\end{abstract}
\section{Introduction}

Chemotaxis, the directed migration of cells or organisms in response to chemical gradients \cite{willard2006signaling}, is a fundamental biological process observed across a wide range of species, from unicellular bacteria to multicellular organisms \cite{willard2006signaling, de2016neutrophil}.
A widely adopted mathematical model of chemotaxis was introduced by Keller and Segel \cite{keller1970initiation, keller1971model, keller1971traveling}, which describes the time evolution of cell density $\rho(x,t)$ via the following equation:
\begin{equation}\label{Keller_Segel}
        \partial_t \rho + \nabla \cdot (\bar\Gamma(S,\nabla S) \rho - \bar D(S,\nabla S)\cdot \nabla \rho) = 0. 
\end{equation}
Here, $S(x,t)$ denotes the concentration of the chemical attractant. The model accounts for the cellular response to the chemical stimulus through a drift coefficient $\bar\Gamma(S,\nabla S)$ and a diffusion coefficient $\bar D(S,\nabla S)$, both of which depend in a generally nonlinear manner on $S$ and its spatial gradient $\nabla S$.  

Experimental studies---such as the pioneering work by Berg \cite{berg1972chemotaxis} tracking \textit{E. coli} trajectories---have shown that individual bacterial movement alternates between two distinct states: a ``run,'' characterized by straight-line swimming at velocity \( v \), and a ``tumble,'' which involves a random reorientation that results in a new velocity \( v' \). To describe this behavior, run-and-tumble kinetic models have been developed \cite{stroock1974some,alt1980biased,othmer1988models}. These mesoscopic models capture chemotactic dynamics through the following transport equation:
\begin{equation}\label{kinetic_theory}
    \partial_t f + v\cdot\nabla_x f = \int_V \bigl[\bar T(v',v,S,\nabla S)f' -\bar T(v,v',S,\nabla S)f \bigr]dv',
\end{equation}
where \( f(x,t,v) \) represents the density of cells at position \( x \in \mathbb{R}^3 \), moving with velocity \( v \in V \subseteq \mathbb{S}^2 \) at time \( t \geq 0 \), and \( f' := f(x,t,v') \). The tumbling kernel \( \bar T(v',v,S,\nabla S) \) gives the probability that cells moving with velocity $v'$ reorientate to velocity $v$. Thus $\int_V\bar T(v,v',S,\nabla S)dv'$ encodes the reorientation rate of cells moving with velocity $v$, which may depend on the local chemoattractant concentration \( S \) and its spatial gradient.

The Keller--Segel model \eqref{Keller_Segel} and the kinetic model \eqref{kinetic_theory} describe chemotaxis behavior at different temporal and spatial scales. Changes in cell density occur over time scales much longer than the typical run duration of a bacterium. For \textit{E. coli}, for instance, runs last roughly one second, while experimental observations of the density evolution often extend over several minutes \cite{patteson2015running, othmer2000diffusion}. To connect these scales, one may apply a parabolic scaling to the kinetic equation by introducing the rescaled coordinates \( \tau = \epsilon^2 t \) and \( \xi = \epsilon x \), yielding:
\begin{equation}\label{kinetic_theory_diffusion}
    \partial_{\tau} f_{\epsilon} + \frac{1}{\epsilon} v\cdot\nabla_{\xi} f_{\epsilon} = \frac{1}{\epsilon^2}\int_V \bigl[ \bar T_{\epsilon}(v',v,S,\nabla S)f'_{\epsilon} - \bar T_{\epsilon}(v,v',S,\nabla S)f_{\epsilon} \bigr]dv'.
\end{equation}
According to \cite[Lemma 1]{chalub2004kinetic}, one can further assume that  the tumbling kernel expands as 
\begin{equation} \label{eq:T0T1bar}\bar T_{\epsilon}(v',v,S,\nabla S) = \bar T_0 (v',v,S,\nabla S)+ \epsilon \bar T_{1} (v',v,S,\nabla S),
\end{equation} 
where  $\bar T_0(-v,v',\cdot) = \bar T_0(v,-v',\cdot)$, $\bar T_1(-v,v',\cdot) = -\bar T_1(v,-v',\cdot)$, which means $T_0$ is symmetric and $T_1$ is antisymmetric. When $\bar T_0(v',v,\cdot), \bar T_1(v',v,\cdot)=O(1)$, in the limit \( \epsilon \ll 1 \), the bacterial density \( \rho_\epsilon(x,t) := \int_V f_\epsilon(x,t,v) dv \) converges to a solution of the Keller--Segel equation \eqref{Keller_Segel} \cite{othmer2000diffusion, chalub2004kinetic, othmer2002diffusion}. In the asymptotic expansion \eqref{eq:T0T1bar}, the symmetric part $\bar T_0$ governs the effective diffusion coefficient $\bar D(S,\nabla S)$, while the antisymmetric correction $\bar T_1$ determines the advection coefficient $\bar\Gamma(S,\nabla S)$.
Nevertheless, it has been noted in \cite{saragosti2011directional} that certain experimental observations are better explained by the detailed kinetic description \eqref{kinetic_theory} than by the macroscopic Keller--Segel approximation, highlighting the complementary roles of these two modeling frameworks.

A critically important issue of using models \eqref{Keller_Segel} and \eqref{kinetic_theory} to fit the experimental data is that their coefficients---such as $\bar D$, $\bar\Gamma$, and $\bar T$---depend on both the cell species and the types of chemoattractant. 
 In practice, explicit forms for $\bar D$, $\bar\Gamma$, and $\bar T$ are known only for a few well-studied bacteria, such as  the complete proteome of \textit{E. coli} \cite{kalwarczyk2012biologistics}. In most cases, however, how $\bar D$, $\bar\Gamma$, and $\bar T$ depend on $(S, \nabla S)$ remain unknown.
The first way to determine the functional dependence of $\bar D$, $\bar\Gamma$, and $\bar T$ on $(S, \nabla S)$ is from a bottom up approach starting from the chemotactic signaling-pathway of individual cells as have been done for E.coli chemotaxis \cite{xue2021individual, xue2025crossover}. However, for general microorganisms, when the signaling pathway is not known, one can only determined them empirically. For example, in \cite{giometto2015generalized}, by relating $\bar \Gamma$ to the macroscopic steady-state bacterial density, the functional form of $\bar\Gamma$ is determined through a systematic model selection process based on the Akaike Information Criterion (AIC) and data fitting.

To determined $\bar D$, $\bar\Gamma$, and $\bar T$ empirically, the first step is to find their spatial dependence given a fixed chemical concentration field \( S(x) \).
 Recovering a spatially dependent tumbling kernel \(  T(x, v', v) = \bar T(v', v, S, \nabla S) \) from the time evolution of the probability density function \( f(x, t, v) \) would require simultaneously tracking the trajectories of many individual cells and performing the necessary statistical analysis. On the other hand, if the time evolution of cell density can be modeled by the Keller--Segel model, one must recover the effective drift and diffusion fields \( \Gamma(x) = \bar\Gamma(S, \nabla S) \) and \(  D(x) =\bar D(S, \nabla S) \) from measurements of the density evolution \( \rho(x, t) \), without the need to track individual cells. In both scenarios, PDE-constrained optimization is a widely adopted framework for solving such inverse problems.

For the Keller--Segel model \eqref{Keller_Segel}, PDE-constrained optimization has been employed in \cite{karalashvili2011identification,fister2008identification} to determine \(D(x)\) and \(\Gamma(x)\). In \cite{hellmuth2025reconstructing}, macroscopic particle density information is leveraged to reconstruct the tumbling kernel based on \eqref{kinetic_theory}. However, for a given species and chemical environment \(S(x)\), the appropriate movement description may shift between microscopic (kinetic) and macroscopic (Keller--Segel) scales depending on the magnitude of the chemical gradient \(|\nabla S|\) or when \(f(x,t,v)\) is far from local equilibrium. In some regions, the detailed run-and-tumble model \eqref{kinetic_theory} may be necessary, while in others, the macroscopic Keller--Segel approximation \eqref{Keller_Segel} may suffice. It is difficult to establish a quantitative criterion for determining where the kinetic description should be retained and where the Keller--Segel approximation is valid; there need not be a sharply defined physical interface between the two regimes.

The core multiscale challenge, can be approached from two complementary perspectives. The first is an identifiability problem: one may ask whether the leading-order turning component \(\bar{T}_0\) and the anisotropic correction \(\bar{T}_1\) can be uniquely and uniformly recovered from observations of \(f(x,t,v)\) as \(\epsilon\) varies. The second is a predictive-consistency problem: using the experimentally available density and flux data, one may instead seek admissible candidate components that reproduce the observed dynamics across both the kinetic and macroscopic regimes, without requiring or establishing uniqueness of the underlying turning kernel. The present work follows the second perspective, focusing on loss functionals that select candidate kernels whose induced density and flux fields remain consistent with the observations across scales.

Inverse problems involving kinetic equations and their diffusion-type limit equations have been studied previously in the context of optical tomography and photoacoustic tomography. Depending on the optical properties of the material, when the mean free path $\epsilon$ of the particle becomes small, the model equation shifts from a transport equation to a diffusion-type model. Since $\epsilon$ is not known in general, an important aspect is stability as $\epsilon\rightarrow0$. Optical tomography needs to reconstruct the internal optical properties of an object using non-invasive boundary measurements \cite{bal2008inverse, lai2019inverse}. As noted in \cite{chen2018stability}, for inverse problems in optical tomography, the reconstruction error is amplified as $\epsilon\to 0$ and the ill-posedness in the diffusive limit is investigated. Photoacoustic tomography reconstructs internal optical properties from interior measurements \cite{bal2010inverse, lai2022inverse}. When such interior data are available, both the transport model and its diffusion‑limit formulation yield theoretically stable reconstructions in photoacoustic imaging \cite{bal2010inverse, lai2022inverse}.
 Furthermore, multiscale inverse problems solved via Bayesian inference have become an active research area in recent years. In \cite{hellmuth2021multiscale}, the authors show that as the parameter $\epsilon$ tends to zero, the two Bayesian posterior distributions—one derived from the microscopic chemotaxis kinetic equation and the other from the macroscopic Keller–Segel model—converge uniformly under suitable metrics. However, the construction of data-driven loss functionals whose induced forward predictions remain stable across the kinetic-to-diffusive transition has received comparatively little attention.

 When the turning kernel is independent of the post-tumble velocity, the expansion reduces to
\[
T_{\epsilon}(v',x) = T_0(x) + \epsilon v' T_1(x).
\]
The objective of this work is to construct practically implementable variational loss functionals for estimating one of the two kernel components when the other is known. Because \(T_0\) and \(T_1\) enter the kinetic equation at different asymptotic orders, their effects on the observed density and flux are expressed on different spatial and temporal scales. Rather than proving unique identifiability of \(T_0\) and \(T_1\), we investigate a conditional loss-to-solution stability property. We design loss functionals such that, under suitable regularity and data-informativeness assumptions, a small loss implies that the forward solution generated by the estimated component remains close to that generated by the true component. By combining kinetic estimates with estimates for the Keller--Segel limit, we obtain a conditional scale-uniform stability argument across the kinetic-to-diffusive transition. 

The resulting framework estimates \(T_1\) given \(T_0\), or \(T_0\) given \(T_1\), from observed density and flux data. Distinct kernel components may produce identical or nearly identical forward dynamics, so the present results do not establish uniqueness of the recovered parameters. Nevertheless, the proposed losses provide a principled way to select kernel candidates with accurate predictive behavior. Numerical experiments combined with sparse inversion illustrate the reconstruction of several dictionary-sparse smooth, nonsmooth, and strongly heterogeneous kernels over a range of constant values of \(\epsilon\), including tests with synthetic measurement noise.

The remainder of this paper is organized as follows. In Section \ref{sec:Model and the Loss Function}, we introduce the mathematical setup of the inverse problem under consideration and propose two loss functions for recovering \(T_0(x)\) and \(T_1(x)\), respectively. In Section \ref{sec:Stability Analysis}, we present the main theoretical results along with a formal argument for uniform stability. The proofs of the main theorems are provided in Section 4. In Section \ref{sec:Numerical results}, we present numerical examples demonstrating the performance of the proposed loss functions. Finally, in Section \ref{sec:Discussion}, we conclude with a discussion and offer future perspectives.

\section{Model and the Loss Function}
\label{sec:Model and the Loss Function}
\subsection{The two flux model}

The full kinetic model \eqref{kinetic_theory_diffusion} is defined for \((x,v) \in \mathbb{R}^3 \times \mathbb{S}^2\). Because microfluidic experiments are often carried out in narrow, elongated channels that effectively confine bacterial motion to quasi-one-dimensional geometries, the two-velocity model has become a standard simplification in kinetic descriptions of chemotaxis. \cite{saragosti2011directional,si2014pathway}. In the two velocity model, the bacteria movement can be considered restricted to one dimension, then \((x,v) \in \mathbb{R} \times \{-1, +1\}\). For a fixed chemical concentration \(S(x)\) inside the channel, one can then observe and analyze bacterial migration behavior in response to the stimulus \cite{ahmed2008experimental,perez2022microfluidic}. After introducing parabolic scaling, the simplified one-dimensional two-flux kinetic model for chemotaxis is written as
\begin{equation*}
    \begin{aligned}
    \epsilon^2 \partial_t f^+ + \epsilon v\partial_x f^+ &= T_{\epsilon}(-v,v,S,\nabla S)f^- - T_{\epsilon}(v,-v,S,\nabla S)f^+,\\
    \epsilon^2 \partial_t f^- - \epsilon v\partial_x f^- &= T_{\epsilon}(v,-v,S,\nabla S)f^+ - T_{\epsilon}(-v,v,S,\nabla S)f^-,
    \end{aligned} 
\end{equation*}
with the boundary conditions
\begin{equation}\label{boundarycond}
\lim_{|x|\to \infty} f^{\pm}(x,t) = 0,\quad \forall t\geq 0.
\end{equation}
Here, \(f^\pm(x,t)\) denotes the number of bacteria moving with constant velocity \(\pm v\), where \(v\) is a positive scalar. Initial data are prescribed at \(t = 0\) such that \(f^\pm(x,0) = \phi^\pm(x)\). Furthermore, we assume that the turning kernel admits an asymptotic expansion of the form given in \eqref{eq:T0T1bar}.

For simplicity, we consider a fixed \(S(x)\), so that \(T_0\) and \(T_1\) depend only on the spatial variable \(x\). The two-flux model then becomes
\begin{equation}\label{twoflux}
    \begin{aligned}
    \epsilon^2 \partial_t f^+ + \epsilon v\partial_x f^+ &= (T_0(x)-\epsilon T_1(x))f^- - (T_0(x)+\epsilon T_1(x))f^+,\\
    \epsilon^2 \partial_t f^- - \epsilon v\partial_x f^- &= (T_0(x)+\epsilon T_1(x))f^+ - (T_0(x)-\epsilon T_1(x))f^-,
    \end{aligned} 
\end{equation}
where the form follows from the symmetry of $T_0$ and antisymmetry of $T_1$.
Here, \(T_0(x) > 0\) and \(T_1(x) \in \mathbb{R}\), with the physical constraints \(T_0(x) \pm \epsilon T_1(x) > 0\).

Let $$\rho(x,t)=f^+(x,t)+f^-(x,t),\qquad J(x,t)=\frac{f^+(x,t)-f^-(x,t)}{\epsilon}.$$  Then the model equation \eqref{twoflux} can be written as
\begin{subequations}\label{twoflux_1}
    \begin{align}
        \partial_t \rho + v &\partial_x J = 0,\label{twoflux_1_1}\\
        \epsilon^2 \partial_t J + v\partial_x \rho = -&2T_0(x)J - 2 T_1(x)\rho \label{twoflux_1_2}.
    \end{align}
\end{subequations}
The initial data are $\rho(x,0)=\phi^+(x,0)+\phi^-(x,0)$, $J(x,0)=\frac{\phi^+(x,0)-\phi^-(x,0)}{\epsilon}$.
Since $f^+$ and $f^-$ represent the densities of bacteria moving in two distinct velocity directions, $\rho$ can be interpreted as the bacterial density, while $J$ corresponds to the flux. As a result, $\rho$ must be nonnegative over the entire spatial domain, whereas the sign of $J$ cannot be determined apriori. 
When $\epsilon \rightarrow 0$, denote $\rho_0 := \lim_{\epsilon\rightarrow 0} \rho$, $J_0 := \lim_{\epsilon\rightarrow 0} J$. From \eqref{twoflux_1_2}, one has 
\begin{equation}\label{eq:limitJrho}
    J_0 = -\frac{v\partial_x \rho_0}{2T_0} - \frac{T_1 \rho_0}{ T_0}.
\end{equation}Substituting \eqref{eq:limitJrho} into equation \eqref{twoflux_1} gives the following limiting equations
\begin{subequations}\label{limiteqtion}
    \begin{align}
        \partial_t \rho_0 + v \partial_x \left[ \frac{-v\partial_x \rho_0}{2T_0} - \frac{T_1}{T_0} \rho_0 \right] = 0, \label{limiteqtion_1}\\
        v\partial_x \rho_0 = -2T_0(x)J_0 - 2 T_1(x)\rho_0,
        \label{limiteqtion_2}
    \end{align}
\end{subequations}
determining $\rho_0$ and $J_0$.

\subsection{Loss functions}
This section is dedicated to recovering \(T_0(x)\) and \(T_1(x)\) using information from the known solutions \(f^\pm(x,t) \in C^1(\mathbb{R}\times [0,T])\). Following the variational approach used in \cite{carrillo2025sparse}, the loss function consists of two parts. The first part quantifies the fidelity of the estimate via the squared error between the true and estimated function values. The second part is constructed exclusively from known quantities and is designed to cancel out the unknown true value appearing in the first part, thereby converting the original optimization problem into a solvable one. Subsequently, we propose two different loss functions: one for recovering \(T_1\) given \(T_0\), and the other for recovering \(T_0\) given \(T_1\).

\subsubsection{Recover $T_1$}
Assume that we know the solution \((\rho, J) \in \mathbb{R} \times [0, T]\) of equation \eqref{twoflux_1} and the function \(T_0\). Let \(\mathcal{E}[\hat{T}_1]\) be the loss function for recovering \(T_1\). The general form of the optimization problem is then written as:
\begin{equation*}
    \mathop{\min}\limits_{\hat{T}_1} \ \mathcal{E}[\hat{T}_1] =\mathop{\min}\limits_{\hat{T}_1} \frac{1}{T}\int_0^T \int_\mathbb{R} |\hat{T}_1 - T_1|^2 \, E(\rho, J, T_0) \, dx \, dt,
\end{equation*}
where the specific choice of $E(\rho,J,T_0)$ is to be determined below.
Expanding the quadratic terms in the loss functional yields
\begin{align}
    \mathcal{E}[\hat{T}_1] &= \frac{1}{T}\int_0^T \int_\mathbb{R} |\hat{T}_1 - T_1|^2 \, E(\rho, J, T_0) \, dx \, dt \nonumber \\
    &= \frac{1}{T}\int_0^T \int_\mathbb{R} \left( \hat{T}_1^2 E(\rho, J, T_0) - 2\hat{T}_1 T_1 E(\rho, J, T_0) + T_1^2 E(\rho, J, T_0) \right) dx \, dt. \nonumber
\end{align}
The third term on the right-hand side, \(T_1^2 E\), contains only the true function \(T_1\) and known quantities \(E(\rho, J, T_0)\); therefore, it does not affect the minimum of the loss function. Consequently, we only need to minimize the following functional \(\tilde{\mathcal{E}}[\hat{T}_1]\) defined by
\begin{align}
    \tilde{\mathcal{E}}[\hat{T}_1] = \frac{1}{T}\int_0^T \int_\mathbb{R} \left( \hat{T}_1^2 E(\rho, J, T_0) - 2\hat{T}_1 T_1 E(\rho, J, T_0) \right) dx \, dt. \nonumber
\end{align}
It is clear that
\begin{equation*}
    \mathop{\arg\min}\limits_{\hat{T}_1} \, \mathcal{E}[\hat{T}_1] = \mathop{\arg\min}\limits_{\hat{T}_1} \, \tilde{\mathcal{E}}[\hat{T}_1].
\end{equation*}
However, the second term in \(\tilde{\mathcal{E}}[\hat{T}_1]\) contains the unknown true value \(T_1\). Therefore, it is necessary to design the specific form of \(E(\rho, J, T_0)\) such that the unknown quantity \(T_1\) in the second term can be expressed in terms of known quantities.

 Since \eqref{twoflux_1} allows us to express $2T_1(x) \rho$ as $-\epsilon^2 \partial_t J - v\partial_x \rho -2T_0(x)J$ (all known quantities),
we choose $E(\rho,J,T_0) = \frac{\rho}{T^2_0}$. 
 We ultimately recover $T_1$ by minimizing the following functional:
\begin{equation*}
    T_1\in \mathop{\arg\min}\limits_{\hat{T}_1} \ \tilde{\mathcal{E}}[\hat{T}_1],
\end{equation*}
with
\begin{equation}\label{loss_1}
    \tilde{\mathcal{E}}[\hat{T}_1] = \frac{1}{ T}\int_0^{T}\int_\mathbb{R}  \hat{T}_1^2(x)\frac{\rho}{T^2_0} + \frac{\hat{T}_1(x)}{T^2_0} \left[\epsilon^2 \partial_t J + v\partial_x \rho +2T_0(x)J \right] dxdt.
\end{equation} 
The advantage of choosing \(E(\rho, J, T_0) = \frac{\rho}{T_0^2}\) is that it enables the recovery of both the turning kernel \(T_1\) in the kinetic equation and the drift coefficient in the limiting equation \eqref{limiteqtion_1} by minimizing the same loss functional. We illustrate this advantage by examining the limit \(\epsilon \to 0\) in the loss function.

When \(\epsilon \to 0\), the term \(-2 T_1(x) \rho\) in \(\tilde{\mathcal{E}}[\hat{T}_1]\) becomes \(v \partial_x \rho_0 + 2 T_0(x) J_0\), and the loss function \eqref{loss_1} is given by
\begin{equation}\label{loss_limit}
    \tilde{\mathcal{E}}_0[\hat{T}_1] = \frac{1}{T} \int_0^{T} \int_\mathbb{R} \left( \hat{T}_1^2(x) \frac{\rho_0}{T_0^2} + \frac{\hat{T}_1(x)}{T_0^2} \left[ v \partial_x \rho_0 + 2 T_0 J_0 \right] \right) dx \, dt.
\end{equation}
By the same reasoning used for the kinetic model, minimizing \(\tilde{\mathcal{E}}_0[\hat{T}_1]\) is equivalent to minimizing
\begin{equation*}
    \mathcal{E}_0[\hat{T}_1] = \frac{1}{T} \int_0^T \int_\mathbb{R} |\hat{T}_1 - T_1|^2 \frac{\rho_0}{T_0^2} \, dx \, dt = \frac{1}{T} \int_0^T \int_\mathbb{R} \left| \frac{\hat{T}_1}{T_0} - \frac{T_1}{T_0} \right|^2 \rho_0 \, dx \, dt.
\end{equation*}
Since \(\frac{T_1}{T_0}\) is the drift coefficient in the limiting equation \eqref{limiteqtion_1}, a small value of \(\mathcal{E}_0[\hat{T}_1]\) indicates that the estimated drift coefficient \(\hat{\Gamma}(x) = -v \frac{\hat{T}_1}{T_0}\) is close to its true value \(\Gamma(x) = -v \frac{T_1}{T_0}\) in the weak sense.

Moreover, choosing $E(\rho,J,T_0) = \frac{\rho}{T^2_0}$ is beneficial to the uniform stability analysis in Section 3, while numerically, $E(\rho,J,T_0)=\rho, \frac{\rho}{T_0}$ are all potentially feasible forms of $E(\rho,J,T_0)$.

\subsubsection{Recover $T_0$}
When \(T_1\) is known, the general form of optimization problem for recovering \(T_0\) can be written as:
\begin{equation*}
    \mathop{\min}\limits_{\hat{T}_0} \ \mathcal{G}[\hat{T}_0] = \mathop{\min}\limits_{\hat{T}_0} \frac{1}{T}\int_0^T \int_\mathbb{R} |\hat{T}_0 - T_0|^2 \, G(\rho, J, T_1) \, dx \, dt,
\end{equation*}
where \(\rho\), \(J\), and \(T_1\) are all known quantities. Expanding the quadratic term yields:
\begin{align}\label{GT0}
    \mathcal{G}[\hat{T}_0] &= \frac{1}{T}\int_0^T \int_\mathbb{R} |\hat{T}_0 - T_0|^2 \, G(\rho, J, T_1) \, dx \, dt \nonumber\\
    &= \frac{1}{T}\int_0^T \int_\mathbb{R} \left( \hat{T}_0^2 G(\rho, J, T_1) - 2\hat{T}_0 T_0 G(\rho, J, T_1) + T_0^2 G(\rho, J, T_1) \right) dx \, dt.
\end{align}
The last term on the right-hand side of \eqref{GT0} depends on the true \(T_0(x)\) and the known quantities \(\rho, J, T_1\), and therefore plays no role in the minimization problem.

When we choose \(G(\rho, J, T_1) = J^2\), one can recover \(T_0\) by minimizing the following functional:
\begin{equation*}
    T_0 \in \mathop{\arg\min}\limits_{\hat{T}_0} \ \tilde{\mathcal{G}}[\hat{T}_0],
\end{equation*}
where
\begin{align}
    \tilde{\mathcal{G}}[\hat{T}_0] &= \frac{1}{T}\int_0^{T}\int_\mathbb{R} \left( \hat{T}_0^2(x) J^2 - 2 T_0(x) \hat{T}_0(x) J^2 \right) dx \, dt \nonumber\\
    &= \frac{1}{T}\int_0^{T}\int_\mathbb{R} \left( \hat{T}_0^2(x) J^2 - 2\hat{T}_0(x) J (T_0(x) J) \right) dx \, dt.
\end{align}
Using the form of equation \eqref{twoflux_1}, we can replace \(-2 T_0(x) J\) with \(\epsilon^2 \partial_t J + v \partial_x \rho + 2 T_1(x) \rho\). The loss function used to recover \(T_0\) then becomes:
\begin{align}\label{loss_2}
    \tilde{\mathcal{G}}[\hat{T}_0] = \frac{1}{T}\int_0^{T}\int_\mathbb{R} \left( \hat{T}_0^2(x) J^2 + \hat{T}_0(x) J \left( \epsilon^2 \partial_t J + v \partial_x \rho + 2 T_1(x) \rho \right) \right) dx \, dt.
\end{align}

In the limit \(\epsilon \to 0\), the above loss function becomes
\begin{equation*}
    \tilde{\mathcal{G}}_0[\hat{T}_0] = \frac{1}{T}\int_0^{T}\int_\mathbb{R} \left( \hat{T}_0^2(x) J_0^2 + \hat{T}_0(x) J_0 \left[ v \partial_x \rho_0 + 2 T_1(x) \rho_0 \right] \right) dx \, dt.
\end{equation*}
This is equivalent to minimizing
\begin{equation*}
    \mathcal{G}_0[\hat{T}_0] = \frac{1}{T}\int_0^{T}\int_\mathbb{R} |\hat{T}_0 - T_0|^2 J_0^2 \, dx \, dt,
\end{equation*}
which can be considered an appropriate loss function for recovering \(T_0\) in the limiting equation \eqref{limiteqtion}.

%As shown in the stability analysis of Section \ref{sec:Stability Analysis}, for \(G(\rho, J, T_1) = J^2\), uniform stability with respect to the multiscale parameter \(\epsilon\) can be proved rigorously. However, in actual numerical experiments, other functional forms---such as \(G(\rho, J, T_1) = (\partial_t J)^2\)---are also potentially feasible.

\section{Main theorems and the uniform stability}
\label{sec:Stability Analysis}

\subsection{Main Theorems}

 Assume that we know the time evolution of the solution for the time period $t\in[0,T]$. We denote 
 \[
 \|f\|_{L^2_x}^2 := \int_{\mathbb{R}} |f|^2 dx, \qquad\|f\|_{L^2_x, L^p_t}^p := \int_0^T \|f\|^p_{L^2_x} dt,\quad\mbox{for $p=1,2$}.
 \]
 Then we have the following stability results for the kinetic model \eqref{twoflux_1} (Theorem 1) and limit Keller-Segel model \eqref{limiteqtion_1} (Theorem 2). We only present the Theorems in this section and leave the detailed proof for the next section.

\begin{thm}\label{thm1}
Let $(\rho,J)$, $(\hat{\rho},\hat{J})$ be solutions to \eqref{twoflux_1} with the true $(T_0, T_1)$ and the estimated $(\hat{T}_0, \hat{T}_1)$ respectively. Assume that $\rho T_0^2$ is bounded on $\mathbb{R}\times[0,T]$, and $\hat T_0$, and $\hat T_1$ are bounded on $\mathbb{R}$. Then for $\bar{\rho}=\rho-\hat{\rho}$ and $\bar{J}=J-\hat{J}$, we have the following stability estimate:
\begin{equation*}
    \begin{aligned}
        \|\bar{\rho}(x,T)\|^2_{L^2_x}+\|\bar{J}(x,T)\|^2_{L^2_x} 
        \leq&   C_1 \|\bar{\rho}(x,0)\|^2_{L^2_x} + C_2 \|\bar{J}(x,0)\|^2_{L^2_x} + C_3\mathcal{E}[\hat{T}_1] + C_4\mathcal{G}[\hat{T}_0],
    \end{aligned}
\end{equation*}
where 
\begin{equation*}
    \mathcal{E}[\hat{T}_1] = \frac{1}{ T}\int_0^{T}\int_\mathbb{R} |\hat{T}_1 - T_1|^2 \frac{\rho}{T_0^2} dxdt,\qquad
    \mathcal{G}[\hat{T}_0] = \frac{1}{T}\int_0^{T}\int_\mathbb{R} |\hat{T}_0 - T_0|^2 J^2 dxdt.
\end{equation*}
Here $C_1,\ C_2,\ C_3,\ C_4 >0$ are constants depending on $\epsilon$ and $T$ diverging as $\epsilon\to 0$.
\end{thm} 

The proof of stability for the limiting Keller-Segel model is based on stochastic analysis methods using the Wasserstein-2 distance.
The limit equation \eqref{limiteqtion_1} can be written into the following Fokker-Planck type equation:
\begin{equation}\label{eq:FK}
    \partial_t \rho_0 + \partial_x \left[ \left(-\frac{vT_1}{T_0}-\frac{v^2\partial_x T_0}{2T_0^2}\right) \rho_0 \right]- \partial_x^2 \left[\frac{v^2}{2T_0} \rho_0\right] = 0.
\end{equation}
 By It\^o's formula, \eqref{eq:FK} can be interpreted as the evolution of probability density function of the following stochastic differential equation (SDE)
\cite{ikeda2014stochastic, carmona2016lectures}:
\begin{align}
    dX_t &= \left(-\frac{vT_1}{T_0}-\frac{v^2 \partial_x T_0}{2T_0^2}\right)(X_t)dt + \frac{v}{\sqrt{T_0}}(X_t)dB_t, \label{eq:dX_t}\\
    X_0 &= X^0\in L^2_x\ \text{independent of}\ (B_t)_{t\in[0,T]}.
\end{align}
Denote $\mu_t(x)$ as the probability measure of the density $\rho_0(x,t)$, which means $\rho_0(x,t)dx = d\mu_t(x)$. The Wasserstein-2 distance $d_2: \mathcal{P}^2(\mathbb{R})\times\mathcal{P}^2(\mathbb{R})\rightarrow\mathbb{R}$ between two measures $\mu,\xi\in\mathcal{P}^2(\mathbb{R})$ is defined as
\begin{equation*}
    d_2(\mu, \xi) := \min{\left\{ \int_{\mathbb{R}\times \mathbb{R}} |x-y|^2 d\gamma(x,y):\ \gamma\in\Pi(\mu,\xi) \right\}^{1/2}}.
\end{equation*}
Here $\mathcal{P}^2(\mathbb{R})$ denotes the space of probability measures with finite second moments; and for any $\mu,\xi \in \mathcal{P}(\mathbb{R})$ we consider $\Pi(\mu,\xi):=\{ \gamma\in \mathcal{P}^2(\mathbb{R}\times \mathbb{R})\ |\ (\pi_x)\#\gamma = \mu,\ (\pi_y)\#\gamma = \xi \}$ is the set of transport plans between the measures $\mu$ and $\xi$. 

Let $\hat{X}_t$ be the approximation of $X_t$ with parameters ($\hat{T}_0$, $\hat{T}_1$), which satisfies
\begin{equation}\label{eq:dhatX_t}
    d\hat{X}_t = \left(-\frac{v\hat{T}_1}{\hat{T}_0}-\frac{v^2\partial_x \hat{T}_0}{2\hat{T}_0^2}\right)(\hat{X}_t)dt + \frac{v}{\sqrt{\hat{T}_0}}(\hat{X}_t)dB_t,
\end{equation}
with the same Brownian motion as $B_t$. With these notations, we can then present our stability result for the limiting Keller-Segel equation.
\begin{thm}\label{thm2}
Let $\rho_0$ and $\hat{\rho}_0$ be solutions to the limiting equation \eqref{eq:FK} with the parameters $(T_0, T_1)$ and $(\hat{T}_0, \hat{T}_1)$ respectively. $\mu$ and $\hat{\mu}$ are respectively the probability measures of the density $\rho_0$ and $\hat{\rho}_0$. We have the following stability estimate:

\textup{\textbf{(\romannumeral1)}} Assume the functions $\frac{v\hat{T}_1}{\hat{T}_0} + \frac{v^2\partial_x \hat{T}_0}{2\hat{T}_0^2}$ and $\frac{v}{\sqrt{\hat{T}_0}}$ are both Lipschitz continuous with Lipschitz constants $L_1$ and $L_2$, respectively. The initial data $X_0$ and $\hat{X}_0$ are chosen such that $d_2^2(\mu_0, \hat{\mu}_0) = E|X_0 − \hat{X}_0|^2$, then 
\begin{align}
    d_2^2(\mu_T, \hat{\mu}_T) \leq& C^0_1 d_2^2(\mu_0, \hat{\mu}_0) +  C^0_2 \int_0^T \int_\mathbb{R} \left|\left(\frac{v\hat{T}_1}{\hat{T}_0} + \frac{v^2\partial_x \hat{T}_0}{2\hat{T}_0^2}\right)-\left(\frac{vT_1}{T_0}+\frac{v^2\partial_x T_0}{2T_0^2}\right)\right|^2 d\mu_t dt \nonumber \\
    &+ C^0_3 \int_0^T \int_\mathbb{R} \left|\frac{v}{\sqrt{T_0}} - \frac{v}{\sqrt{\hat{T}_0}}\right|^2 d\mu_t dt,\nonumber
\end{align}
where $C_{1}^0,\ C^0_2,\ C^0_3 >0$ are constants depending on the final time $T$ and $L_1$, $L_2$.

\textup{\textbf{(\romannumeral2)}} Let $$ J_0 = -\frac{v\partial_x \rho_0}{2T_0} - \frac{T_1 \rho_0}{ T_0},\qquad  \hat J_0 = -\frac{v\partial_x \hat\rho_0}{2\hat T_0} - \frac{\hat T_1 \hat\rho_0}{ \hat T_0},$$ and $\bar{J}_0=J_0-\hat{J}_0$. Assume that $\partial_x(\frac{\hat{T}_1}{\hat{T}_0}), \hat{T}_0, \partial_x\left(\frac{1}{\hat{T}_0}\right)\in L^\infty(\mathbb{R})$ and 
$$
\left(\frac{vT_1}{T_0}-\frac{v\hat{T}_1}{\hat{T}_0}\right)\partial_x J_0+\left(\frac{v^2}{2T_0}-\frac{v^2}{2\hat{T}_0}\right)\partial_x^2 J_0 \in L^2(\mathbb{R}),
$$ 
then
\begin{align}
    \|\bar{J}_0(x,T)\|^2_{L^2_x}
    \leq C^0_4\|\bar{J}_0(x,0)\|^2_{L^2_x} + C^0_5\int_0^T\int_\mathbb{R}\left(\left(\frac{vT_1}{T_0}-\frac{v\hat{T}_1}{\hat{T}_0}\right)\partial_x J_0 + \left(\frac{v^2}{2T_0}-\frac{v^2}{2\hat{T}_0}\right)\partial_x^2 J_0\right)^2 dx dt,\nonumber
\end{align}
where $C^0_4,\ C^0_5 >0$ are constants depending on the final time $T$ and the upper bounds of $\|\partial_x(\frac{\hat{T}_1}{\hat{T}_0})\|_{L^\infty_x}$, $\|\hat{T}_0\|_{L^\infty_x}$, $\|\partial_x\left(\frac{1}{\hat{T}_0}\right)\|_{L^\infty_x}$, but not on $\epsilon$.
\end{thm}

\subsection{Conditional Uniform stability}

\paragraph{Conditional Scale-Uniform Forward Stability for  $T_1$.}
    When $T_0$ is known, which means $\hat{T}_0=T_0$, then from \textup{\textbf{Theorem 1}}, one gets 
    \begin{equation}\label{rhoboundk}
        \| \bar\rho(x,T) \|^2_{L^2_x} \leq C_1\| \rho(x,0)-\hat{\rho}(x,0) \|^2_{L^2_x} + C_2\| J(x,0)-\hat{J}(x,0) \|^2_{L^2_x} + C_3\mathcal{E}[\hat{T}_1].
    \end{equation} 
    Here $\bar\rho(x,T)=\rho(x,T)-\hat{\rho}(x,T)$, the constant $C_3\to\infty$ as $\epsilon\to 0$, which indicates that even if $\mathcal{E}[\hat{T}_1]$ is small, the obtained $\hat T_1$ may provide a wrong time evolution of the density $\rho$ when $\epsilon$ is small.

Let us now assume uniform in time spatial Sobolev regularity on the densities $\rho_0(x,t)$ and $\hat{\rho}_0(x,t)$, so we can make use of interpolation inequalities involving $d_2$ and Sobolev norms as in \cite[Propostion 2.23 and Proposition 2.12]{CTPortoErcole} and \cite[Theorem 4.1]{CGT99}.
From \textup{\textbf{Theorem 2 (\romannumeral1)}} and this interpolation conditional to enough regularity, one can control the $L^2$ norm of the density by the weaker Wasserstein-2 distance such that
    $$
    \| \rho_0(x,T)-\hat{\rho}_0(x,T) \|^2_{L^2_x}\leq C_rd_2^{2\beta}(\mu_T, \hat{\mu}_T), 
    $$ 
where $0<\beta<1$ and $C_r$ are constants related to the assumed uniform in $[0,T]$ spatial Sobolev regularity of the solutions $\rho_0$ and $\hat\rho_0$. Then, if the solutions are regular enough, one can derive 
    \begin{align}
         \| \bar\rho_0(x,T) \|^{2/\beta}_{L^2_x} \leq&\, C^0_1 d_2^2(\mu(0), \hat{\mu}(0)) +  C^0_2v^2 \int_0^T \int_\mathbb{R} \left|\frac{\hat{T}_1}{T_0}-\frac{T_1}{T_0}\right|^2 \rho_0 dx dt\nonumber\\
        =&\, C^0_1 d_2^2(\mu(0), \hat{\mu}(0)) +  C^0_2 v^2T\mathcal{E}_0[\hat{T}_1],
        \label{rho0distance}
    \end{align}
with $C^0_2$ being independent of $\epsilon$.
    
    The combination of \eqref{rhoboundk} and \eqref{rho0distance} allows us to construct a uniform stability analysis with respect to the multiscale parameter $\epsilon$ for the recovery of $T_1$. More precisely, as illustrated in Figure \ref{Fig_uniform},  
    $$\|\bar \rho(x,T) \|^2_{L^2_x}\leq  3\| \rho(x,T)-\rho_0(x,T) \|^2_{L^2_x}+ 3\| \rho_0(x,T)-\hat{\rho}_0(x,T) \|^2_{L^2_x}+ 3\| \hat\rho(x,T)-\hat{\rho}_0(x,T) \|^2_{L^2_x}.$$
    It is known from \cite{chalub2004kinetic} that, away from the boundary and initial layers, $\| \rho(x,T)-\rho_0(x,T) \|^2_{L^2_x},  \| \hat\rho(x,T)-\hat{\rho}_0(x,T) \|^2_{L^2_x}\sim O(\epsilon)$, combined with the bound provided in \eqref{rho0distance}, one has
    \begin{equation}
    \|\bar\rho(x,T)\|^{2/\beta}_{L^2_x}=\| \rho(x,T)-\hat{\rho}(x,T) \|^{2/\beta}_{L^2_x}\leq O(\epsilon^{1/\beta})+ 3C^0_1 d_2^2(\mu(0), \hat{\mu}(0)) +  3C^0_2 v^2T\mathcal{E}_0[\hat{T}_1].\label{eq:rhodif}
    \end{equation}

    Since $\rho$ and $J$ are obtained from the real experimental data, noise is unavoidable, thus
    $\mathcal{E}_0[\hat{T}_1]$ is at least at the magnitude of noise, which can be $1-10$ percent of the magnitude of the true data in reality. On the other hand, according to the rescaling considered in real E. coli chemotaxis \cite{si2014pathway, sun2017macroscopic, calvez2015confinement}, when the environment changes slowly, i.e., the chemical gradient is small, the Keller-Segel equation is a good approximation to the kinetic model, in which the rescaled parameter $\epsilon$ is around $0.01\sim 0.1$. 
    Therefore, in the Keller-Segel regime where $\epsilon$ is small, the bound in \eqref{rhoboundk} can be large (due to the way of $C_3$'s dependence on $\epsilon$). On the other hand, when $\epsilon$ is small, $\|\bar\rho(x,T)\|_{L^2_x}$ is at similar magnitude as $\mathcal{E}_0[\hat{T}_1]$ thanks to \eqref{eq:rhodif}. 
    Therefore, one can expect, the distance between the solutions $\hat{\rho}$ and $\rho$ obtained from the estimated $\hat{T}_1$ and the true $T_1$ will be at the magnitude of the noise, regardless of the value of $\epsilon$ under this regularity assumption.

    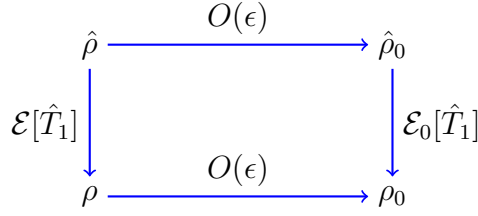
\begin{figure}[!htbp]
    \centering
    \begin{tikzpicture}[
        node distance=2cm and 4cm,
        every node/.style={font=\large, color=black}
    ]
        \node (hatrho) at (0,0) {$\hat{\rho}$};
        \node (hatrho0) at (4,0) {$\hat{\rho}_0$};
        \node (rho)  at (0,-2) {$\rho$};
        \node (rho0)  at (4,-2) {$\rho_0$};
        
        \draw[->, blue, thick] (hatrho) -- node[above] {$O(\epsilon) $} (hatrho0);
        \draw[->, blue, thick] (rho)  -- node[above] {$O(\epsilon) $} (rho0);
        
        \draw[->, blue, thick] (hatrho) -- node[left]  {$\mathcal{E}[\hat{T}_1] $} (rho);
        \draw[->, blue, thick] (hatrho0) -- node[right] {$\mathcal{E}_0[\hat{T}_1] $} (rho0);
        
    \end{tikzpicture}
    \caption{Illustration of Conditional Scale-Uniform Forward Stability for $T_1$. $\rho$ ($\hat{\rho}$) is the density obtained from the kinetic model \eqref{twoflux_1} with $(T_0,T_1)$ ($(\hat T_0,\hat T_1)$), and $\rho_0$ ($\hat{\rho}_0$) is the solution to its macroscopic limit equation \eqref{limiteqtion} with $(T_0,T_1)$ ($(\hat T_0,\hat T_1)$). From \textup{\textbf{Theorem 1}} and \textup{\textbf{Theorem 2}}, we know that $\hat{\rho}\ (\hat{\rho}_0)$ is a good approximation of $\rho\ (\rho_0)$ for all $\epsilon$ when $\mathcal{E}[\hat{T}_1]\ (\mathcal{E}_0[\hat{T}_1])$ is small under regularity assumptions.}\label{Fig_uniform}
    \end{figure}

\paragraph{Conditional Scale-Uniform Forward Stability for $T_0$.}
    When $T_1$ is known, then $\hat{T}_1=T_1$ in \textup{\textbf{Theorem 1}}, we obtain 
    \begin{equation}\label{J_bound}
        \| \bar J(x,T) \|^2_{L^2_x} \leq C_1\| \rho(x,0)-\hat{\rho}(x,0) \|^2_{L^2_x} + C_2\| J(x,0)-\hat{J}(x,0) \|^2_{L^2_x} + C_4\mathcal{G}[\hat{T}_0].
    \end{equation}

In \textup{\textbf{Theorem 2 (\romannumeral2)}}, the right-hand side of the inequality used to control $\bar J_0$ is not $\mathcal{G}[\hat{T}_0]$; rather, it involves the spatial derivative of the current $J_0$. However, in practice $J_0$ can be noisy, and $\partial_x J_0$ is even worse. Therefore, we have used $J_0$ within $\mathcal{G}[\hat{T}_0]$. To ensure that the theorem is consistent with the numerical implementation, in Lemma \ref{lemma_forJ} (Appendix~\ref{Appendix_A}) we add further assumptions on the regularity of the solutions to provide an upper bound that includes $\mathcal{G}[\hat{T}_0]$. If $\hat{T}_1=T_1$ and \textbf{Assumptions} \ref{ass:T_0}-\ref{ass:data_gram} in Appendix~\ref{Appendix_A} hold, then we can obtain that
\begin{equation*}
    \|\bar J_0(x,T)\|^2_{L^2_x} \leq C^0_4\|J_0(x,0)-\hat{J}_0(x,0)\|^2_{L^2_x} + C^0_6\mathcal{G}_0[\hat{T}_0],
\end{equation*}
with $C_6^0>0$ is a constant independent of $\epsilon$. This allows us to get uniform stability results with respect to the multiscale parameter $\epsilon$ for the recovery of $T_0$. Similarly as for the recovery of $T_1$, as illustrated in Figure \ref{Fig:uniform_J},
\begin{equation}\label{eq:difJ}
\| \bar J(x,T)\|^2_{L^2_x}\leq O(\epsilon)+ 3C^0_4\|J_0(x,0)-\hat{J}_0(x,0)\|^2_{L^2_x} + 3C^0_6\mathcal{G}_0[\hat{T}_0].
\end{equation}
Therefore, when $\epsilon$ is small, though the bound of $\|\bar J\|_{L_x^2}$in \eqref{J_bound} can be large (due to the $C_4$’s dependence on $\epsilon$),  \eqref{eq:difJ} provides a good bound when $\epsilon$ is small under the assumed regularity of solutions in \textbf{Assumptions} \ref{ass:T_0}-\ref{ass:data_gram} in Appendix~\ref{Appendix_A}. As a consequence, the distance between the solutions $\hat{J}$ and $J$ obtained respectively from the estimated $\hat{T}_0$ and the true $T_0$ should be at the magnitude of the possible noise similarly to the case of the $T_1$ recovery.

\begin{figure}[!htbp]
    \centering
    \begin{tikzpicture}[
        node distance=2cm and 4cm,
        every node/.style={font=\large, color=black}
    ]
        \node (hatJ) at (0,0) {$\hat{J}$};
        \node (hatJ0) at (4,0) {$\hat{J}_0$};
        \node (J)  at (0,-2) {$J$};
        \node (J0)  at (4,-2) {$J_0$};
        
        \draw[->, blue, thick] (hatJ) -- node[above] {$O(\epsilon) $} (hatJ0);
        \draw[->, blue, thick] (J)  -- node[above] {$O(\epsilon) $} (J0);
        
        \draw[->, blue, thick] (hatJ) -- node[left]  {$\mathcal{G}[\hat{T}_0] $} (J);
        \draw[->, blue, thick] (hatJ0) -- node[right] {$\mathcal{G}_0[\hat{T}_0] $} (J0);
    \end{tikzpicture}
    \caption{Illustration of conditional Scale-Uniform Forward Stability for $T_0$. $J\ (\hat{J})$ is flux density obtained from the kinetic model \eqref{twoflux_1} with $(T_0,T_1)$ ($(\hat T_0,\hat T_1)$), and $J_0$ ($\hat{J}_0$) is the solution to its macroscopic limit equation \eqref{limiteqtion} with $(T_0,T_1)$ ($(\hat T_0,\hat T_1)$). From \textup{\textbf{Theorem 1}} and \textup{\textbf{Theorem 2}}, we know that $\hat{J}\ (\hat{J}_0)$ is a good approximation of $J\ (J_0)$ for all $\epsilon$ when $\mathcal{G}[\hat{T}_0]\ (\mathcal{G}_0[\hat{T}_0])$ is small under regularity assumptions. }\label{Fig:uniform_J}
    \end{figure}
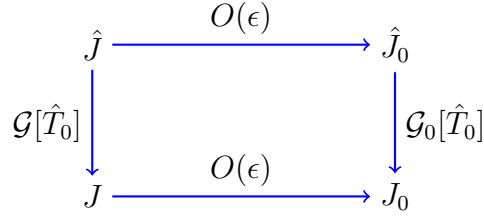

\section{Proof of the main theorems}
\subsection{Proof of \textbf{Theorem 1}}
\begin{proof}
Let $X^+(t) = x-\frac{v}{\epsilon}t$, $X^-(t) = x+\frac{v}{\epsilon}t$.
Using the characteristic method, the solutions to Eq. \eqref{twoflux} can be written as
\begin{align}\label{solutionbychar}
        f^+(x,T)=& f^+(X^+(T),0) + \int_0^T \left[ \frac{T_0(X^+(T-t))-\epsilon T_1(X^+(T-t))}{\epsilon^2}f^-(X^+(T-t),t)\right. \nonumber\\
        &-\left. \frac{T_0(X^+(T-t))+\epsilon T_1(X^+(T-t))}{\epsilon^2}f^+(X^+(T-t),t)\right]dt,\nonumber\\
        f^-(x,T)=&f^-(X^-(T),0) + \int_0^T \left[ \frac{T_0(X^-(T-t))+\epsilon T_1(X^-(T-t))}{\epsilon^2}f^+(X^-(T-t),t)\right. \nonumber\\
        &-\left. \frac{T_0(X^-(T-t))-\epsilon T_1(X^-(T-t))}{\epsilon^2}f^-(X^-(T-t),t)\right]dt.
\end{align}

Since $\rho(x,t)=f^+(x,t)+f^-(x,t)$, $J(x,t)=\frac{f^+(x,t)-f^-(x,t)}{\epsilon}$, and $f^+=\frac{\rho+\epsilon J}{2}$, $f^-=\frac{\rho-\epsilon J}{2}$, then
\begin{align}
        \rho(x,T)=&\frac{\rho+\epsilon J}{2} (X^+(T),0) + \frac{\rho-\epsilon J}{2} (X^-(T),0)  \nonumber\\
        &+ \int_0^T \left[ \frac{T_0-\epsilon T_1}{\epsilon^2}\frac{\rho-\epsilon J}{2}\right](X^+(T-t),t)dt-\int_0^T\left[ \frac{T_0+\epsilon T_1}{\epsilon^2}\frac{\rho+\epsilon J}{2}\right](X^+(T-t),t)dt\nonumber\\
        &+ \int_0^T\left[ \frac{T_0+\epsilon T_1}{\epsilon^2}\frac{\rho+\epsilon J}{2}\right](X^-(T-t),t)dt -\int_0^T\left[ \frac{T_0-\epsilon T_1}{\epsilon^2}\frac{\rho-\epsilon J}{2}\right](X^-(T-t),t)dt \nonumber\\
        =&\frac{\rho+\epsilon J}{2} (X^+(T),0) + \frac{\rho-\epsilon J}{2} (X^-(T),0) + \int_0^T \left[ -\frac{T_1}{\epsilon} \rho-\frac{T_0}{\epsilon}J\right](X^+(T-t),t)dt \nonumber\\
        &+\int_0^T \left[ \frac{T_1}{\epsilon} \rho+\frac{T_0}{\epsilon}J\right](X^-(T-t),t)dt,\nonumber
\end{align}
\begin{align}
        J(x,T)=&\frac{\rho+\epsilon J}{2\epsilon} (X^+(T),0) - \frac{\rho-\epsilon J}{2\epsilon} (X^-(T),0)  \nonumber\\
        &+ \frac{1}{\epsilon} \int_0^T \left[ \frac{T_0-\epsilon T_1}{\epsilon^2}\frac{\rho-\epsilon J}{2}\right](X^+(T-t),t)dt- \frac{1}{\epsilon} \int_0^T\left[ \frac{T_0+\epsilon T_1}{\epsilon^2}\frac{\rho+\epsilon J}{2}\right](X^+(T-t),t)dt\nonumber\\
        &+ \frac{1}{\epsilon} \int_0^T\left[ \frac{T_0-\epsilon T_1}{\epsilon^2}\frac{\rho-\epsilon J}{2}\right](X^-(T-t),t)dt - \frac{1}{\epsilon} \int_0^T\left[ \frac{T_0+\epsilon T_1}{\epsilon^2}\frac{\rho+\epsilon J}{2}\right](X^-(T-t),t)dt\nonumber\\
        =&\frac{\rho+\epsilon J}{2\epsilon} (X^+(T),0) - \frac{\rho-\epsilon J}{2\epsilon} (X^-(T),0) + \frac{1}{\epsilon}\int_0^T \left[ -\frac{T_1}{\epsilon} \rho-\frac{T_0}{\epsilon}J\right](X^+(T-t),t)dt \nonumber\\
        &+ \frac{1}{\epsilon} \int_0^T \left[ -\frac{T_1}{\epsilon} \rho-\frac{T_0}{\epsilon}J\right](X^-(T-t),t)dt. \nonumber
\end{align}

Since $\bar{\rho}=\rho-\hat{\rho}$ and $\bar{J}=J-\hat{J}$, we can then get
\begin{align}
        \bar{\rho}(x,T) =& \frac{\bar{\rho}+\epsilon \bar{J}}{2}(X^+(T),0)+\frac{\bar{\rho}-\epsilon \bar{J}}{2}(X^-(T),0)\nonumber\\
        &+\int_0^T \left[ -\left( \frac{T_1}{\epsilon}\rho - \frac{\hat{T}_1}{\epsilon}\hat{\rho} \right) - \left( \frac{T_0}{\epsilon}J - \frac{\hat{T}_0}{\epsilon}\hat{J} \right) \right] (X^+(T-t),t) dt\nonumber\\
        &+\int_0^T \left[ \left( \frac{T_1}{\epsilon}\rho - \frac{\hat{T}_1}{\epsilon}\hat{\rho} \right) + \left( \frac{T_0}{\epsilon}J - \frac{\hat{T}_0}{\epsilon}\hat{J} \right) \right] (X^-(T-t),t) dt,\nonumber
\end{align}
\begin{align}
        \bar{J}(x,T) =& \frac{\bar{\rho}+\epsilon\bar{J}}{2\epsilon}(X^+(T),0)-\frac{\bar{\rho}-\epsilon\bar{J}}{2\epsilon}(X^-(T),0)\nonumber\\
        &+\frac{1}{\epsilon}\int_0^T \left[ -\left( \frac{T_1}{\epsilon}\rho - \frac{\hat{T}_1}{\epsilon}\hat{\rho} \right) - \left( \frac{T_0}{\epsilon}J - \frac{\hat{T}_0}{\epsilon}\hat{J} \right) \right] (X^+(T-t),t) dt\nonumber\\
        &+\frac{1}{\epsilon}\int_0^T \left[ -\left( \frac{T_1}{\epsilon}\rho - \frac{\hat{T}_1}{\epsilon}\hat{\rho} \right) - \left( \frac{T_0}{\epsilon}J - \frac{\hat{T}_0}{\epsilon}\hat{J} \right) \right] (X^-(T-t),t) dt.\nonumber
\end{align}
 By considering the $L^2_x$ norm of $\bar{\rho}$ and $\bar{J}$, we obtain
\begin{align}
    \|\bar{\rho}(x,T)\|_{L^2_x} \leq& \|\bar{\rho}(x,0)\|_{L^2_x} + \epsilon \|\bar{J}(x,0)\|_{L^2_x}+ 2 \left\lVert\left( \frac{\hat{T}_1}{\epsilon}\hat{\rho} - \frac{T_1}{\epsilon}\rho \right) + \left( \frac{\hat{T}_0}{\epsilon}\hat{J} - \frac{T_0}{\epsilon}J \right)\right\rVert_{L^2_x,L^1_t} \nonumber\\
    \leq& \|\bar{\rho}(x,0)\|_{L^2_x} + \epsilon \|\bar{J}(x,0)\|_{L^2_x}\nonumber\\
    &+ 2 \left\lVert\left( \frac{\hat{T}_1}{\epsilon}\hat{\rho} - \frac{\hat{T}_1}{\epsilon}\rho + \frac{\hat{T}_1}{\epsilon}\rho- \frac{T_1}{\epsilon}\rho \right)\right\rVert_{L^2_x,L^1_t} + 2 \left\lVert\left( \frac{\hat{T}_0}{\epsilon}\hat{J} -\frac{\hat{T}_0}{\epsilon}J+\frac{\hat{T}_0}{\epsilon}J-\frac{T_0}{\epsilon}J \right)\right\rVert_{L^2_x,L^1_t} \nonumber\\
    \leq& \|\bar{\rho}(x,0)\|_{L^2_x} + \epsilon \|\bar{J}(x,0)\|_{L^2_x}+ 2\left\lVert \frac{\hat{T}_1}{\epsilon}\right\rVert_{L^\infty_x} \|\bar{\rho}\|_{L^2_x,L^1_t} + 2\left\lVert \frac{\hat{T}_0}{\epsilon}\right\rVert_{L^\infty_x} \|\bar{J}\|_{L^2_x,L^1_t}\nonumber\\
    &+ \frac{2}{\epsilon} \|(\hat{T}_1 - T_1)\rho\|_{L^2_x,L^1_t} + \frac{2}{\epsilon} \|(\hat{T}_0 - T_0)J\|_{L^2_x,L^1_t},  \nonumber
\end{align}
\begingroup
\allowdisplaybreaks[0]
\begin{align}
    \|\bar{J}(x,T)\|_{L^2_x} 
    \leq& \frac{1}{\epsilon} \|\bar{\rho}(x,0)\|_{L^2_x} + \|\bar{J}(x,0)\|_{L^2_x} + \frac{2}{\epsilon} \left\lVert \frac{\hat{T}_1}{\epsilon}\right\rVert_{L^\infty_x} \|\bar{\rho}\|_{L^2_x,L^1_t} + \frac{2}{\epsilon} \left\lVert \frac{\hat{T}_0}{\epsilon}\right\rVert_{L^\infty_x} \|\bar{J}\|_{L^2_x,L^1_t}\nonumber\\
    &+ \frac{2}{\epsilon^2}  \|(\hat{T}_1 - T_1)\rho\|_{L^2_x,L^1_t} + \frac{2}{\epsilon^2} \|(\hat{T}_0 - T_0)J\|_{L^2_x,L^1_t}.  \nonumber
\end{align}
\endgroup
Then, denote
\begin{equation}
    C_{\epsilon} = \max\left\{\left(\frac{2}{\epsilon}+\frac{2}{\epsilon^2}\right)\left\lVert \hat{T}_1\right\rVert_{L^\infty_x} , \left(\frac{2}{\epsilon}+\frac{2}{\epsilon^2}\right)\left\lVert \hat{T}_0\right\rVert_{L^\infty_x}\right\},
\end{equation}
and by Gr$\ddot{\text{o}}$nwall's inequality and H$\ddot{\text{o}}$lder's inequality,
\begin{align*}
        \|\bar{\rho}(x,T)\|_{L^2_x}+\|\bar{J}(x,T)\|_{L^2_x} 
        \leq& e^{T C_{\epsilon}} \bigg( \left(1+\frac{1}{\epsilon}\right)\|\bar{\rho}(x,0)\|_{L^2_x} + (\epsilon+1) \|\bar{J}(x,0)\|_{L^2_x} \\
        &+ \frac{2 (1+\epsilon)}{\epsilon^2} \|(\hat{T}_1 - T_1)\rho\|_{L^2_x,L^1_t} + \frac{2 (1+\epsilon)}{\epsilon^2} \|(\hat{T}_0 - T_0)J\|_{L^2_x,L^1_t} \bigg)\\
        \leq& e^{T C_{\epsilon}}\Bigg( \left(1+\frac{1}{\epsilon}\right)\|\bar{\rho}(x,0)\|_{L^2_x} + (\epsilon+1) \|\bar{J}(x,0)\|_{L^2_x} \\
        &+ \frac{2\sqrt{T} (1+\epsilon)}{\epsilon^2} \|(\hat{T}_1 - T_1)\rho\|_{L^2_x,L^2_t} + \frac{2\sqrt{T} (1+\epsilon)}{\epsilon^2} \|(\hat{T}_0 - T_0)J\|_{L^2_x,L^2_t} \bigg).\nonumber
\end{align*}
Therefore, squaring both sides and applying the Cauchy–Schwarz inequality, we obtain
\begin{align}
    \|\bar{\rho}(x,T)\|^2_{L^2_x}+\|\bar{J}(x,T)\|^2_{L^2_x} \leq& \big( \|\bar{\rho}(x,T)\|_{L^2_x}+\|\bar{J}(x,T)\|_{L^2_x}\big)^2\nonumber\\
        \leq& 4e^{2TC_\epsilon}\Bigg( \left(1+\frac{1}{\epsilon}\right)^2\|\bar{\rho}(x,0)\|^2_{L^2_x} + (\epsilon+1)^2 \|\bar{J}(x,0)\|^2_{L^2_x} \\
        &+ \frac{4T (1+\epsilon)^2}{\epsilon^4} \|(\hat{T}_1 - T_1)\rho\|^2_{L^2_x,L^2_t} + \frac{4T (1+\epsilon)^2}{\epsilon^4} \|(\hat{T}_0 - T_0)J\|^2_{L^2_x,L^2_t} \bigg).\nonumber
\end{align}
Then from
\begin{align}
    \|(\hat{T}_1 - T_1)\rho\|^2_{L^2_x,L^2_t} &= \int_0^{T}\int_\mathbb{R} |\hat{T}_1 - T_1|^2 \rho^2 dxdt\leq T\|\rho T_0^2\| _{L^\infty_x} \mathcal{E}[\hat{T}_1], \nonumber\\
    \|(\hat{T}_0 - T_0)J\|^2_{L^2_x,L^2_t} &=\int_0^{T}\int_\mathbb{R} |\hat{T}_0 - T_0|^2 J^2 dxdt = T \mathcal{G}[\hat{T}_0],\nonumber 
\end{align}
we can conclude that
\begin{align}
    \|\bar{\rho}(x,T)\|^2_{L^2_x}+\|\bar{J}(x,T)\|^2_{L^2_x} 
        \leq&   C_1 \|\bar{\rho}(x,0)\|^2_{L^2_x} + C_2 \|\bar{J}(x,0)\|^2_{L^2_x} + C_3\mathcal{E}[\hat{T}_1] + C_4\mathcal{G}[\hat{T}_0].\nonumber
\end{align}
\end{proof}

\subsection{Proof of \textbf{Theorem 2}}
\begin{proof}
 
\textup{\textbf{(\romannumeral1)}}:    
The proof follows the strategy used in \cite{carrillo2025sparse}. By the definition of the Wasserstein-2 distance we have
\begin{equation*}
    d_2^2(\mu_T, \hat{\mu}_T) \leq \mathbb{E} \sup_{t\in[0,T]} |X_t-\hat{X}_t|^2.
\end{equation*}
From Eq. \eqref{eq:dX_t} and \eqref{eq:dhatX_t},
\begin{align}\label{difference_of_SDE}
    d(X_t-\hat{X}_t) =& \left[\left(-\frac{vT_1}{T_0}-\frac{v^2\partial_x T_0}{2T_0^2}\right)(X_s) - \left(-\frac{v\hat{T}_1}{\hat{T}_0} - \frac{v^2\partial_x \hat{T}_0}{2\hat{T}_0^2}\right)(\hat{X}_s)\right] dt \nonumber\\
    &+ \left[ \frac{v}{\sqrt{T_0}}(X_s) - \frac{v}{\sqrt{\hat{T}_0}}(\hat{X}_s)\right] dB_t.
\end{align}
Integrating both sides of Eq. \eqref{difference_of_SDE} with respect to $t$, squaring the result, and applying the Cauchy–Schwarz inequality, we obtain the following estimate:
\begin{align*}
    \mathbb{E} \sup_{t\in[0,T]} |X_t-\hat{X}_t|^2 \leq& 3\mathbb{E} \sup_{t\in[0,T]}\Bigg( |X_0-\hat{X}_0|^2 \\
    &+\left| \int_0^t \left[ \left(-\frac{vT_1}{T_0}-\frac{v^2\partial_x T_0}{2T_0^2}\right)(X_s) - \left(-\frac{v\hat{T}_1}{\hat{T}_0} - \frac{v^2\partial_x \hat{T}_0}{2\hat{T}_0^2}\right)(\hat{X}_s) \right]ds \right|^2\\
    &+ \left| \int_0^t \left[ \frac{v}{\sqrt{T_0}}(X_s) - \frac{v}{\sqrt{\hat{T}_0}}(\hat{X}_s)\right] dB_s \right|^2 \Bigg)\\
    :=& \ \uppercase\expandafter{\romannumeral1}\ + \ \uppercase\expandafter{\romannumeral2}\ + \ \uppercase\expandafter{\romannumeral3}.
\end{align*}

From H$\ddot{\text{o}}$lder's inequality, we then have
\begin{align*}
    \uppercase\expandafter{\romannumeral2} =& 3\mathbb{E} \sup_{t\in[0,T]} \left| \int_0^t \left[ \left(-\frac{vT_1}{T_0}-\frac{v^2\partial_x T_0}{2T_0^2}\right)(X_s) - \left(-\frac{v\hat{T}_1}{\hat{T}_0} - \frac{v^2\partial_x \hat{T}_0}{2\hat{T}_0^2}\right)(\hat{X}_s) \right]ds \right|^2 \\
    \leq& 3\mathbb{E} \sup_{t\in[0,T]} t\int_0^t \left| \left(-\frac{vT_1}{T_0}-\frac{v^2\partial_x T_0}{2T_0^2}\right)(X_s) - \left(-\frac{v\hat{T}_1}{\hat{T}_0} - \frac{v^2\partial_x \hat{T}_0}{2\hat{T}_0^2}\right)(\hat{X}_s)\right|^2 ds \\
    \leq&6T \mathbb{E} \sup_{t\in[0,T]}\int_0^t \left| \left(-\frac{vT_1}{T_0}-\frac{v^2\partial_x T_0}{2T_0^2}\right)(X_s) - \left(-\frac{v\hat{T}_1}{\hat{T}_0} - \frac{v^2\partial_x \hat{T}_0}{2\hat{T}_0^2}\right)(X_s)\right|^2 ds\\
    &+ 6T \mathbb{E} \sup_{t\in[0,T]}\int_0^t \left| \left(-\frac{v\hat{T}_1}{\hat{T}_0} - \frac{v^2\partial_x \hat{T}_0}{2\hat{T}_0^2}\right)(X_s) - \left(-\frac{v\hat{T}_1}{\hat{T}_0} - \frac{v^2\partial_x \hat{T}_0}{2\hat{T}_0^2}\right)(\hat{X}_s)\right|^2 ds\\
    \leq& 6T \int_0^T \int_\mathbb{R} \left|\left(-\frac{vT_1}{T_0}-\frac{v^2\partial_x T_0}{2T_0^2}\right) - \left(-\frac{v\hat{T}_1}{\hat{T}_0} - \frac{v^2\partial_x \hat{T}_0}{2\hat{T}_0^2}\right)\right|^2 d\mu_t dt\\
    &+6TL_1^2 \int_0^T \mathbb{E}\sup_{s\in[0,t]} |X_s-\hat{X}_s|^2dt.
\end{align*}
Since $\int_0^t \left[ \frac{v}{\sqrt{T_0}}(X_s) - \frac{v}{\sqrt{\hat{T}_0}}(\hat{X}_s)\right] dB_s$ is a continuous martingale, we can use Burkholder-Davis-Gundy inequality, as has been done in \cite[ Section A.2]{carrillo2025sparse}, and get
\begin{align*}
    \uppercase\expandafter{\romannumeral3} =& 3\mathbb{E} \sup_{t\in[0,T]}\left| \int_0^t \left[ \frac{v}{\sqrt{T_0}}(X_s) - \frac{v}{\sqrt{\hat{T}_0}}(\hat{X}_s)\right] dB_s \right|^2
    \leq 3CT\  \mathbb{E} \int_0^T \left| \frac{v}{\sqrt{T_0}}(X_t) - \frac{v}{\sqrt{\hat{T}_0}}(\hat{X}_t) \right|^2dt\\
    \leq& 6CT\ \mathbb{E} \int_0^T \left| \frac{v}{\sqrt{T_0}}(X_t) - \frac{v}{\sqrt{\hat{T}_0}}(X_t) \right|^2dt + 6CT\ \mathbb{E} \int_0^T \left| \frac{v}{\sqrt{\hat{T}_0}}(X_t) - \frac{v}{\sqrt{\hat{T}_0}}(\hat{X}_t) \right|^2dt\\
    \leq& 6CT \int_0^T \int_\mathbb{R} \left|\frac{v}{\sqrt{T_0}} - \frac{v}{\sqrt{\hat{T}_0}}\right|^2 d\mu_t dt + 6CTL_2^2 \int_0^T \sup_{s\in[0,t]} \mathbb{E} |X_s - \hat{X}_s|^2 dt.
\end{align*}
Then,
\begin{align*}
    \mathbb{E} \sup_{t\in[0,T]} |X_t-\hat{X}_t|^2 
    \leq& 3\mathbb{E} |X_0-\hat{X}_0|^2 +  6T \int_0^T \int_\mathbb{R} \left|\left(\frac{v\hat{T}_1}{\hat{T}_0} + \frac{v^2\partial_x \hat{T}_0}{2\hat{T}_0^2}\right)-\left(\frac{vT_1}{T_0}+\frac{v^2\partial_x T_0}{2T_0^2}\right)\right|^2 d\mu_t dt\\
    &+ 6TL_1^2 \int_0^T \mathbb{E}\sup_{s\in[0,t]} |X_s-\hat{X}_s|^2dt + 6CT \int_0^T \int_\mathbb{R} \left|\frac{v}{\sqrt{T_0}} - \frac{v}{\sqrt{\hat{T}_0}}\right|^2 d\mu_t dt\\
    &+ 6CTL_2^2 \int_0^T \sup_{s\in[0,t]} \mathbb{E} |X_s - \hat{X}_s|^2 dt.
\end{align*}
Since $\sup_{s\in[0,t]} \mathbb{E} |X_s - \hat{X}_s|^2 \leq \mathbb{E} \sup_{s\in[0,t]}  |X_s - \hat{X}_s|^2$, by Gr$\ddot{\text{o}}$nwall's inequality
\begin{align*}
    \mathbb{E} \sup_{t\in[0,T]} |X_s-\hat{X}_s|^2 \leq& \Bigg(3\mathbb{E} |X_0-\hat{X}_0|^2 +  6T \int_0^T \int_\mathbb{R} \left|\left(\frac{v\hat{T}_1}{\hat{T}_0} + \frac{v^2\partial_x \hat{T}_0}{2\hat{T}_0^2}\right)-\left(\frac{vT_1}{T_0}+\frac{v^2\partial_x T_0}{2T_0^2}\right)\right|^2 d\mu_t dt\\
    &+ 6CT \int_0^T \int_\mathbb{R} \left|\frac{v}{\sqrt{T_0}} - \frac{v}{\sqrt{\hat{T}_0}}\right|^2 d\mu_t dt\Bigg) \exp{(6T^2 L_1^2 + 6CT^2L_2^2)}.
\end{align*}
Since $d_2^2(\mu(0), \hat{\mu}(0)) = E|X_0 − \hat{X}_0|^2$, we can get the conclusion:
\begin{align}
    d_2^2(\mu(T), \hat{\mu}(T)) \leq& C^0_1 d_2^2(\mu(0), \hat{\mu}(0)) +  C^0_2 \int_0^T \int_\mathbb{R} \left|\left(\frac{v\hat{T}_1}{\hat{T}_0} + \frac{v^2\partial_x \hat{T}_0}{2\hat{T}_0^2}\right)-\left(\frac{vT_1}{T_0}+\frac{v^2\partial_x T_0}{2T_0^2}\right)\right|^2 d\mu_t dt \nonumber \\
    &+ C^0_3 \int_0^T \int_\mathbb{R} \left|\frac{v}{\sqrt{T_0}} - \frac{v}{\sqrt{\hat{T}_0}}\right|^2 d\mu_t dt.\nonumber
\end{align}

\noindent \textup{\textbf{(\romannumeral2)}}: Taking derivative with respect to $t$ on both sides of \eqref{limiteqtion_2}, we can get
\begin{equation*}
    v\partial_t\partial_x \rho_0 = -2T_0(x)\partial_t J_0 - 2T_1(x)\partial_t\rho_0.
\end{equation*}
From Eq. \eqref{twoflux_1_1}, $\partial_t \rho_0 = -v\partial_x J_0$, we have
\begin{equation}\label{equation_J}
    -v^2\partial_x^2 J_0 = -2T_0(x)\partial_t J_0+2vT_1(x)\partial_x J_0.
\end{equation}
Similarly, for $\hat{J}_0$,
\begin{equation}\label{equation_Jhat}
    -v^2\partial_x^2 \hat{J}_0 = -2\hat{T}_0(x)\partial_t \hat{J}_0+2v\hat{T}_1(x)\partial_x \hat{J}_0.
\end{equation}
Taking difference between \eqref{equation_J} and \eqref{equation_Jhat}, we get
\begin{align}\label{equation _of_difference_of_J}
    \partial_t \bar{J}_0 =& \frac{vT_1}{T_0}\partial_x J_0-\frac{v\hat{T}_1}{\hat{T}_0}\partial_x \hat{J}_0+\frac{v^2\partial_x^2 J_0}{2T_0}-\frac{v^2\partial_x^2 \hat{J}_0}{2\hat{T}_0} \nonumber\\
    =& \frac{v\hat{T}_1}{\hat{T}_0}\partial_x \bar{J}_0 + \frac{v^2}{2\hat{T}_0}\partial_x^2 \bar{J}_0 + \left(\frac{vT_1}{T_0}-\frac{v\hat{T}_1}{\hat{T}_0}\right)\partial_x J_0 + \left(\frac{v^2}{2T_0}-\frac{v^2}{2\hat{T}_0}\right)\partial_x^2 J_0.
\end{align}
Multiply both sides of Eq. \eqref{equation _of_difference_of_J} by $\bar{J}_0$ and integrate with respect to $x$ over $\mathbb{R}$, we have
\begin{align}
    \int_\mathbb{R} \bar{J}_0\partial_t \bar{J}_0 dx =& \int_\mathbb{R} \frac{v\hat{T}_1}{\hat{T}_0}\bar{J}_0\partial_x \bar{J}_0\ dx + \int_\mathbb{R} \frac{v^2}{2\hat{T}_0}\bar{J}_0\partial_x^2 \bar{J}_0\ dx \nonumber\\
    &+ \int_\mathbb{R}\left[ \bar{J}_0\left(\frac{vT_1}{T_0}-\frac{v\hat{T}_1}{\hat{T}_0}\right)\partial_x J_0+ \bar{J}_0\left(\frac{v^2}{2T_0}-\frac{v^2}{2\hat{T}_0}\right)\partial_x^2 J_0\ \right] dx \nonumber\\
    :=& A_{\uppercase\expandafter{\romannumeral1}}\ + A_{\uppercase\expandafter{\romannumeral2}} + A_{\uppercase\expandafter{\romannumeral3}}.\nonumber
\end{align}

In the subsequent part, we control each of them.
\begin{itemize}
    \item By using integration by parts and the boundary condition for $\bar{J}_0$, we get
    \begin{align}
        |A_{\uppercase\expandafter{\romannumeral1}}| =& \left|\int_\mathbb{R} \frac{v\hat{T}_1}{\hat{T}_0}\bar{J}_0\partial_x \bar{J}_0\ dx \right|= \left|\frac{1}{2}\int_\mathbb{R} \frac{v\hat{T}_1}{\hat{T}_0}\partial_x\bar{J}_0^2\ dx\right|= \left|\frac{1}{2} \frac{v\hat{T}_1}{\hat{T}_0}\bar{J}_0^2 \Bigg|_{-\infty}^{+\infty}-\frac{1}{2}\int_\mathbb{R} \bar{J}_0^2 \partial_x(\frac{v\hat{T}_1}{\hat{T}_0})\ dx\right|\nonumber\\
        \leq& \frac{1}{2}\left\lVert \partial_x(\frac{v\hat{T}_1}{\hat{T}_0})\right\rVert_{L^{\infty}_x} \|\bar{J}_0\|_{L^2_x}^2. \nonumber
    \end{align}
    \item Using the integration by parts and boundary condition for $\bar{J}_0$,
    \begin{align}
        A_{\uppercase\expandafter{\romannumeral2}} =& \int_\mathbb{R} \frac{v^2}{2\hat{T}_0}\bar{J}_0\partial_x^2 \bar{J}_0\ dx = -\frac{1}{2}\int_\mathbb{R} \partial_x\left(\frac{v^2}{\hat{T}_0}\right)\bar{J}_0\partial_x \bar{J}_0\ dx -\frac{1}{2}\int_\mathbb{R} \frac{v^2}{\hat{T}_0}(\partial_x \bar{J}_0)^2dx. \nonumber
    \end{align}
    Then, from H$\ddot{\text{o}}$lder's inequality and Young's inequality
    \begin{align}
        A_{\uppercase\expandafter{\romannumeral2}} \leq& \frac{1}{2} \left\lVert \partial_x\left(\frac{v^2}{\hat{T}_0}\right) \right\rVert_{L^{\infty}_x} \|\bar{J}_0\|_{L^2_x} \|\partial_x \bar{J}_0\|_{L^2_x}-\frac{v^2}{2\|\hat{T}_0\|_{L^\infty}} \|\partial_x \bar{J}_0\|^2_{L^2_x} \nonumber\\
        \leq& \frac{1}{4} \left\lVert \partial_x\left(\frac{v^2}{\hat{T}_0}\right) \right\rVert_{L^{\infty}_x} \left(\frac{1}{c}\|\bar{J}_0\|^2_{L^2_x} + c\|\partial_x \bar{J}_0\|^2_{L^2_x}\right)-\frac{v^2}{2\|\hat{T}_0\|_{L^\infty}} \|\partial_x \bar{J}_0\|^2_{L^2_x}.\nonumber
    \end{align}
    If we take $c = (\|\hat{T}_0\|_{L^\infty} \left\lVert \partial_x\left(\frac{1}{\hat{T}_0}\right) \right\rVert_{L^{\infty}_x})^{-1}$, we can get
    \begin{equation*}
        A_{\uppercase\expandafter{\romannumeral2}} \leq \frac{v^2}{4c} \left\lVert \partial_x\left(\frac{1}{\hat{T}_0}\right) \right\rVert_{L^{\infty}_x}  \|\bar{J}_0\|_{L^2_x}^2 - \frac{v^2}{4\|\hat{T}_0\|_{L^\infty}} \|\partial_x \bar{J}_0\|^2_{L^2_x} \leq \frac{v^2}{4c} \left\lVert \partial_x\left(\frac{1}{\hat{T}_0}\right) \right\rVert_{L^{\infty}_x}  \|\bar{J}_0\|_{L^2_x}^2.
    \end{equation*}
    \item By H$\ddot{\text{o}}$lder's inequality, 
    \begin{align}
        |A_{\uppercase\expandafter{\romannumeral3}}| =& \left|\int_\mathbb{R} \bar{J}_0\left(\frac{vT_1}{T_0}-\frac{v\hat{T}_1}{\hat{T}_0}\right)\partial_x J_0\ +\bar{J}_0\left(\frac{v^2}{2T_0}-\frac{v^2}{2\hat{T}_0}\right)\partial_x^2 J_0\ dx\right|\nonumber\\
        \leq& \|\bar{J}_0\|_{L^2_x} \left\lVert\left(\frac{vT_1}{T_0}-\frac{v\hat{T}_1}{\hat{T}_0}\right)\partial_x J_0+\left(\frac{v^2}{2T_0}-\frac{v^2}{2\hat{T}_0}\right)\partial_x^2 J_0\right\rVert_{L^2_x}.\nonumber
    \end{align}
\end{itemize}

Therefore, since $\partial_x(\frac{\hat{T}_1}{\hat{T}_0}), \hat{T}_0, \partial_x\left(\frac{1}{\hat{T}_0}\right)\in L^\infty(\mathbb{R})$ and $\left(\frac{vT_1}{T_0}-\frac{v\hat{T}_1}{\hat{T}_0}\right)\partial_x J_0+\left(\frac{v^2}{2T_0}-\frac{v^2}{2\hat{T}_0}\right)\partial_x^2 J_0 \in L^2(\mathbb{R})$, we can conclude that
\begin{equation*}
    \frac{1}{2}\frac{d}{dt}\|\bar{J}_0\|_{L^2_x}^2(t) \leq C \|\bar{J}_0\|_{L^2_x}^2(t) + D(t) \|\bar{J}_0\|_{L^2_x}(t),
\end{equation*}
which means
\begin{equation*}
    \frac{d\|\bar{J}_0\|_{L^2_x} (t)}{dt}\leq C \|\bar{J}_0\|_{L^2_x}(t) + D(t),
\end{equation*}
where
\begin{align}
    C =& \frac{1}{2}\left\lVert \partial_x(\frac{v\hat{T}_1}{\hat{T}_0})\right\rVert_{L^{\infty}_x} + \frac{v^2\|\hat{T}_0\|_{L^\infty}}{4}\left\lVert \partial_x\left(\frac{1}{\hat{T}_0}\right) \right\rVert^2_{L^{\infty}_x},\\
    D(t) =& \left\lVert\left(\frac{vT_1}{T_0}-\frac{v\hat{T}_1}{\hat{T}_0}\right)\partial_x J_0 + \left(\frac{v^2}{2T_0}-\frac{v^2}{2\hat{T}_0}\right)\partial_x^2 J_0\right\rVert_{L^2_x}.\nonumber
\end{align}
By Gr$\ddot{\text{o}}$nwall's inequality and H$\ddot{\text{o}}$lder's inequality, we have
\begin{align}\label{difference_of_J}
    \|\bar{J}_0(x,T)\|_{L^2_x}\leq& \exp{(C T)} \left(\|\bar{J}_0(x,0)\|_{L^2_x} + \int_0^T D(t)dt \right)
    = \exp{(C T)}\left(\|\bar{J}_0(x,0)\|_{L^2_x} + \left\lVert D(t)\right\rVert_{L^1_t}\right)\nonumber\\
    \leq&\exp{(C T)} \left( \|\bar{J}_0(x,0)\|_{L^2_x} + \sqrt{T}\left\lVert D(t)\right\rVert_{L^2_t} \right).
\end{align}
In summary, squaring both sides and applying the Cauchy–Schwarz inequality, we can get the conclusion:
\begin{align}
    \|\bar J_0(x,T)\|^2_{L^2_x}
    \leq C^0_4\|\bar J_0(x,0)\|^2_{L^2_x}+ C^0_5\int_0^T\int_\mathbb{R}\left(\left(\frac{vT_1}{T_0}-\frac{v\hat{T}_1}{\hat{T}_0}\right)\partial_x J_0 + \left(\frac{v^2}{2T_0}-\frac{v^2}{2\hat{T}_0}\right)\partial_x^2 J_0\right)^2 dx dt.\nonumber
\end{align}

\end{proof}

\section{Numerical results}
\label{sec:Numerical results}
\subsection{Optimization problem formulation}

As in \cite{carrillo2025sparse}, we assume that $T_0$ and $T_1$ are composed of a set of basis functions. Meanwhile, considering the ill-posedness of the inverse problem \cite{lang2023identifiability,lu2021learning,li2021identifiability,tang2024identifiability}, we regularize the loss function by promoting sparsity and apply the PartInv (Partial Inversion) algorithm to address the optimization problems, as was done in \cite{carrillo2025sparse}.

Specifically, we consider $\hat{T}_0$ and $\hat{T}_1$ to be composed of some basis functions, that is,  we let $\hat{T}_1 = \sum_{i=1}^{n} c_i \Psi_i$, for $i=1,\cdots, n$ and define
\begin{align*}
    A_{ij} &= \frac{1}{T}\int_0^{T}\int_\mathbb{R} \Psi_i(x) \Psi_j(x) \frac{\rho}{T^2_0}dxdt,\\
    b_i &= -\frac{1}{2 T}\int_0^{T}\int_\mathbb{R} \frac{\Psi_i(x)}{T^2_0(x)} \left[ \epsilon^2 \partial_t J + v\partial_x \rho +2T_0(x)J \right] dxdt. 
\end{align*}
Since in practice, the data obtained from experiments are only discrete measurements of $\rho$ and $J$ at given time and spatial grid points, while for the basis functions, we can set them as continuous functions of the spatial variable $x$. Therefore, we use integration by parts to rewrite $b_i$ into the following form:
\begin{align}
    b_i 
    =&-\frac{\epsilon^2}{2 T}\int_\mathbb{R} \frac{\Psi_i(x)}{T^2_0 (x)} \left[ J(x,T)-J(x,0) \right] dx-\frac{1}{T}\int_0^T \int_\mathbb{R} \frac{\Psi_i(x)}{T_0 (x)} J dxdt \nonumber\\
    &- \frac{v}{2T}\int_0^T \left( \frac{\Psi_i}{T^2_0} \rho \right) \Bigg|_{-\infty}^{+\infty} dt + \frac{v}{2 T}\int_0^{T}\int_\mathbb{R} \partial_x\left(\frac{\Psi_i}{T^2_0}\right) \rho dxdt.\nonumber
\end{align}
The discrete version of the loss function now becomes:
\begin{equation*}
    \tilde{\mathcal{E}}[\hat{T}_1] = c^T A c - 2b^T c.
\end{equation*}

Similarly, we can also let $\hat{T}_0 = \sum_{i=1}^{n} \tilde{c}_i \Psi_i$, for $i=1,\cdots, n$ and since $\partial_t \rho + v \partial_x J = 0$ in Eq. \eqref{twoflux_1_1}, we can define 
\begin{align}
    \tilde{A}_{ij} =& \frac{1}{T}\int_0^{T}\int_\mathbb{R} \Psi_i(x) \Psi_j(x) J^2 dxdt,\nonumber\\
    \tilde{b}_i =& -\frac{1}{2T}\int_0^{T}\int_\mathbb{R} \Psi_i(x)  J \left[ \epsilon^2 \partial_t J + v\partial_x \rho +2 T_1(x) \rho \right] dxdt\nonumber\\
    =& -\frac{\epsilon^2}{4T}\int_\mathbb{R} \Psi_i(x) J^2  \Bigg|_{0}^{T}dx - \frac{v}{2T}\int_0^{T} \Psi_i(x) J \rho  \Bigg|_{-\infty}^{+\infty} dt \nonumber\\
    &+ \frac{v}{2T}\int_0^{T} \int_{\mathbb{R}} \partial_x (\Psi_i) J\rho dx dt - \frac{1}{4T}\int_{\mathbb{R}}\Psi_i(x) \rho^2  \Bigg|_{0}^{T}dx  - \frac{1}{T}\int_0^{T}\int_\mathbb{R}  T_1(x)\Psi_i(x) J \rho dxdt.\nonumber
\end{align}
The discrete loss function to recover $T_0$ now becomes:
\begin{equation*}
    \tilde{\mathcal{G}}[\hat{T}_0] = \tilde{c}^T \tilde{A} \tilde{c} - 2\tilde{b}^T \tilde{c}.
\end{equation*}

Let $R$ be chosen large enough such that the essential supports of $\rho$ and $J$ are contained in $[0, T] \times [-R,R]$, with $T,R>0$.
In the subsequent numerical tests, we present the fully discretized optimization problem in a computational domain $[0, T] \times [-R,R]$.

Given the space-time mesh size to be $(\Delta x, \Delta t)$, let's denote $t_l = l\Delta t$, where $l$ ranges from $0$ to $L = T/\Delta t$, and $x_m = m\Delta x$, where $m$ spans from $-M$ to $M$, with $M$ defined as $R/\Delta x$. For any function $f(x,t)$, let $ f_m^l\approx  f(x_m,t_l) $ and $f_m(t) \approx f(x_m,t)$. The given discrete input data are $\{ \rho_m^l:=\rho(x_m,t_l) \}_{m=-M, l=0}^{m=M, l=L}$, $\{ J_m^l:=J(x_m,t_l) \}_{m=-M, l=0}^{m=M, l=L}$.

Then, to recover $T_1$, the matrix $A$ and vector $b$ will take the following form:
\begin{align}
    A_{ij} \approx& \frac{1}{T} \sum_{m=-M, l=1}^{M, L} \Psi_{i,m} \Psi_{j,m} \frac{\rho_m^l}{T_{0,m}^2} \Delta x \Delta t,\\
    b_i \approx& -\frac{1}{2 T} \left[ \sum_{m=-M}^{M}  \epsilon^2 \frac{\Psi_{i,m}}{T_{0,m}^2} (J_m^T-J_m^0) \Delta x +  \sum_{l=1}^{L} v  \left(\frac{\Psi_{i,M}}{T_{0,M}^2} \rho_M^l - \frac{\Psi_{i,-M}}{T_{0,-M}^2} \rho_{-M}^l\right) \Delta t \right. \nonumber \\
    &\left. + \sum_{m=-M, l=1}^{M, L}\left( -v \rho_m^l\partial_x \left( \frac{\Psi_i (x_m)}{T_0^2(x_m)}\right) + 2 \frac{\Psi_{i,m}}{T_{0,m}}J_m^l \right) \Delta x  \Delta t \right] .\nonumber
\end{align}

Similarly, to recover $T_0$, the matrix $\tilde{A}$ and vector $\tilde{b}$ will take the following form:
\begin{align}
    \tilde{A}_{ij} \approx& \frac{1}{T} \sum_{m=-M, l=1}^{M, L} \Psi_{i,m} \Psi_{j,m} (J_m^l)^2 \Delta x \Delta t,\\
    \tilde{b}_i \approx& -\frac{1}{4T} \left[ \sum_{m=-M}^{M} \epsilon^2 \Psi_{i,m} \bigg((J_m^L)^2 - (J_m^0)^2 \bigg) \Delta x +  \sum_{l=1}^{L} 2v  \left(\Psi_{i,M}J_M^l \rho_M^l - \Psi_{i,-M}J_{-M}^l \rho_{-M}^l\right) \Delta t \right. \nonumber \\
    & - \sum_{m=-M, l=1}^{M, L} 2v\partial_x(\Psi_{i,m})J_m^l \rho_m^l\Delta x  \Delta t + \sum_{m=-M}^{M} \Psi_{i,m} \bigg((\rho_m^L)^2 - (\rho_m^0)^2 \bigg) \Delta x  \nonumber\\
    &\left. + \sum_{m=-M, l=1}^{M, L} 4 T_{1,m}\Psi_{i,m} J_m^l  \rho_m^l  \Delta x  \Delta t \right]. \nonumber
\end{align}

\subsection{Data generation}
The data are generated by solving the equations \eqref{twoflux} and \eqref{twoflux_1} using an asymptotic preserving (AP) scheme proposed in \cite{carrillo2013asymptotic}, with a space-time mesh size of $(\Delta x, \Delta t)$. One can refer Appendix. \ref{Appendix_B} for the details of the AP discretization. Let $f^+(x,0)=f^-(x,0)$, for $x\in[-R,R]$. Reflective boundary conditions are applied to both $f^+$ and $f^−$, such that
\begin{equation*}
    f^+(\pm R,t)=f^-(\pm R,t), \quad \forall t \in [0,T].
\end{equation*}
Thus, the boundary conditions for Eq. \eqref{twoflux_1} are for the flux density $J$, such that
\begin{equation*}
    J(-R,t)=J(R,t)=0, \quad \forall t \in [0,T].
\end{equation*} 
In this way, not only does it satisfy the boundary condition of $J$ required in the Section \ref{sec:Stability Analysis} as being 0, but it also ensures that the solution maintains conservation of mass throughout the entire time evolution process. 

Specifically, the data is produced with the parameters listed in Table \ref{Parameters_to_produce_data}, where $\mathcal{N}(\mu, \sigma^2)$ represents a Gaussian distribution with a mean of $\mu$ and a variance of $\sigma^2$.

\begin{table}[h]
    \centering
    \begin{tabular}{|c|c|c|c|c|c|}
        \hline
        $\Delta t$ & $\Delta x$ & Time domain & Spatial domain & Initial condition  \\ 
        \hline
        $1.25\cdot 10^{-5}$ & $5\cdot 10^{-3}$ & $[0,1]$ & $[-2,2]$ & $\rho: \mathcal{N}(0,0.5^2) $, $J: 0$  \\
        \hline
    \end{tabular}  
    \caption{Parameters to generate the solution data for recovery, using an numerical scheme that has uniform accuracy and stability with respect to $\epsilon$. The details of the numerical discretization is shown in Appendix. \ref{Appendix_B}.} 
    \label{Parameters_to_produce_data}
\end{table}

Since the real experimental data are noisy, to test the robustness of our recovery method, we add noise to the obtained numerical solution. The data used in the loss function are given by 
\begin{equation*}
    \{\rho(x_m,t_l) + \epsilon_m^l \}_{m=-M, l=0}^{m=M, l=L}, \quad\ \{ J(x_m,t_l) + \tilde{\epsilon}_m^l \}_{m=-M, l=0}^{m=M, l=L}.
\end{equation*}
Here $\{\epsilon_m^l \}_{m=-M, l=0}^{m=M, l=L}$ and $\{\tilde{\epsilon}_m^l \}_{m=-M, l=0}^{m=M, l=L}$ are the discrete added noise which satisfies a mean zero normal distribution, such that
\begin{align}
    \epsilon_m^l \textasciitilde \mathcal{N} (0, \sigma_l^2),\quad \sigma_l = s_{n} \times \sqrt{\frac{1}{2M+1} \Sigma_{m=-M}^{m=M} |\rho_m^l|^2},\nonumber\\
     \tilde{\epsilon}_m^l \textasciitilde \mathcal{N} (0, \tilde{\sigma}_l^2),\quad \tilde{\sigma}_l = s_{n} \times \sqrt{\frac{1}{2M+1} \Sigma_{m=-M}^{m=M} |J_m^l|^2},\nonumber
\end{align}
where $s_n$ is the noise level which can be taken as $1\%,\ 5\% $, and so on. 

\subsection{Setup of the test cases}\label{Sec:setup of the text cases}
The performance of the recovery method is systematically evaluated across three representative test cases, each featuring distinct functional forms of the turning kernel components $T_0(x)$ and $T_1(x)$. The true $T_0(x)$, $T_1(x)$, and their corresponding basis function sets $\{\Psi_k\}_{k=1}^{10}$ are specified as follows:
\begin{itemize}
    \item Case 1: 
    \begin{equation*}
        \begin{cases}
            T_0(x)= 3.5(1-0.2 \tanh(x)),\\
            T_1(x)= 3(1-0.2 \tanh(x)),
        \end{cases}
        \ \text{with}\ 
        \begin{cases}
            \Psi_1 = 1,\\
            \Psi_k(x) = \tanh((k-1)x),\ k\geq 2.
        \end{cases}
    \end{equation*}

    \item Case 2: 
    \begin{equation*}
        \begin{cases}
            T_0(x) = \frac{3}{\sqrt{\pi}} \exp(-x^2)+\frac{3}{\sqrt{2 \pi}} \exp(-x^2/2),\\
            T_1(x) = \frac{2}{\sqrt{\pi}} \exp(-x^2)+\frac{2}{\sqrt{2 \pi}} \exp(-x^2/2),
        \end{cases}
        \ \text{with }\
        \Psi_k(x) = \frac{\exp(-k x^2/2)}{\sqrt{\pi}}.
    \end{equation*}

    \item Case 3: 
    \begin{equation*}
        \begin{cases}
            T_0(x) &= \frac{|x|+1}{4},\\
             T_1(x) &= \frac{|x|+1}{8},
        \end{cases}
        \ \text{with }\
        \begin{cases}
            \Psi_1 = 1,\\
            \Psi_k(x) = (\frac{|x|}{2})^{k-1}, k\geq 2.
        \end{cases}
    \end{equation*}
\end{itemize}

The true $T_0(x)$, $T_1(x)$, and basis functions for all three cases are visualized in Figure \ref{fig:T0T1} and Figure \ref{fig:basis_functions}, respectively. The corresponding true coefficient vectors $\mathbf{c} = (c_1, c_2, c_3, \dots)^\intercal$  for the optimization problem are summarized in Table \ref{true_solutions}.

\begin{figure}[!htbp]  
    \centering 
    \begin{subfigure}{0.32\textwidth}
        \centering
        \includegraphics[width=\linewidth]{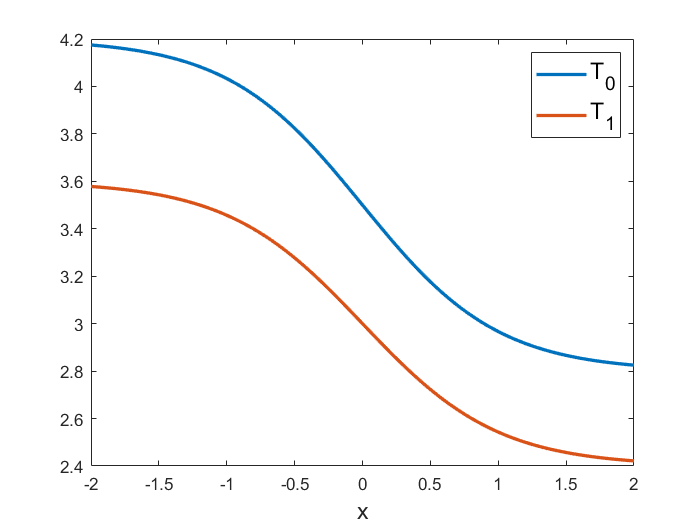} 
        \vspace{-2em}  
        \caption{Case 1}
    \end{subfigure}
    \begin{subfigure}{0.32\textwidth}
        \centering
        \includegraphics[width=\linewidth]{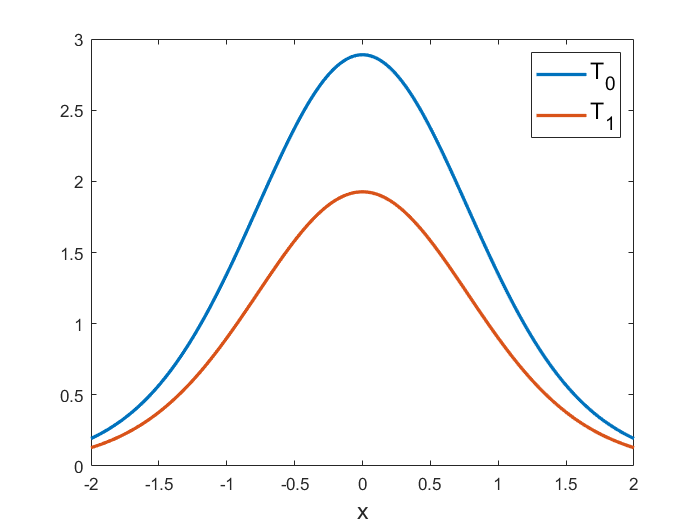}  
        \vspace{-2em}  
        \caption{Case 2}
    \end{subfigure}
    \begin{subfigure}{0.32\textwidth}
        \centering
        \includegraphics[width=\linewidth]{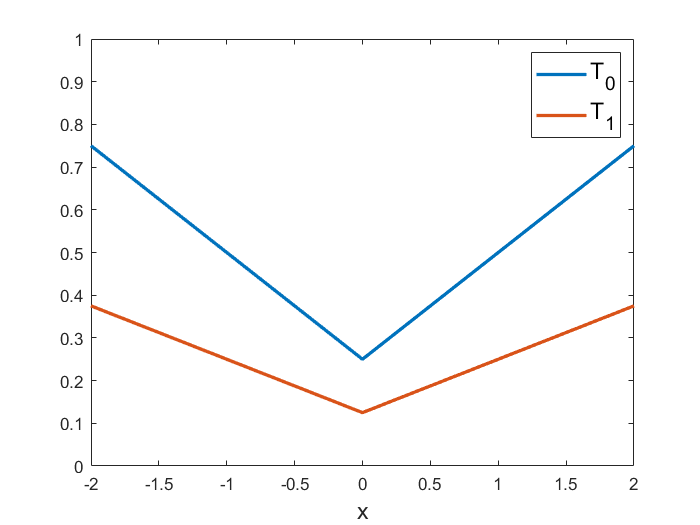}  
        \vspace{-2em} 
        \caption{Case 3}
    \end{subfigure}
    \caption{True $T_0(x)$ and $T_1(x)$ used in the three test cases. (a) Case 1: $T_0(x)$ and $T_1(x)$ are smooth monotone decreasing functions constructed from a constant term and a hyperbolic-tangent term. (b) Case 2: $T_0(x)$ and $T_1(x)$ are symmetric functions formed by linear combinations of two Gaussian profiles with different widths.  (c) Case 3: $T_0(x)$ and $T_1(x)$ are piecewise-linear functions proportional to $|x|+1$. }
    \label{fig:T0T1} 
\end{figure}
\begin{figure}[!htbp]  
    \centering 
    
    \begin{subfigure}{0.31\textwidth}
        \centering
        \includegraphics[width=\linewidth]{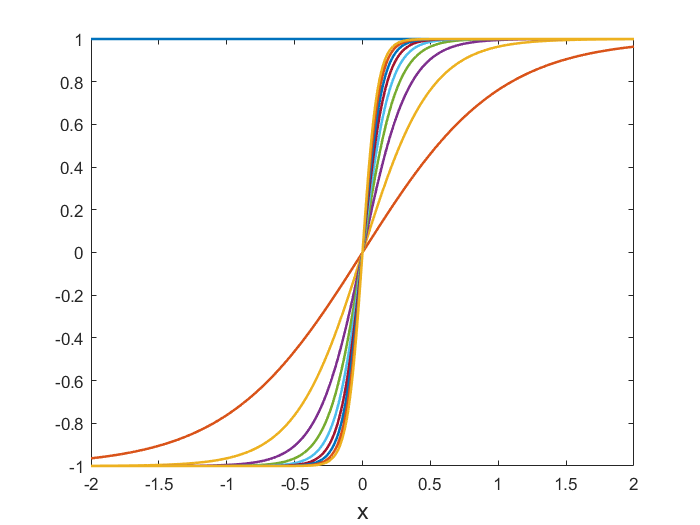} 
        \vspace{-2em}  
        \caption{Case 1}
    \end{subfigure}
    \begin{subfigure}{0.31\textwidth}
        \centering
        \includegraphics[width=\linewidth]{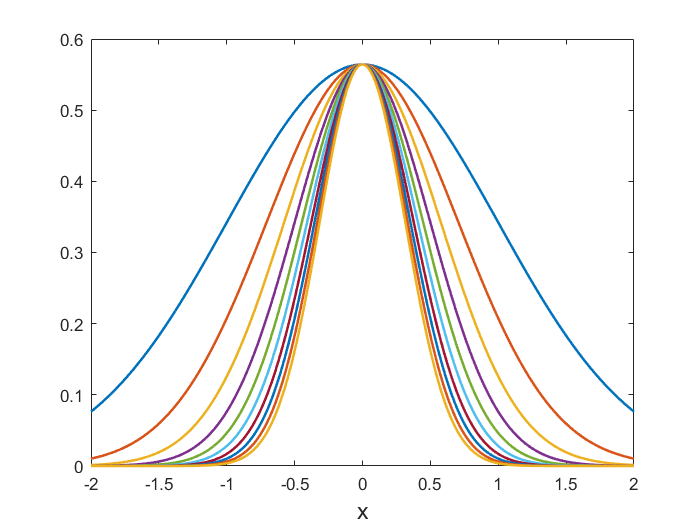}  
        \vspace{-2em} 
        \caption{Case 2}
    \end{subfigure}
    \begin{subfigure}{0.31\textwidth}
        \centering
        \includegraphics[width=\linewidth]{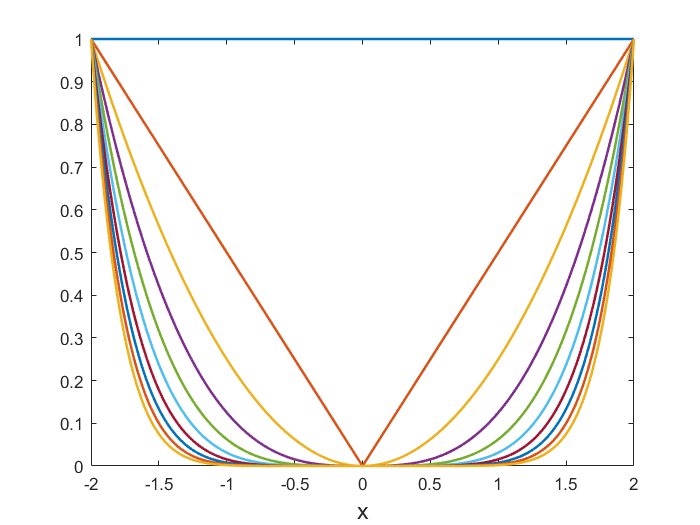}  
        \vspace{-2em}  
        \caption{Case 3}
    \end{subfigure}
    \hfill
    \hspace{-0.7cm}
    \begin{subfigure}{0.06\textwidth}
        \centering
        \includegraphics[width=\linewidth]{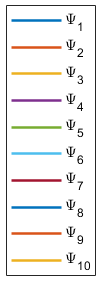} 
        \vspace{2.2em}
    \end{subfigure}
    \begin{minipage}{\textwidth}
    \centering
    \caption{Ten basis functions used in the three test cases. (a) Case 1: a constant basis function and nine hyperbolic-tangent basis functions with different spatial scales. (b) Case 2: symmetric Gaussian basis functions with different widths. (c) Case 3: a constant function and nine symmetric polynomial basis functions of $|x|/2$.}
    \label{fig:basis_functions}
    \end{minipage}
\end{figure}

\begin{table}[H]
    \centering
    \begin{tabular}{c|cc|cc}
        \hline
         & \multicolumn{2}{c|}{$T_0$} & \multicolumn{2}{c}{$T_1$}\\
         \hline
         & $c_1$ & $c_2$ & $c_1$ & $c_2$  \\
         \hline
        \textbf{Case 1}  & 3.5 & -0.7 & 3 & -0.6 \\
        \textbf{Case 2}  & 2.1213 & 3 & 1.4142 & 2 \\
        \textbf{Case 3}  & 0.25 & 0.5 & 0.125 & 0.25 \\
        \hline
    \end{tabular}
    \caption{Coefficient vectors $\boldsymbol{c}=(c_1,\ldots,c_{10})$ for the three test cases. In all cases, only $c_1$ and $c_2$ are nonzero, while $c_3,\ldots,c_{10}$ are set to zero.}

    \label{true_solutions}
\end{table}
\subsection{Algorithm and the Comparison Critiria}\label{sec:algorithm}
We assume that $T_0(x)$ and $T_1(x)$ are typically sparse relative to the given basis functions, and consider the following Basis Pursuit (BP) problem:
\begin{equation}\label{optimization}
    \begin{aligned}
    \mbox{minimize}_{\mathbf{c}\in \mathbb{R}^n} \|\mathbf{c}\|_1,\\
    \mbox{subject to}\ \mathbf{Ac=b}.
    \end{aligned} 
\end{equation}
Same as in \cite{carrillo2025sparse}, we use the PartInv Algorithm proposed in \cite{chen2014guaranteed} to solve \eqref{optimization}. The specific steps of PartInv are summarized in the following Algorithm 1 and one can refer to \cite{chen2014guaranteed} for more details.

\begin{algorithm}\label{algorithm}
    \caption{\textbf{PartInv}: Given $\mathbf{Ac=b}$ where the ground truth is s-sparse, return the best $K$-sparse approximation $\mathbf{\hat{c}}$}
    \LinesNumbered 
    \KwIn{$\mathbf{A}$, $\mathbf{b}$, $K$ (an upper bound on sparsity $s$)}
    $\tilde{\mathbf{c}} \leftarrow \mathbf{A}^*\mathbf{b}$;\quad $I^{(0)}\leftarrow$ indices of the $K$-largest magnitudes of $\tilde{\mathbf{c}}$;\quad $k \leftarrow 0$\\
    \textbf{while} Stopping condition not met \textbf{do}\\
    \qquad $\tilde{\mathbf{c}}_{I^{(k)}} \leftarrow \mathbf{A}_{I^{(k)}}^{+} \mathbf{b}$\\
    \qquad $\mathbf{r} \leftarrow \mathbf{b} - \mathbf{A}_{I^{(k)}} \tilde{\mathbf{c}}_{I^{(k)}}$\\
    \qquad $J^{(k)} \leftarrow \Omega$ \textbackslash $I^{(k)}$\\
    \qquad $\tilde{\mathbf{c}}_{J^{(k)}} \leftarrow \mathbf{A}_{J^{(k)}}^{*} \mathbf{r}$\\
    \qquad $I^{(k+1)}\leftarrow$ indices of the $K$-largest magnitudes of $\tilde{\mathbf{c}}$\\
    \qquad $k \leftarrow k+1$\\
    \textbf{end while}\\
    \textbf{Return} $\hat{\mathbf{c}} = \mathbf{A}_{I^{(k)}}^{+} \mathbf{b}$
\end{algorithm}

In the following part, to check how well the proposed parameter recovery method works and how reliable it is, we use three different comparison criteria. Each one looks at a different aspect of the performance:

\begin{itemize}
    \item[CCI.] \textbf{Direct Coefficient Recovery Accuracy:} The main goal here is to recover the optimization variables — the coefficient vectors $c=(c_1,c_2,\dots,c_{10})$ for both $T_0$ and $T_1$. 
    \item[CCII.]  \textbf{Accuracy of the Reconstructed Function:} Beyond just the coefficients, we also compare the full functions $T_0$ and $T_1$ — built from the recovered coefficients — with the true functions. The algorithm's performance is measured by the relative reconstruction error:
    \begin{equation*}
        E^{i}_{\text{reconst}} = \frac{\|T_i-\hat{T}_i\|_{L^2_x}}{\|T_i\|_{L^2_x}},\quad i=1,2,
    \end{equation*}
    where $\hat{T}_i$ is the approximation of the true solution $T_i$.
    \item[CCIII.] \textbf{Accuracy of the Obtained Forward Solution:} Concerning the stability results from Section \ref{sec:Stability Analysis}, we directly test the recovery by checking the solutions we get. Specifically, the recovered $\hat{T}_0$ and $\hat{T}_1$ (from both PartInv and LS) are put back into the forward system (Eq. \eqref{twoflux_1}) to compute the corresponding density $\hat{\rho}$ and flux $\hat{J}$. We then compare these with the true solutions $\rho$ and $J$ using the mixed error defined as:
    \begin{equation*}
        E^{\rho}_{\text{mixed}} = \min\left\{ \frac{|\rho-\hat{\rho}|}{|\rho|},\ |\rho-\hat{\rho}| \right\}, \quad E^{J}_{\text{mixed}} = \min\left\{ \frac{|J-\hat{J}|}{|J|},\ |J-\hat{J}| \right\}.
    \end{equation*}
\end{itemize}

By looking at multiple aspects — such as direct parameter error, how well the functions are reconstructed, and the stability of the forward solution— we give a complete and careful check of the proposed method's accuracy, reliability, and real-world usefulness for recovering key parameters in the chemotaxis multiscale model.

\subsection{Results}
\subsubsection{Comparison of PartInv and Least Square}
This section compares the sparsity-promoting PartInv algorithm with the standard Least Squares (LS) method. The goal is to recover \(T_0(x)\) and \(T_1(x)\) in three test cases, using noise-free data with \(\epsilon = 1\). The evaluation is based on three main criteria: how well the coefficients are recovered, how accurately the functions are reconstructed, and whether the forward solution remains consistent.

Figures \ref{fig:LS_vs_PI_Case1} to \ref{fig:LS_vs_PI_Case3} show a clear difference between the two methods in estimating the coefficients. LS, which minimizes residual errors without enforcing sparsity, fails to recover the true coefficients in all cases. For example, in Case 1, where the true coefficients for \(T_0\) are \(c_1 = 3.5\) and \(c_2 = -0.7\), LS produces large errors in both coefficients. In contrast, PartInv correctly identifies the exact 2-sparse structure. The same pattern holds for Case 2 and Case 3: LS either overestimates or underestimates the coefficients, while PartInv recovers them with high accuracy.

Beyond coefficient errors, the reconstructed functions \(\hat{T}_0(x)\) and \(\hat{T}_1(x)\) from PartInv match the true profiles very well in all cases (see Figures \ref{fig:LS_vs_PI_Case1} to \ref{fig:LS_vs_PI_Case3}). LS, however, leads to clear distortions. 

From Figures \ref{fig:LS_vs_PI_Case1} to \ref{fig:LS_vs_PI_Case3}, one has observed that the obtained $T_0(x)$ and $T_1(x)$ are the furthest for Case 3, by the two different optimaization method: PartInv and LS. By plugging the recovered coefficients \(\hat{c}\) from both methods back into the forward system (Eq. \eqref{twoflux_1}), we then compare the resulting density \(\hat{\rho}\) and flux \(\hat{J}\) with the true solutions.  Figure \ref{fig:compare_rhoJ_case3} shows this comparison of Case 3, using the mixed error measures \(E_{\text{mixed}}^{\rho}\) and \(E_{\text{mixed}}^{J}\) defined in Section \ref{sec:algorithm}. From Figure \ref{fig:compare_rhoJ_case3}, we see that even though LS performs poorly in recovering the coefficients and functions, its forward solutions \(\hat{\rho}\) and \(\hat{J}\) are still very close to the true ones, with mixed errors as small as \(10^{-3}\). PartInv achieves even smaller errors.

These results show that both methods can produce reliable forward solutions for the kinetic model. However, PartInv is better than LS in two important ways, it accurately recovers the sparse physical parameters \(T_0(x)\) and \(T_1(x)\) and yields forward solutions with smaller errors. 

\begin{figure}[!htbp]  
    \centering

    \begin{subfigure}{0.34\textwidth}
        \centering
        \includegraphics[width=\linewidth]{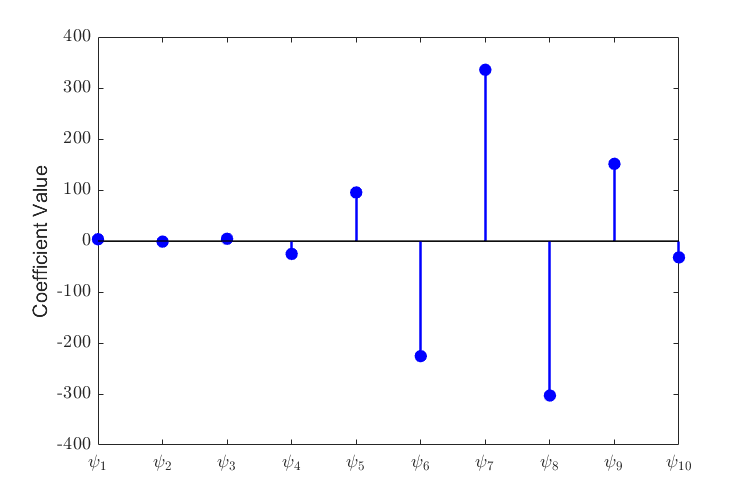} 
        \vspace{-2em} 
        \caption{Coefficient of $T_0$ by LS}
    \end{subfigure}
    \begin{subfigure}{0.34\textwidth}
        \centering
        \includegraphics[width=\linewidth]{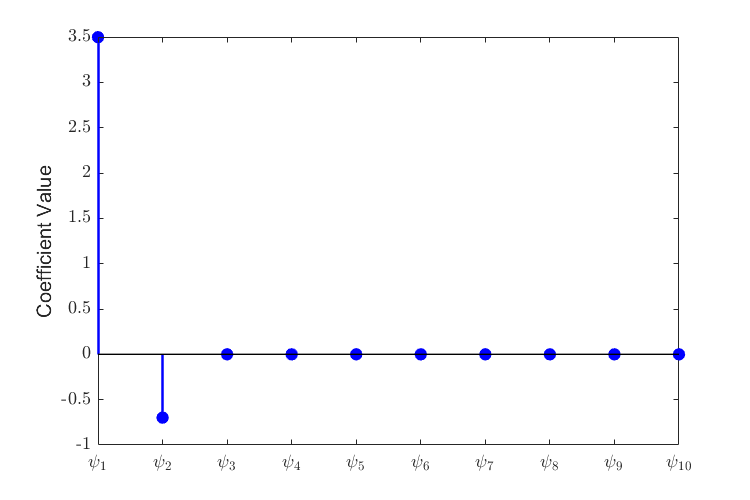}  
        \vspace{-2em}  
        \caption{Coefficient of $T_0$ by PartInv}
    \end{subfigure}
    \begin{subfigure}{0.3\textwidth}
        \centering
        \includegraphics[width=\linewidth]{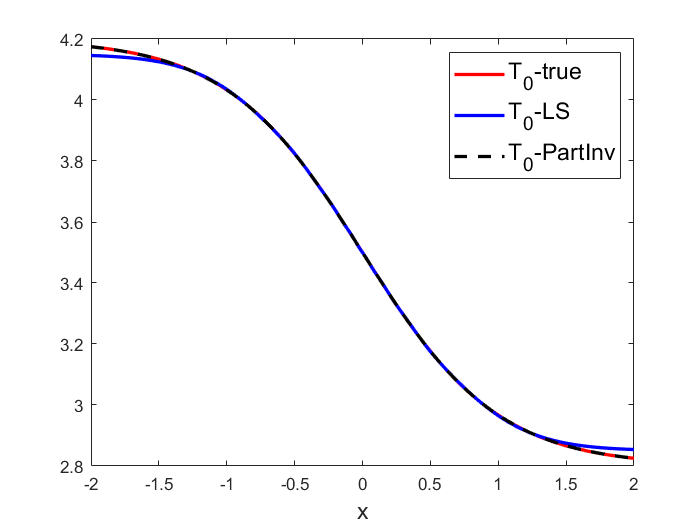}  
        \vspace{-2em}  
        \caption{Comparison of $T_0$}
    \end{subfigure}
    
    \vspace{0.5em}
    
    \begin{subfigure}{0.34\textwidth}
        \centering
        \includegraphics[width=\linewidth]{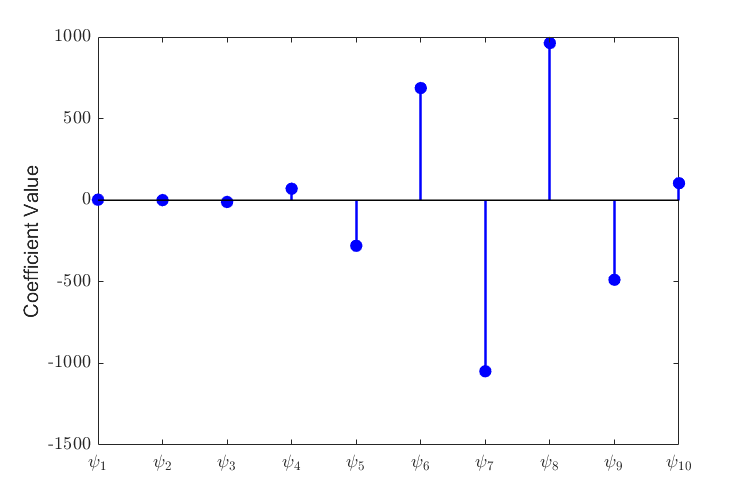}  
        \vspace{-2em}  
        \caption{Coefficient of $T_1$ by LS}
    \end{subfigure}
    \begin{subfigure}{0.34\textwidth}
        \centering
        \includegraphics[width=\linewidth]{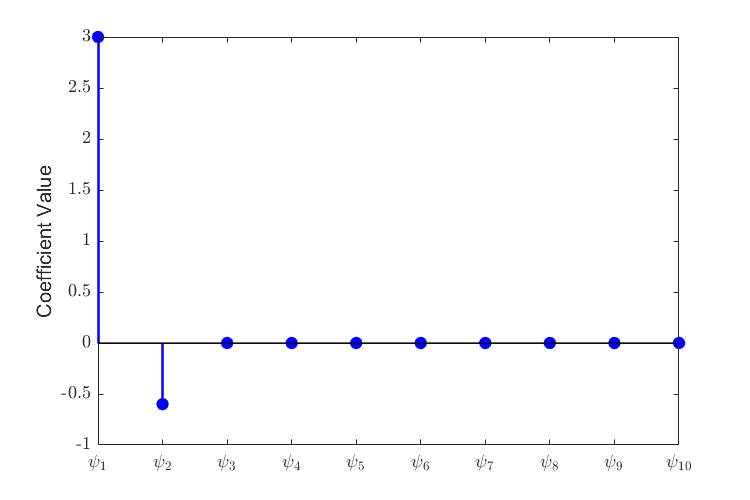} 
        \vspace{-2em}  
        \caption{Coefficient of $T_1$ by PartInv}
    \end{subfigure}
    \begin{subfigure}{0.3\textwidth}
        \centering
        \includegraphics[width=\linewidth]{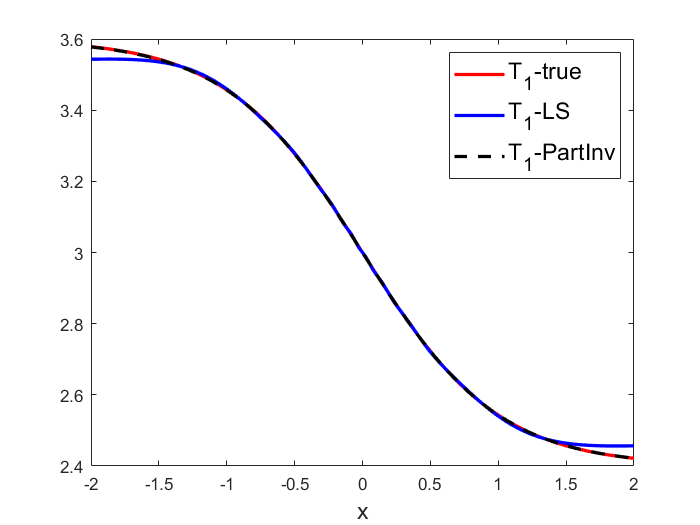}  
        \vspace{-2em}  
        \caption{Comparison of $T_1$}
    \end{subfigure}
    
    \begin{minipage}{\textwidth}
    \centering
    \caption{Recovery results for $T_0(x)$ and $T_1(x)$ in Case 1 using noise-free data with $\epsilon=1$. (a) (d): Coefficients of $T_0$ and $T_1$, respectively, recovered by the LS method. (b) (e): Coefficients of $T_0$ and $T_1$, respectively, recovered by the PartInv method with sparsity parameter $K=2$. (c) (f): Comparisons between the true and reconstructed functions for $T_0$ and $T_1$, respectively. }
    \label{fig:LS_vs_PI_Case1}
    \end{minipage}
    
\end{figure}

\begin{figure}
    \centering
    \begin{subfigure}{0.34\textwidth}
        \centering
        \includegraphics[width=\linewidth]{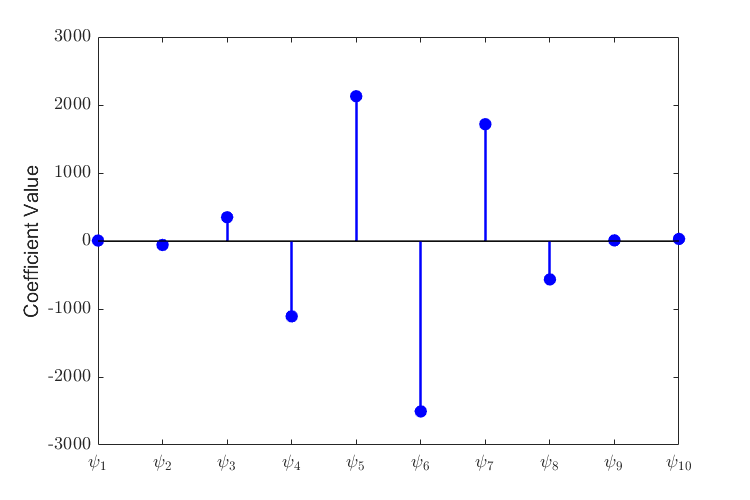} 
        \vspace{-2em}  
        \caption{Coefficient of $T_0$ by LS}
    \end{subfigure}
    \begin{subfigure}{0.34\textwidth}
        \centering
        \includegraphics[width=\linewidth]{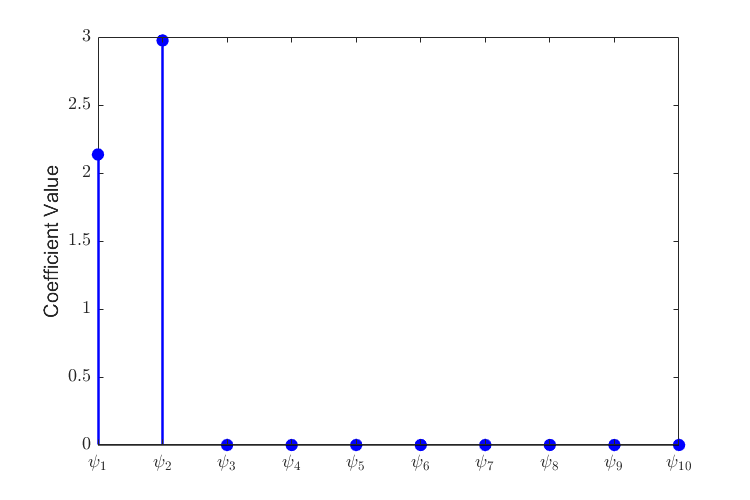}  
        \vspace{-2em} 
        \caption{Coefficient of $T_0$ by PartInv}
    \end{subfigure}
    \begin{subfigure}{0.3\textwidth}
        \centering
        \includegraphics[width=\linewidth]{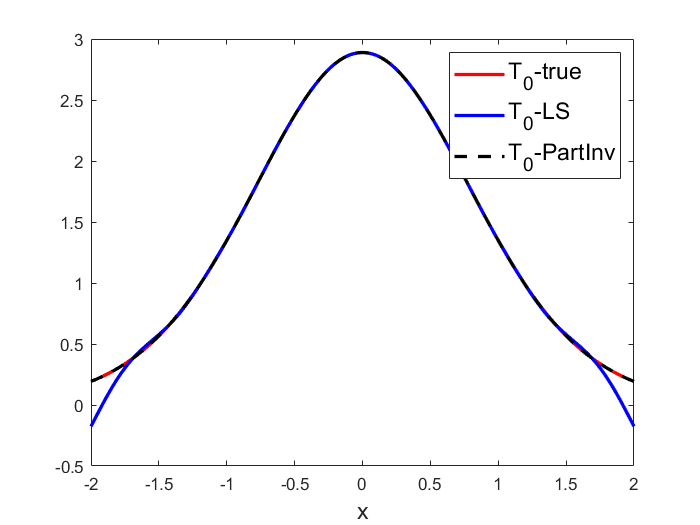}  
        \vspace{-2em}  
        \caption{Comparison of $T_0$}
    \end{subfigure}
    
    \vspace{0.5em}

    \begin{subfigure}{0.34\textwidth}
        \centering
        \includegraphics[width=\linewidth]{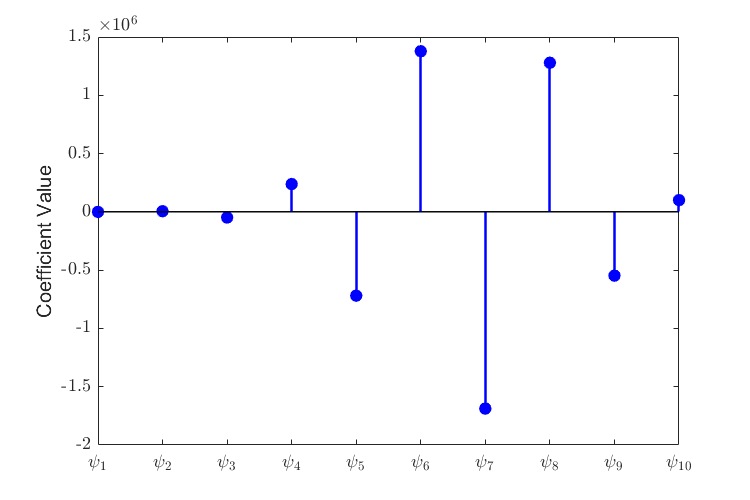}  
        \vspace{-2em}  
        \caption{Coefficient of $T_1$ by LS}
    \end{subfigure}
    \hfill
    \begin{subfigure}{0.34\textwidth}
        \centering
        \includegraphics[width=\linewidth]{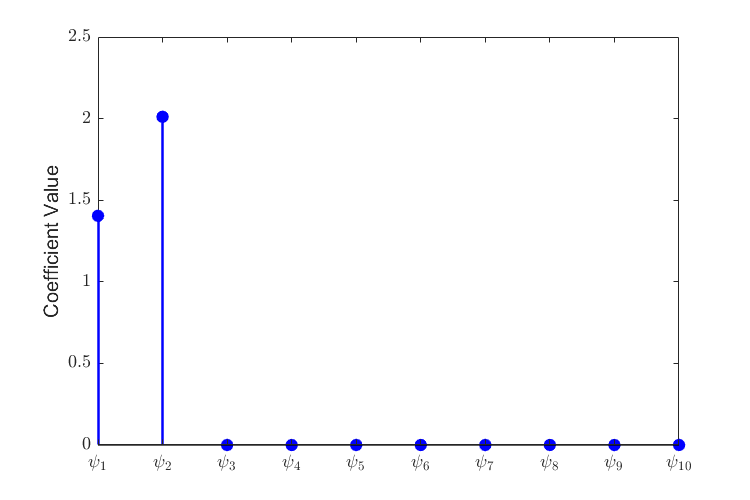}  
        \vspace{-2em}  
        \caption{Coefficient of $T_1$ by PartInv}
    \end{subfigure}
    \hfill
    \begin{subfigure}{0.3\textwidth}
        \centering
        \includegraphics[width=\linewidth]{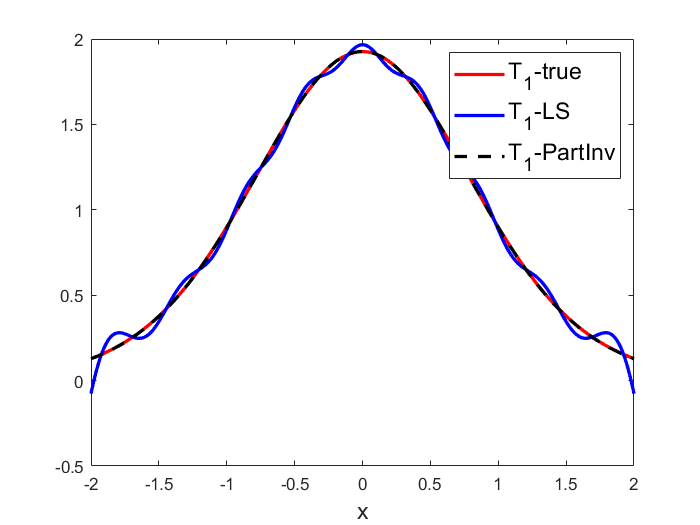}  
        \vspace{-2em}  
        \caption{Comparison of $T_1$}
    \end{subfigure}
    \caption{Recovery results for $T_0(x)$ and $T_1(x)$ in Case 2 using noise-free data with $\epsilon=1$. (a) (d): Coefficients of $T_0$ and $T_1$, respectively, recovered by the LS method. (b) (e): Coefficients of $T_0$ and $T_1$, respectively, recovered by the PartInv method with sparsity parameter $K=2$. (c) (f): Comparisons between the true and reconstructed functions for $T_0$ and $T_1$, respectively. }
    \label{fig:LS_vs_PI_Case2}
\end{figure}

\begin{figure} 
    \centering
    \begin{subfigure}{0.34\textwidth}
        \centering
        \includegraphics[width=\linewidth]{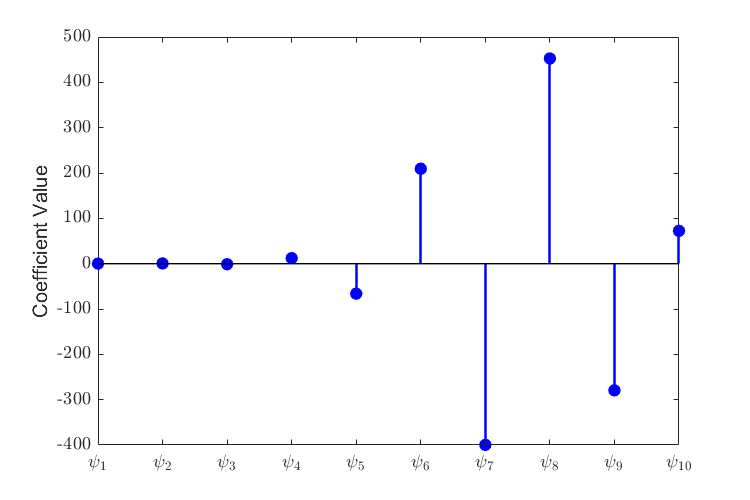}  
        \vspace{-2em} 
        \caption{Coefficient of $T_0$ by LS}
    \end{subfigure}
    \hfill
    \begin{subfigure}{0.34\textwidth}
        \centering
        \includegraphics[width=\linewidth]{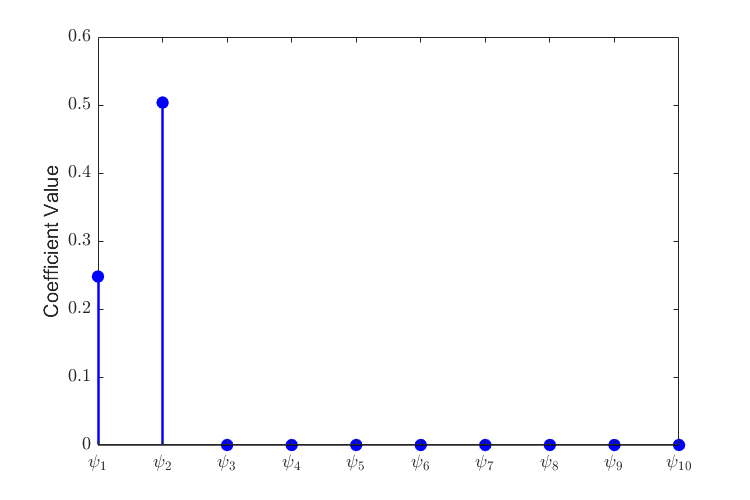}  
        \vspace{-2em}  
        \caption{Coefficient of $T_0$ by PartInv}
    \end{subfigure}
    \begin{subfigure}{0.3\textwidth}
        \centering
        \includegraphics[width=\linewidth]{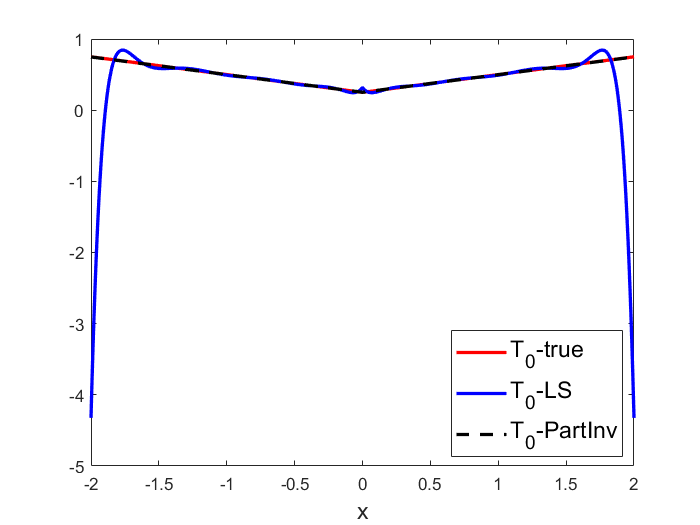} 
        \vspace{-2em}  
        \caption{Comparison of $T_0$}
    \end{subfigure}
    
    \vspace{0.5em}

    \begin{subfigure}{0.34\textwidth}
        \centering
        \includegraphics[width=\linewidth]{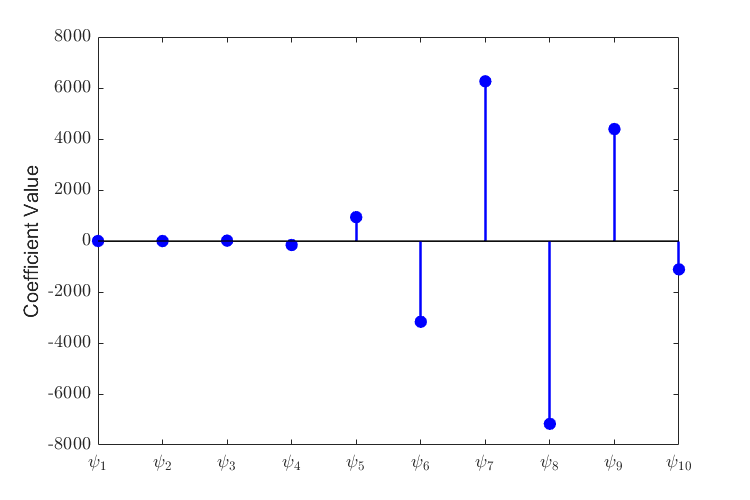}  
        \vspace{-2em} 
        \caption{Coefficient of $T_1$ by LS}
    \end{subfigure}
    \hfill
    \begin{subfigure}{0.34\textwidth}
        \centering
        \includegraphics[width=\linewidth]{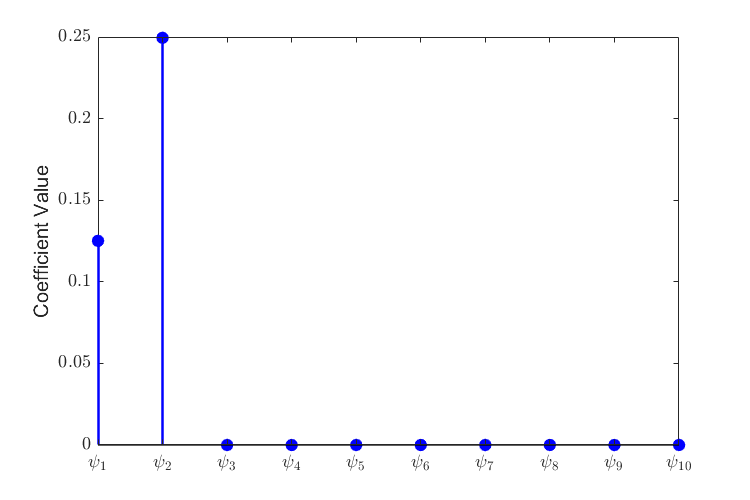}  
        \vspace{-2em}  
        \caption{Coefficient of $T_1$ by PartInv}
    \end{subfigure}
    \hfill
    \begin{subfigure}{0.3\textwidth}
        \centering
        \includegraphics[width=\linewidth]{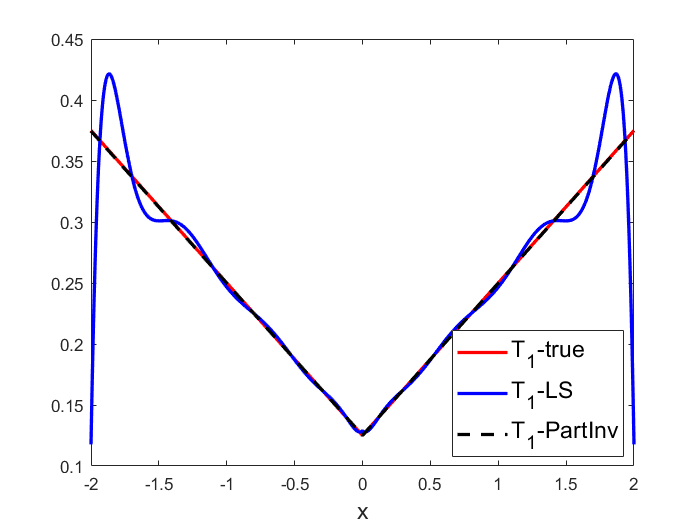}  
        \vspace{-2em}  
        \caption{Comparison of $T_1$}
    \end{subfigure}
    
    \begin{minipage}{\textwidth}
    \centering
    \caption{Recovery results for $T_0(x)$ and $T_1(x)$ in Case 3 using noise-free data with $\epsilon=1$. (a) (d): Coefficients of $T_0$ and $T_1$, respectively, recovered by the LS method. (b) (e): Coefficients of $T_0$ and $T_1$, respectively, recovered by the PartInv method with sparsity parameter $K=2$. (c) (f): Comparisons between the true and reconstructed functions for $T_0$ and $T_1$, respectively. }
    \label{fig:LS_vs_PI_Case3}
    \end{minipage}
    
\end{figure}

\begin{figure}[!htbp]  
    \centering 
    
    \begin{subfigure}{0.28\textwidth}
        \centering
        \includegraphics[width=\linewidth]{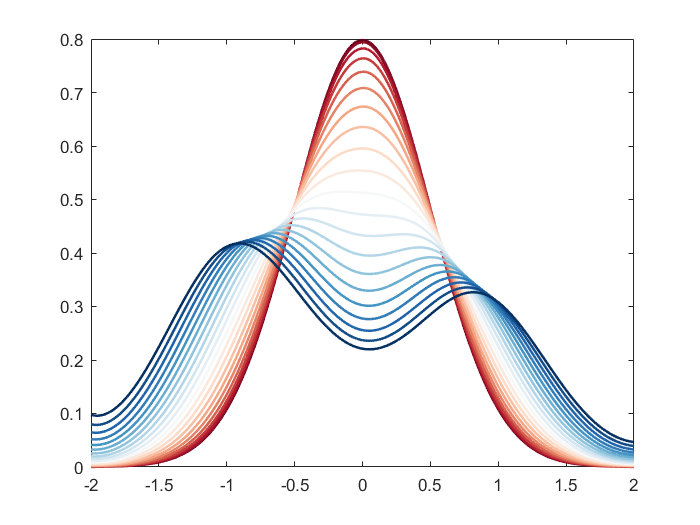}  
        \vspace{-2em} 
        \caption{True solution of $\rho$}
    \end{subfigure}
    \hfill
    \begin{subfigure}{0.28\textwidth}
        \centering
        \includegraphics[width=\linewidth]{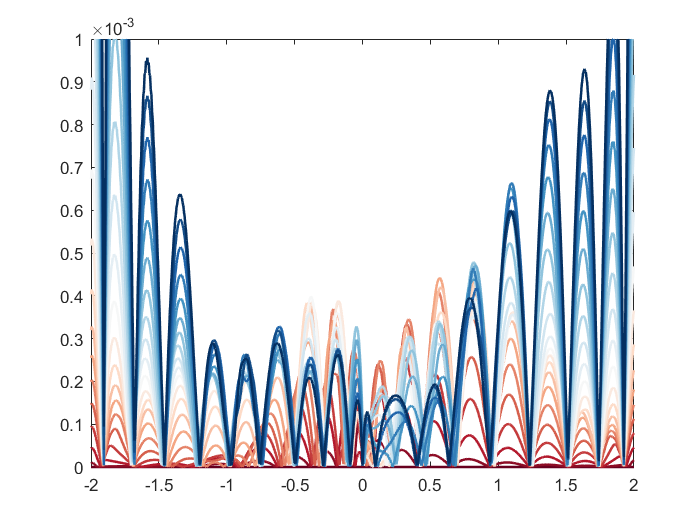}  
        \vspace{-2em} 
        \caption{$E_{mixed}^\rho$ of LS}
    \end{subfigure}
    \hfill
    \begin{subfigure}{0.28\textwidth}
        \centering
        \includegraphics[width=\linewidth]{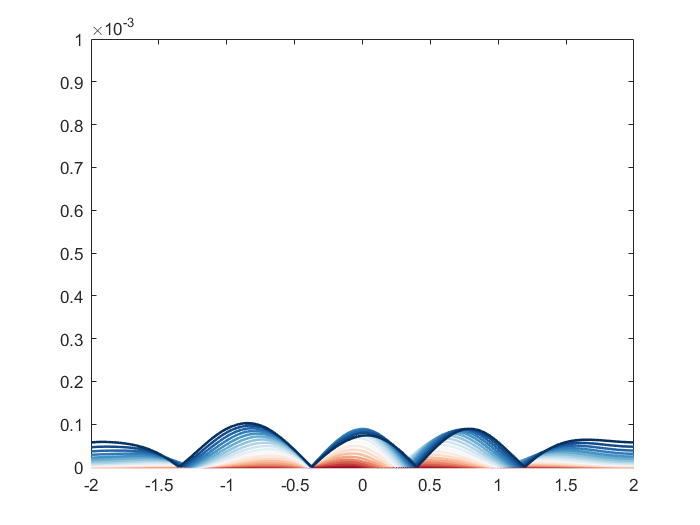}  
        \vspace{-2em}  
        \caption{$E_{mixed}^\rho$ of PartInv}
    \end{subfigure}
    \hfill
    \hspace{-0.7cm}
    \begin{subfigure}{0.06\textwidth}
        \centering
        \includegraphics[width=\linewidth]{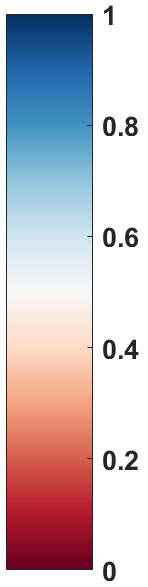} 
        \vspace{-1em}
    \end{subfigure}

    \vspace{0.5em}

    \begin{subfigure}{0.28\textwidth}
        \centering
        \includegraphics[width=\linewidth]{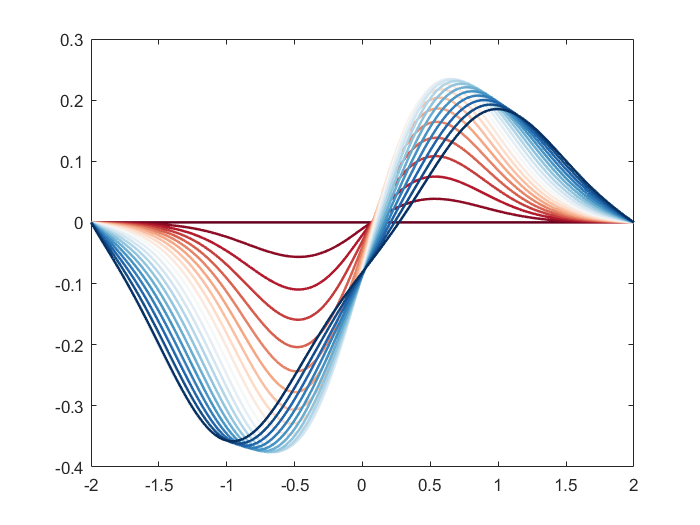}  
        \vspace{-2em}  
        \caption{True solution of $J$}
    \end{subfigure}
    \hfill   
    \begin{subfigure}{0.28\textwidth}
        \centering
        \includegraphics[width=\linewidth]{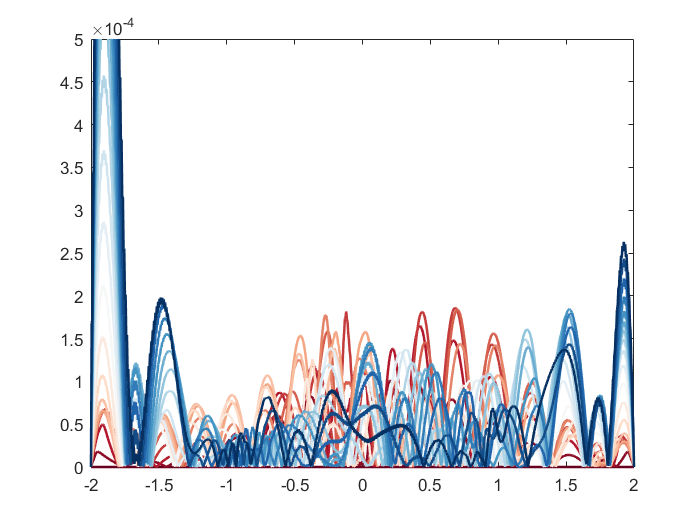}  
        \vspace{-2em}  
        \caption{$E_{mixed}^J$ of LS}
    \end{subfigure}
    \hfill
    \begin{subfigure}{0.28\textwidth}
        \centering
        \includegraphics[width=\linewidth]{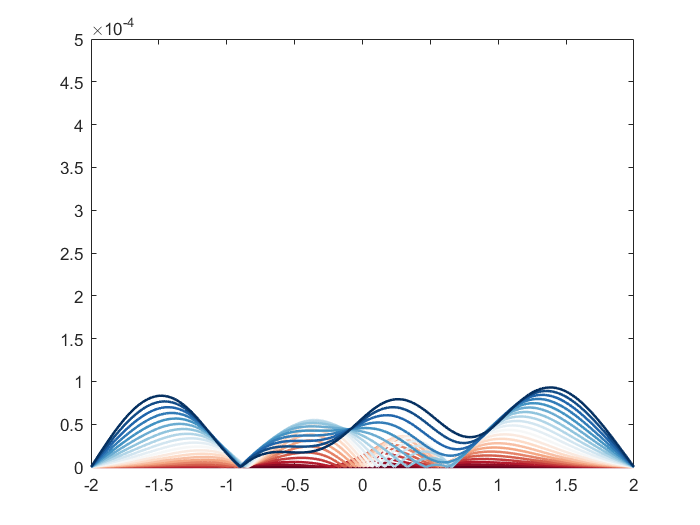}  
        \vspace{-2em}  
        \caption{$E_{mixed}^J$ of PartInv }
    \end{subfigure}
    \hfill
    \hspace{-0.7cm}
    \begin{subfigure}{0.06\textwidth}
        \centering
        \includegraphics[width=\linewidth]{colorbar.png} 
        \vspace{-1em}
    \end{subfigure}

    \begin{minipage}{\textwidth}
    \centering
    \caption{Mixed errors of the forward solutions in Case 3 using noise-free data with $\epsilon=1$. (a) (d): True solutions $\rho$ and $J$, respectively. (b) (e): Mixed errors of $\rho$ and $J$, respectively, obtained using the LS-recovered turning kernels. (c) (f):  Mixed errors obtained using the PartInv-recovered turning kernels. Different color indicates different time.}
    \label{fig:compare_rhoJ_case3}
    \end{minipage}
\end{figure}

\subsubsection{Validation of the Stability Analysis }\label{sec:Stability analysis validation}

This section tests the theoretical stability results in Section \ref{sec:Stability Analysis}. Figure \ref{fig:compare_rhoJ_3cases} shows the forward solution errors for all three test cases with \(\epsilon = 1\), using noise-free data. As seen in Figures \ref{fig:LS_vs_PI_Case1} to \ref{fig:LS_vs_PI_Case3}, PartInv recovers the true \(T_0\) and \(T_1\) very accurately, meaning the loss functionals \(\mathcal{G}[\hat{T}_0]\) and \(\mathcal{E}[\hat{T}_1]\) are extremely small. When these recovered parameters are plugged back into the forward system (Eq. \eqref{twoflux_1}), the resulting density \(\hat{\rho}\) and flux \(\hat{J}\) are very close to the true solutions, with the mixed errors \(E_{\text{mixed}}^\rho\) and \(E_{\text{mixed}}^J\) are on the order of \(10^{-4}\) or smaller. This indicates that \(\|\bar{\rho}(x,T)\|^2_{L^2_x} := \|\rho(x,T) - \hat{\rho}(x,T)\|^2_{L^2_x}\) and \(\|\bar{J}\|^2_{L^2_x} := \|J(x,T) - \hat{J}(x,T)\|^2_{L^2_x}\) are very small.

To show the uniform accuracy with respect to $\epsilon$, Table \ref{Case_123_multiscale} summarizes the recovery performance for different values of \(\epsilon = 1,\ 0.1,\ 0.01\). For all three cases, the loss functionals \(\mathcal{G}[\hat{T}_0]\) and \(\mathcal{E}[\hat{T}_1]\), as well as the forward errors \(\|\bar{\rho}\|^2_{L^2_x}\) and \(\|\bar{J}\|^2_{L^2_x}\), do not change much as \(\epsilon\) decreases from \(1\) to \(0.01\). This shows that the designed loss functions and PartInv work reliably across all regimes — from the kinetic regime (\(\epsilon = O(1)\)) to the challenging diffusion limit (\(\epsilon \to 0\)) — without losing accuracy corroborating the conditional theoretical results of Section 3.

\begin{figure}[!htbp]  
    \centering 
    
    \begin{subfigure}{0.3\textwidth}
        \centering
        \includegraphics[width=\linewidth]{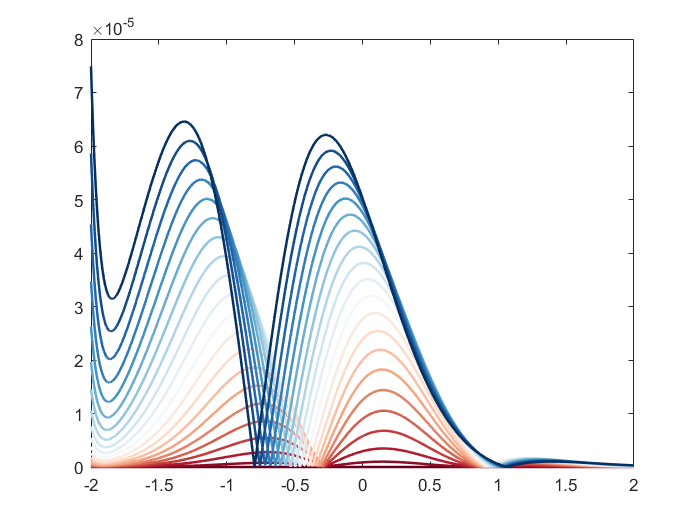} 
        \vspace{-2em}  
        \caption{$E_{mixed}^\rho$ of PartInv in Case 1}
    \end{subfigure}
    \hfill
    \begin{subfigure}{0.3\textwidth}
        \centering
        \includegraphics[width=\linewidth]{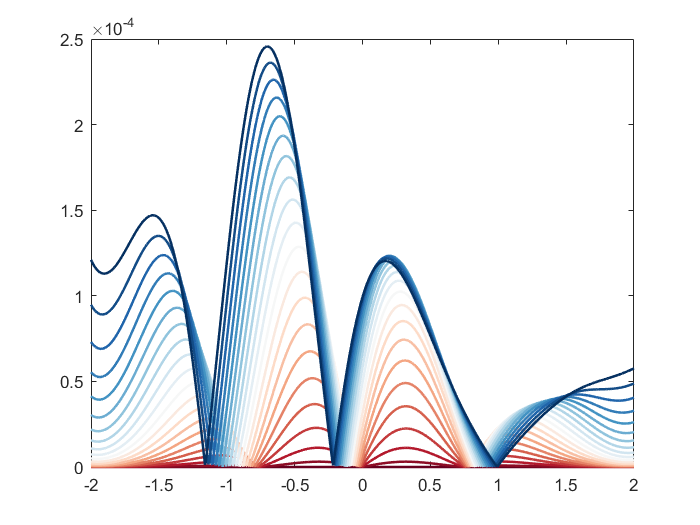}  
        \vspace{-2em}  
        \caption{$E_{mixed}^\rho$ of PartInv in Case 2}
    \end{subfigure}
    \hfill
    \begin{subfigure}{0.3\textwidth}
        \centering
        \includegraphics[width=\linewidth]{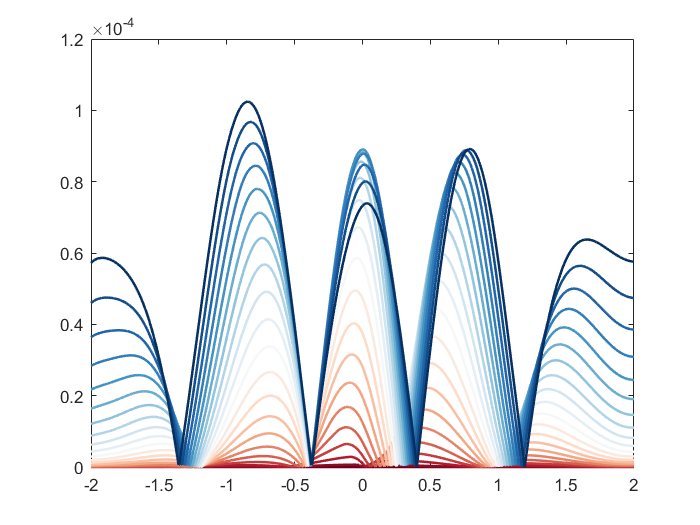}  
        \vspace{-2em}  
        \caption{$E_{mixed}^\rho$ of PartInv in Case 3}
    \end{subfigure}
    \hfill
    \hspace{-0.7cm}
    \begin{subfigure}{0.06\textwidth}
        \centering
        \includegraphics[width=\linewidth]{colorbar.png} 
        \vspace{-0.2em}
    \end{subfigure}

    \vspace{0.5em}

    \begin{subfigure}{0.3\textwidth}
        \centering
        \includegraphics[width=\linewidth]{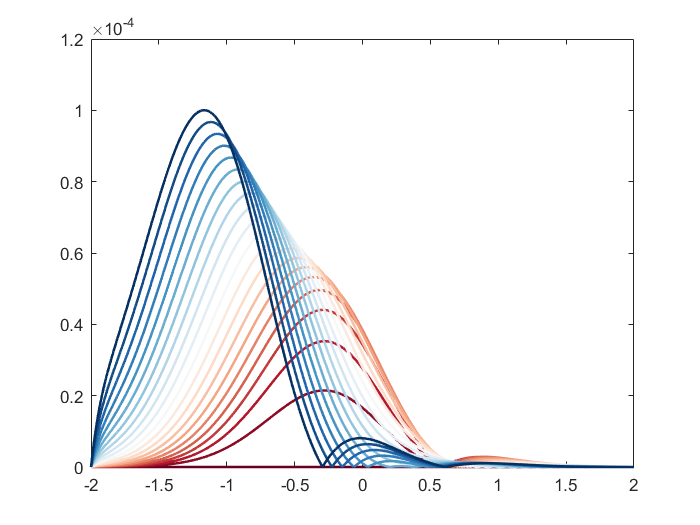} 
        \vspace{-2em} 
        \caption{$E_{mixed}^J$ of PartInv in Case 1}
    \end{subfigure}
    \hfill
    \begin{subfigure}{0.3\textwidth}
        \centering
        \includegraphics[width=\linewidth]{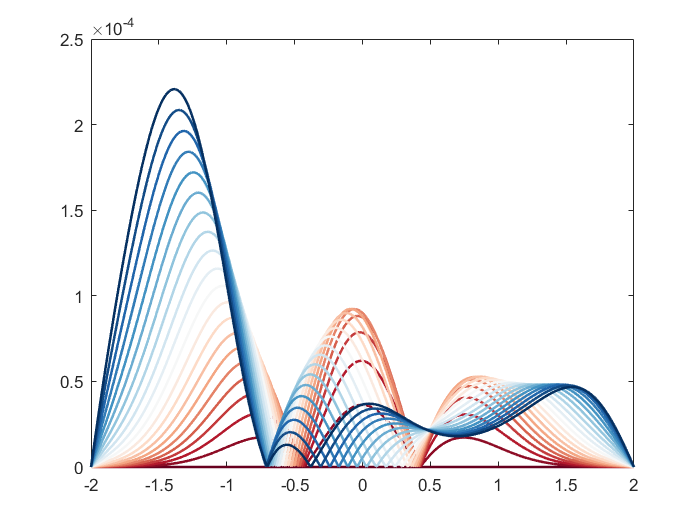}  
        \vspace{-2em}  
        \caption{$E_{mixed}^J$ of PartInv in Case 2}
    \end{subfigure}
    \hfill
    \begin{subfigure}{0.3\textwidth}
        \centering
        \includegraphics[width=\linewidth]{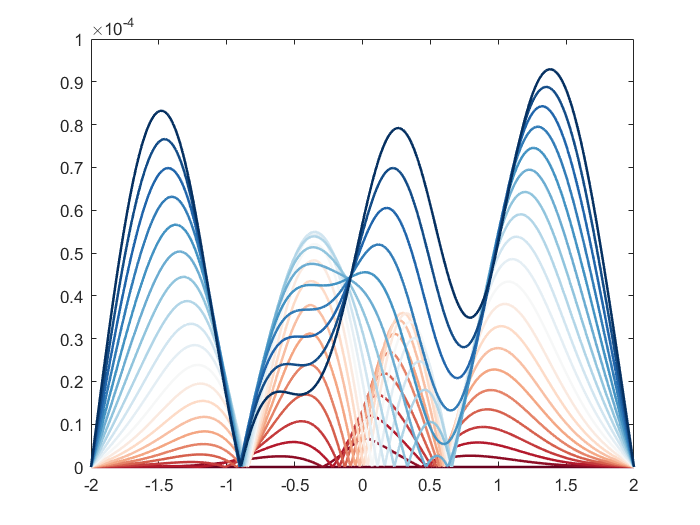}  
        \vspace{-2em}  
        \caption{$E_{mixed}^J$ of PartInv in Case 3}
    \end{subfigure}
    \hfill
    \hspace{-0.7cm}
    \begin{subfigure}{0.06\textwidth}
        \centering
        \includegraphics[width=\linewidth]{colorbar.png} 
        \vspace{-0.2em}
    \end{subfigure}

    \begin{minipage}{\textwidth}
    \centering
    \caption{Mixed errors of the forward solutions reconstructed by PartInv for all three test cases using noise-free data with $\epsilon=1$. The errors remain on the order of $10^{-4}$ or smaller, confirming that the recovered turning kernels reproduce the true forward solutions accurately. (a)--(c): Mixed errors of $\rho$ in Cases 1--3, respectively. (d)--(f): Corresponding mixed errors for $J$. Different color indicates different time.}

    \label{fig:compare_rhoJ_3cases}
    \end{minipage}
\end{figure}

\begin{table}[H]
\small
    \centering
    \begin{tabular}{cccccccccc}
        \hline
        \textbf{Case} & $T=1$ & \textbf{$T_0$: $c_1$} & \textbf{$T_0$: $c_2$} & $\mathcal{G}[\hat{T}_0]$ & $\|\bar{J}\|^2_{L_x^2}$ & \textbf{$T_1$: $c_1$} & \textbf{$T_1$: $c_2$} & $\mathcal{E}[\hat{T}_1]$ & $\|\bar{\rho}\|^2_{L_x^2}$\\
        \hline
        \textbf{Case 1} & \textbf{True solution} & \textbf{3.5} & \textbf{-0.7} &  &  & \textbf{3} & \textbf{-0.6} &  &  \\ 
          & $\epsilon=1$  & 3.5000 & -0.6998 & 2.3E-04 & 1.4E-06 & 2.9998 & -0.5995 & 1.7E-04 & 1.0E-06 \\
          & $\epsilon=0.1$  & 3.5000 & -0.6979 & 2.6E-03 & 6.3E-06 & 2.9992 & -0.5979 & 6.4E-04 & 6.8E-06 \\
          & $\epsilon=0.01$ & 3.5000 & -0.6978 & 2.6E-03 & 6.3E-06 & 2.9992 & -0.5978 & 6.4E-04 & 6.8E-06 \\
        \hline
        \textbf{Case 2} & \textbf{True solution} & \textbf{2.1213} & \textbf{3} &  &  & \textbf{1.4142} & \textbf{2} &  &  \\ 
          & $\epsilon=1$  & 2.1209 & 3.0004 & 2.0E-05 & 6.3E-06 & 1.4043 & 2.0109 & 3.2E-04 & 9.6E-06 \\
          & $\epsilon=0.1$  & 2.1195 & 3.0019 & 6.4E-05 & 1.0E-04 & 1.3482 & 2.0738 & 3.2E-03 & 6.8E-04 \\
          & $\epsilon=0.01$ & 2.1195 & 3.0019 & 6.7E-05 & 1.0E-04 & 1.3487 & 2.0733 & 3.3E-03 & 6.8E-04 \\
        \hline
        \textbf{Case 3} & \textbf{True solution} & \textbf{0.25} & \textbf{0.5} &  &  & \textbf{0.125} & \textbf{0.25} &  &  \\ 
          & $\epsilon=1$  & 0.2511 & 0.4968 & 1.1E-05 & 2.6E-06 & 0.1251 & 0.2495 & 4.2E-05 & 2.9E-06 \\
          & $\epsilon=0.1$  & 0.2574 & 0.4791 & 1.8E-04 & 1.2E-04 & 0.1261 & 0.2457 & 2.7E-04 & 7.9E-05 \\
          & $\epsilon=0.01$ & 0.2573 & 0.4791 & 1.8E-04 & 1.0E-04 & 0.1261 & 0.2456 & 2.7E-04 & 7.7E-05 \\
        \hline
    \end{tabular}
    \caption{Recovery results obtained by PartInv for the three test cases using noise-free data and different values of $\epsilon$. $T_0:c_1$ and $T_0:c_2$ ($T_1:c_1$ and $T_1:c_2$) denote the recovered coefficients $c_1$ and $c_2$ of $T_0$ ($T_1$); $\mathcal{G}[\hat{T}_0]$ ($\mathcal{E}[\hat{T}_1]$) is the value of loss functional for recovering $T_0$ ($T_1$); and $\|\bar{J}\|_{L_x^2}^2$ ($T_1$) is the squared error between the true and reconstructed fluxes (densities) at $T=1$.}
    \label{Case_123_multiscale}
\end{table}

\subsubsection{Robustness results}\label{sec:Robustness results}

To evaluate how well the proposed framework works in practice, this section tests the robustness of the PartInv algorithm against noise in the observations. We carry out two sets of tests: (i) a quantitative assessment using the true data with an additive $5\%$ noise level, and (ii) a statistical analysis of how errors change as the noise level varies.

Table \ref{Case_123_multiscale_noise} summarizes the recovery results for all three test cases under a fixed additive Gaussian noise level of $5\%$. Compared to the noise-free results in Table \ref{Case_123_multiscale}, adding noise naturally leads to a small increase in the loss functionals $\mathcal{G}[\hat{T}_0]$ and $\mathcal{E}[\hat{T}_1]$. Since the noise is on the order of $10^{-1}$, the forward solution errors $\|\bar{J}\|^2_{L_x^2}$ and $\|\bar{\rho}\|^2_{L_x^2}$ also stay around $O(10^{-1})$. However, several key observations confirm that the method is still robust:

\begin{itemize}
    \item \textbf{Accurate Parameter Recovery:} Even with noise, PartInv successfully recovers the true coefficients $c$ with good accuracy. For example, in Case 2 with $\epsilon=0.01$, the estimated coefficients for $T_0$ are $\hat{c}_1 = 2.1398$ and $\hat{c}_2 = 2.9629$, which are close to the true values $c_1 = 2.1213$ and $c_2 = 3$. This shows that the sparsity constraint helps regularize the inverse problem and reduces the impact of measurement noise.
    \item \textbf{Preserved Error-Loss Consistency:} For all three cases, the loss functionals $\mathcal{G}[\hat{T}_0]$ and $\mathcal{E}[\hat{T}_1]$, as well as the forward errors $\|\bar{\rho}\|^2_{L^2_x}$ and $\|\bar{J}\|^2_{L^2_x}$, do not change much as $\epsilon$ decreases from $1$ to $0.01$. This implies that even with noisy data, the forward solution error is still bounded by the optimization loss, confirming the theoretical results in Section \ref{sec:Stability Analysis}.
\end{itemize}

Moreover, Figure \ref{fig:robust_result_case123} shows the relative reconstruction error $E^i_{\text{reconst}} = \frac{\|T_i-\hat{T}_i\|_{L^2_x}}{\|T_i\|_{L^2_x}},\ i=1,2$, as the noise level increases from $0\%$ to $10\%$. Each data point shows the mean and standard deviation over 100 independent trials. The results reveal two robust features: (1) Even at a high noise level of $10\%$, the relative reconstruction error stays below $7\%$ for all cases. This confirms that the PartInv algorithm does not suffer from error blow-up due to noise. (2) The average relative errors for all cases are almost independent of the multiscale parameter $\epsilon$. Whether $\epsilon=1$ (kinetic regime) or $\epsilon=0.01$ (diffusion limit), the average recovery accuracy remains almost the same.

In summary, the robustness tests show that the proposed loss function and PartInv framework is well-suited for real-world applications where measurements are not perfect. It maintains accurate parameter recovery and reliable forward solutions across different scales and noise levels.

\begin{table}[H]
\small
    \centering
    \begin{tabular}{cccccccccc}
        \hline
        \textbf{Case} & $T=1$ & \textbf{$T_0$: $c_1$} & \textbf{$T_0$: $c_2$} & $\mathcal{G}[\hat{T}_0]$ & $\|\bar{J}\|^2_{L_x^2}$ & \textbf{$T_1$: $c_1$} & \textbf{$T_1$: $c_2$} & $\mathcal{E}[\hat{T}_1]$ & $\|\bar{\rho}\|^2_{L_x^2}$\\
        \hline
        \textbf{Case 1} & \textbf{True solution} & \textbf{3.5} & \textbf{-0.7} &  &  & \textbf{3} & \textbf{-0.6} &  &  \\
          & $\epsilon=1$  & 3.5000 & -0.6998 & 1.6E-02 & 1.5E-01 & 2.9998 & -0.5995 & 3.0E-04 & 2.3E-01 \\
          & $\epsilon=0.1$  & 3.5000 & -0.6979 & 2.5E-02 & 1.8E-01 & 2.9992 & -0.5979 & 6.8E-04 & 2.2E-01 \\
          & $\epsilon=0.01$ & 3.4762 & -0.7310 & 2.6E-02 & 1.6E-01 & 2.9992 & -0.5979 & 6.1E-04 & 2.0E-01 \\
        \hline
        \textbf{Case 2} & \textbf{True solution} & \textbf{2.1213} & \textbf{3} & & & \textbf{1.4142} & \textbf{2} & & \\ 
          & $\epsilon=1$  & 2.1248 & 3.0103 & 6.6E-04 & 8.4E-02 & 1.4909 & 1.9152 & 1.3E-03 & 2.1E-01 \\
          & $\epsilon=0.1$  & 2.1272 & 2.9874 & 4.2E-03 & 1.2E-01 & 1.3485 & 2.0735 & 3.2E-03 & 1.8E-01 \\
          & $\epsilon=0.01$ & 2.1398 & 2.9629 & 1.0E-02 & 1.2E-01 & 1.3465 & 2.0758 & 3.5E-03 & 1.7E-01 \\
        \hline
        \textbf{Case 3} & \textbf{True solution} & \textbf{0.25} & \textbf{0.5} & &  & \textbf{0.125} & \textbf{0.25} & & \\ 
          & $\epsilon=1$  & 0.2437 & 0.5238 & 4.3E-04 & 4.9E-02 &  0.1263 & 0.2411 & 1.5E-03 & 1.3E-01\\
          & $\epsilon=0.1$  & 0.2578 & 0.4672 & 8.5E-04 & 1.1E-01 & 0.1260 & 0.2459 & 2.9E-04 & 1.3E-01 \\
          & $\epsilon=0.01$ & 0.2522 & 0.4915 & 1.2E-04 & 1.2E-01 & 0.1262 & 0.2454 & 2.8E-04 & 1.3E-01 \\
        \hline
    \end{tabular}
    \caption{Recovery results obtained by PartInv for the three test cases using $5\%$ noisy data and different values of $\epsilon$. $T_0:c_1$ and $T_0:c_2$ ($T_1:c_1$ and $T_1:c_2$) denote the recovered coefficients $c_1$ and $c_2$ of $T_0$ ($T_1$); $\mathcal{G}[\hat{T}_0]$ ($\mathcal{E}[\hat{T}_1]$) is the value of loss functional for recovering $T_0$ ($T_1$); and $\|\bar{J}\|_{L_x^2}^2$ ($T_1$) is the squared error between the true and reconstructed fluxes (densities) at $T=1$.}

    \label{Case_123_multiscale_noise}
\end{table}

\begin{figure}[!htbp]  
    \centering

    \begin{subfigure}{0.32\textwidth}
        \centering
        \includegraphics[width=\linewidth]{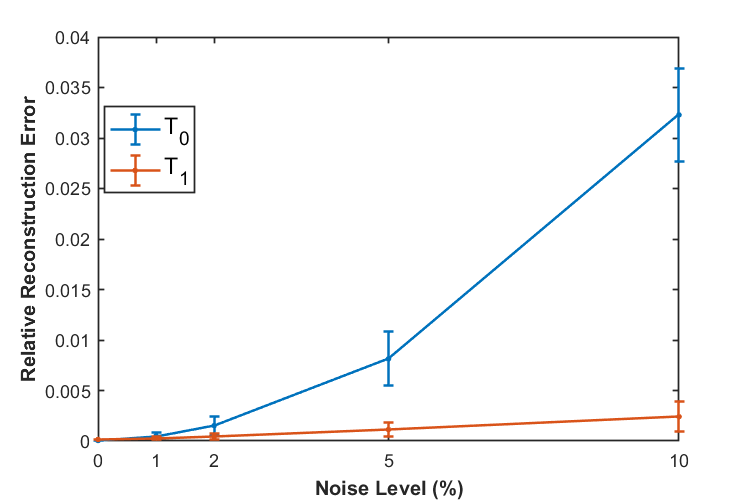}
        \vspace{-2em}  
        \caption{Case 1, $\epsilon = 1$}
    \end{subfigure}
    \hfill
    \begin{subfigure}{0.32\textwidth}
        \centering
        \includegraphics[width=\linewidth]{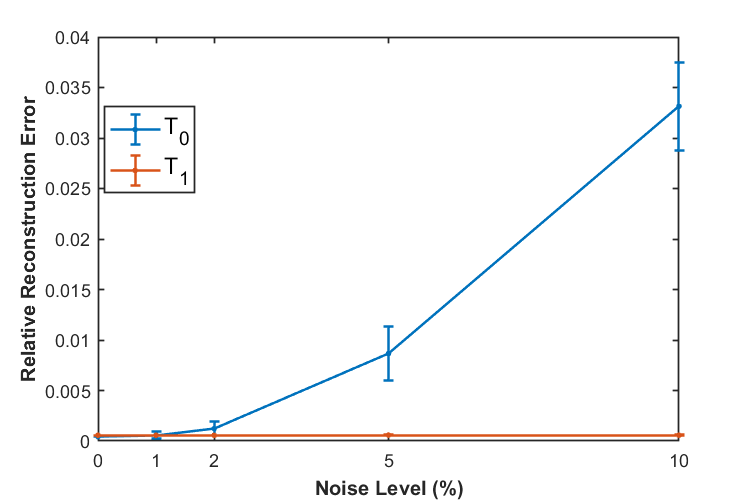}  
        \vspace{-2em}  
        \caption{Case 1, $\epsilon = 0.1$}
    \end{subfigure}
    \hfill
    \begin{subfigure}{0.32\textwidth}
        \centering
        \includegraphics[width=\linewidth]{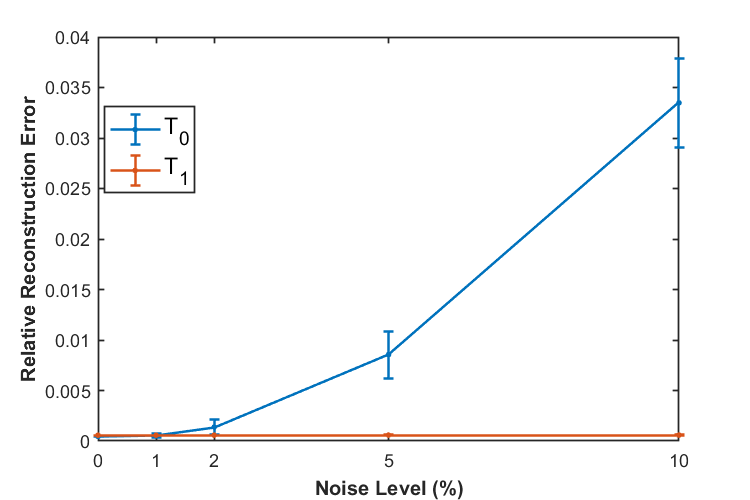}  
        \vspace{-2em}  
        \caption{Case 1, $\epsilon = 0.01$}
    \end{subfigure}

    \vspace{0.5em}

    \begin{subfigure}{0.32\textwidth}
        \centering
        \includegraphics[width=\linewidth]{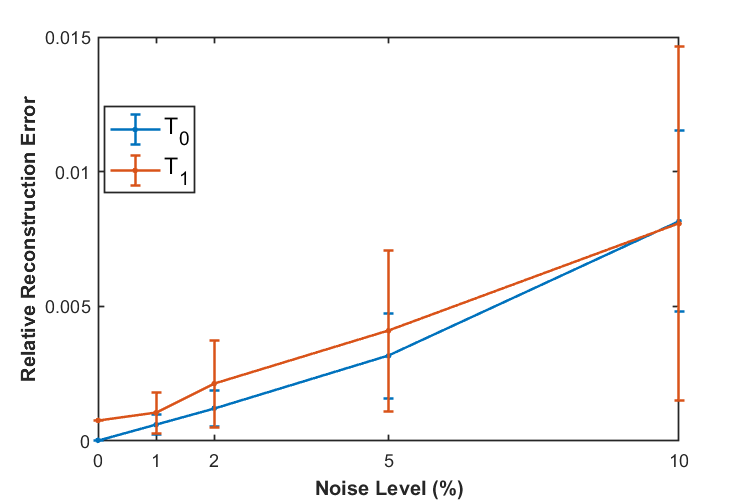}  
        \vspace{-2em}  
        \caption{Case 2, $\epsilon = 1$}
    \end{subfigure}
    \hfill
    \begin{subfigure}{0.32\textwidth}
        \centering
        \includegraphics[width=\linewidth]{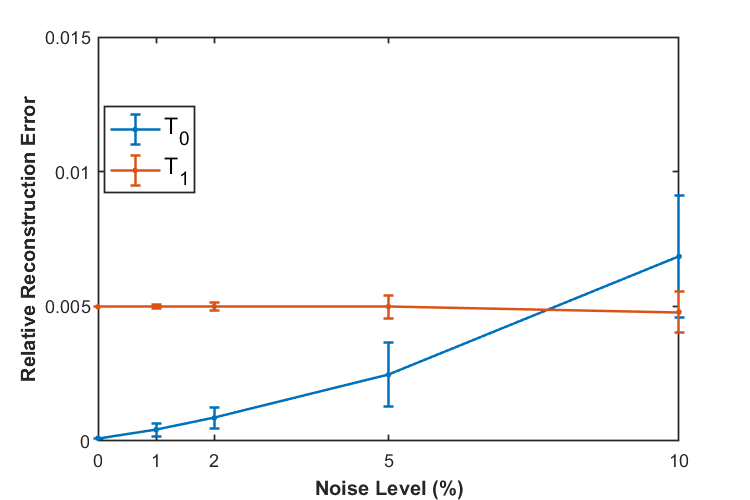}  
        \vspace{-2em}  
        \caption{Case 2, $\epsilon = 0.1$}
    \end{subfigure}
    \hfill
    \begin{subfigure}{0.32\textwidth}
        \centering
        \includegraphics[width=\linewidth]{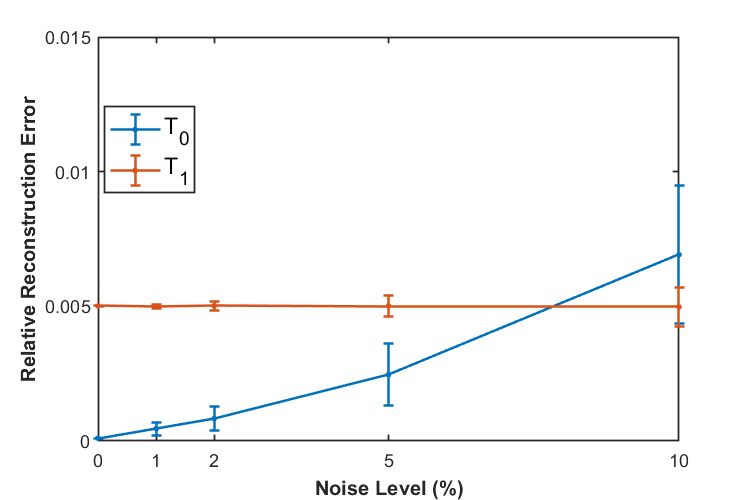}  
        \vspace{-2em} 
        \caption{Case 2, $\epsilon = 0.01$}
    \end{subfigure}

    \vspace{0.5em}

    \begin{subfigure}{0.32\textwidth}
        \centering
        \includegraphics[width=\linewidth]{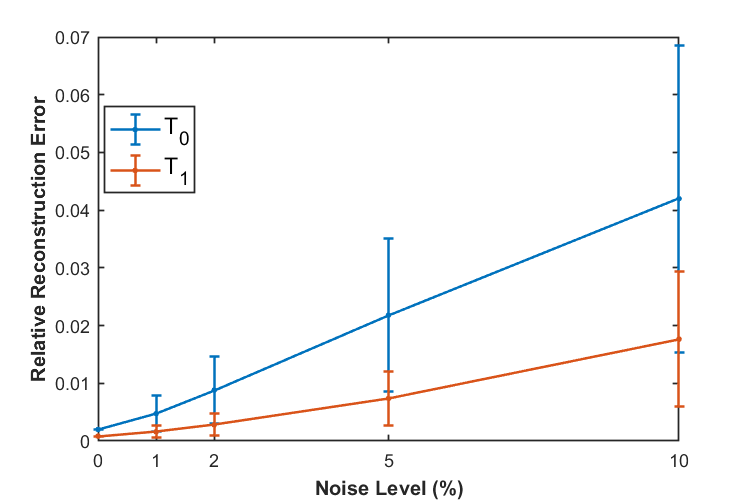}  
        \vspace{-2em}  
        \caption{Case 3, $\epsilon = 1$}
    \end{subfigure}
    \hfill
    \begin{subfigure}{0.32\textwidth}
        \centering
        \includegraphics[width=\linewidth]{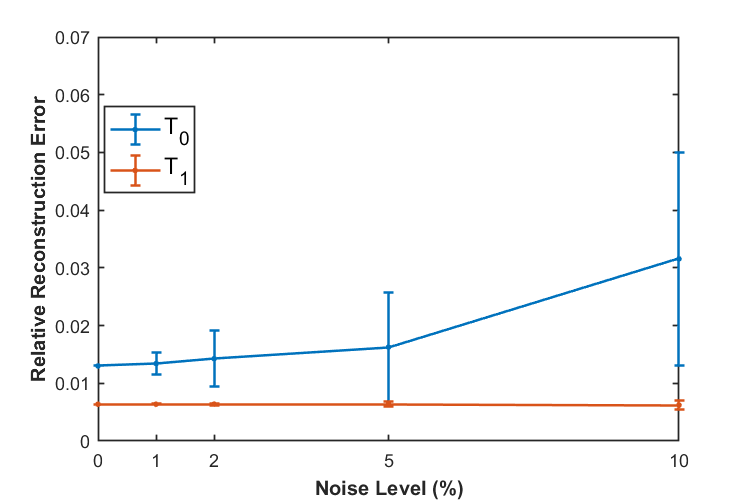}  
        \vspace{-2em}  
        \caption{Case 3, $\epsilon = 0.1$}
    \end{subfigure}
    \hfill
    \begin{subfigure}{0.32\textwidth}
        \centering
        \includegraphics[width=\linewidth]{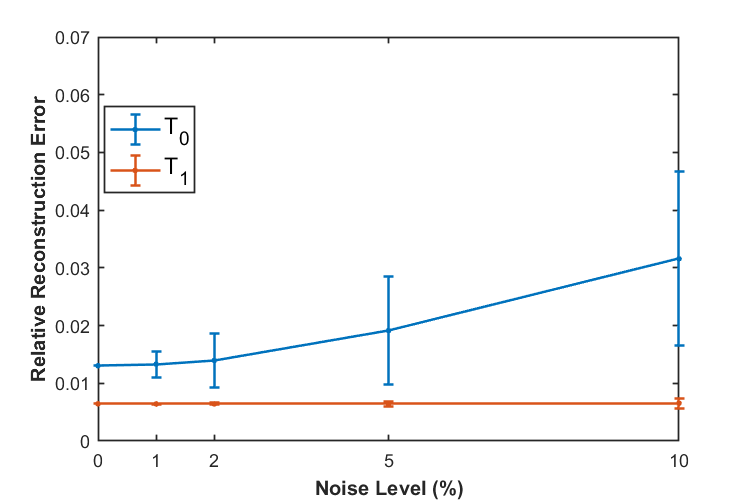}  
        \vspace{-2em} 
        \caption{Case 3, $\epsilon = 0.01$}
    \end{subfigure}
    
    \begin{minipage}{\textwidth}
    \centering
    \caption{Relative reconstruction errors of $T_0$ and $T_1$ obtained by PartInv under different noise levels. Each point and error bar represent the mean and standard deviation over 100 independent trials, respectively. The reconstruction errors remain below $7\%$ and show no significant deterioration as $\epsilon$ decreases. (a)--(c): Case 1 with $\epsilon=1$, $0.1$, and $0.01$, respectively. (d)--(f): Case 2 with $\epsilon=1$, $0.1$, and $0.01$, respectively. (g)--(i): Case 3 with $\epsilon=1$, $0.1$, and $0.01$, respectively.}
    \label{fig:robust_result_case123}
    \end{minipage}

\end{figure}

\subsubsection{Mixed case recovery}

While Sections \ref{sec:Stability analysis validation} and \ref{sec:Robustness results} have shown that the method is stable with respect to the multiscale parameter $\epsilon$ and data noise, real biological environments often have spatial variation in the environments, thus $T_0(x)$ and $T_1(x)$ may vary several orders of magnitude in different spatial regions. To mimic this physical scenario, we design Case 4, where the true $T_0(x)$ and $T_1(x)$ range from $O(10^{-9})$ to $O(1)$:

\noindent \textbf{Case 4:}
\begin{equation*}
        \begin{cases}
            T_0(x) = \frac{2}{\sqrt{0.2\pi}} \exp(-\frac{x^2}{0.2})+\frac{2}{\sqrt{0.16\pi}} \exp(-\frac{x^2}{0.16}),\\
            T_1(x) = \frac{1}{\sqrt{0.2\pi}} \exp(-\frac{x^2}{0.2})+\frac{1}{\sqrt{0.16\pi}} \exp(-\frac{x^2}{0.16}),
        \end{cases}
        \ \text{with }\
        \Psi_k(x) = \frac{\exp(-50 x^2/2k)}{\sqrt{2\pi}},\ k=1, \dots, 10.
    \end{equation*}

The true $T_0(x)$, $T_1(x)$, and the basis functions are shown in Figure \ref{fig:T0T1_case4}. The true coefficient vector $\mathbf{c} = (c_1, c_2, c_3, \dots)^\intercal$ for Case 4 is listed in Table \ref{true_solution_case4}.

\begin{figure}[!htbp]  
    \centering 
    
    \begin{subfigure}{0.48\textwidth}
        \centering
        \includegraphics[width=\linewidth]{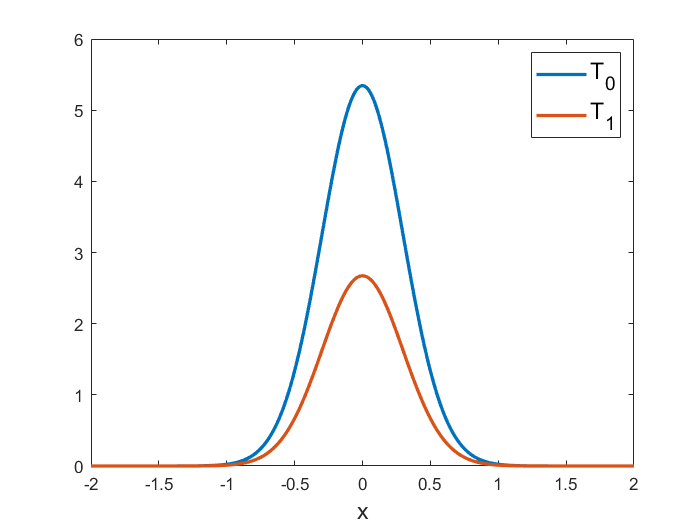}  
        \vspace{-2em}  
        \caption{$T_0\ \&\ T_1$  }
    \end{subfigure}
    \hfill
    \begin{subfigure}{0.48\textwidth}
        \centering
        \includegraphics[width=\linewidth]{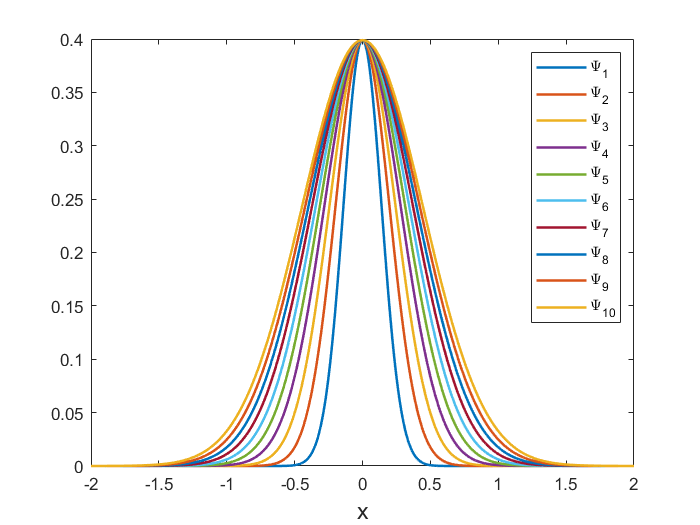}  
        \vspace{-2em}  
        \caption{Basis functions}
    \end{subfigure}
    
    \begin{minipage}{\textwidth}
    \centering
    \caption{True turning-kernel and basis functions used in Case 4, where the functional values in the kernels vary from $O(10^{-9})$ to $O(1)$ across the spatial domain. (a) True $T_0(x)$ and $T_1(x)$, each formed by a linear combination of two Gaussian profiles with different widths. (b) Ten symmetric Gaussian basis functions with different widths used to represent the turning kernels.}
    \label{fig:T0T1_case4}
    \end{minipage}
    
\end{figure}

\begin{table}[H]
    \centering
    \begin{tabular}{c|cc|cc}
        \hline
         & \multicolumn{2}{c|}{$T_0$} & \multicolumn{2}{c}{$T_1$}\\
         \hline
         & $c_4$ & $c_5$ & $c_4$ & $c_5$  \\
         \hline
        \textbf{Case 4}  & 6.3246 & 7.0711 & 3.1623 & 3.5355 \\
        \hline
    \end{tabular}
    \caption{Coefficient vectors $\boldsymbol{c}=(c_1,\ldots,c_{10})$ for $T_0$ and $T_1$ in Case 4. For both turning-kernel components, only $c_4$ and $c_5$ are nonzero, while the remaining coefficients are set to zero.}
    \label{true_solution_case4}
\end{table}
The recovery results are summarized in Table \ref{Case_4_multiscale_noise} and Figure \ref{fig:T0T1_compare_case4}. More specifically, Table \ref{Case_4_multiscale_noise} shows the recovery performance under $5\%$ additive Gaussian noise for $\epsilon = 1,\ 0.1,\ 0.01$. Even with noise and the large variation in the magnitudes of $T_0(x)$ and $T_1(x)$, PartInv still recovers the true coefficients $c$ for both $T_0(x)$ and $T_1(x)$ accurately. Figure 12 shows the reconstructed $T_0(x)$ and $T_1(x)$ for $\epsilon = 0.01$, where the recovered profiles are very close to the true functions. The results for $\epsilon = 1$ and $\epsilon = 0.1$ are similar and we omit the details.

Although $\epsilon$ is fixed, spatial variation in $T_0(x)$ induces a strong multiscale effect. The turning rates $\lambda_{\pm} = (T_0 \pm \epsilon T_1)/\epsilon^2$ yield mean run times $\tau_{\pm} = \epsilon^2/(T_0 \pm \epsilon T_1)$. For $\epsilon=0.01$ and $\epsilon |T_1| \ll T_0$, varying $T_0$ from $\mathcal{O}(1)$ to $\mathcal{O}(10^{-9})$ changes $\tau_{\pm}$ from $\mathcal{O}(10^{-4})$ to $\mathcal{O}(10^{5})$ (and the turning frequency from $\mathcal{O}(10^{4})$ to $\mathcal{O}(10^{-5})$). Hence, this fixed-$\epsilon$ kernel alone creates a heterogeneous environment with both rapid-tumbling and ballistic regions, without requiring a spatially varying $\epsilon(x)$.

\begin{table}[H]
\small
    \centering
    \begin{tabular}{cccccccccc}
        \hline
        \textbf{Case} & $T=1$ & \textbf{$T_0$: $c_4$} & \textbf{$T_0$: $c_5$} & $\mathcal{G}[\hat{T}_0]$ & $\|\bar{J}\|^2_{L_x^2}$ & \textbf{$T_1$: $c_4$} & \textbf{$T_1$: $c_5$} & $\mathcal{E}[\hat{T}_1]$ & $\|\bar{\rho}\|^2_{L_x^2}$\\
        \hline
        \textbf{Case 4} & \textbf{True solution} & \textbf{6.3246} & \textbf{7.0711} &  &  & \textbf{3.1623} & \textbf{3.5355} &  &  \\
          & $\epsilon=1$  & 6.2725 & 7.1063 & 3.5E-03 & 7.5E-02 & 3.1915 & 3.5074 & 4.0E-03 & 1.5E-01 \\
          & $\epsilon=0.1$  & 6.0064 & 7.2711 & 1.3E-02 & 2.5E-01 & 3.1604 & 3.5370 & 1.6E-04 & 1.5E-01 \\
          & $\epsilon=0.01$ & 6.2603 & 7.0208 & 1.3E-02 & 3.7E-01 & 3.1596 & 3.5378 & 1.4E-04 & 1.5E-01 \\
        \hline
    \end{tabular}
    \caption{Recovery results obtained by PartInv for Case 4 using $5\%$ noisy data and different values of $\epsilon$. $T_0$: $c_4$ and $T_0$: $c_5$ ($T_1$: $c_4$ and $T_1$: $c_5$) denote the recovered coefficients $c_4$ and $c_5$ of $T_0$ ($T_1$), $\mathcal{G}[\hat{T}_0]$ ($\mathcal{E}[\hat{T}_1]$) is the value of loss functional for recovering $T_0$ ($T_1$), and $\|\bar{J}\|_{L_x^2}^2$ ($\|\bar{\rho}\|_{L_x^2}^2$) is the squared error between the true and reconstructed fluxes (densities) at $T=1$.}
    \label{Case_4_multiscale_noise}
\end{table}

\begin{figure}[!htbp]  
    \centering 
    
    \begin{subfigure}{0.48\textwidth}
        \centering
        \includegraphics[width=\linewidth]{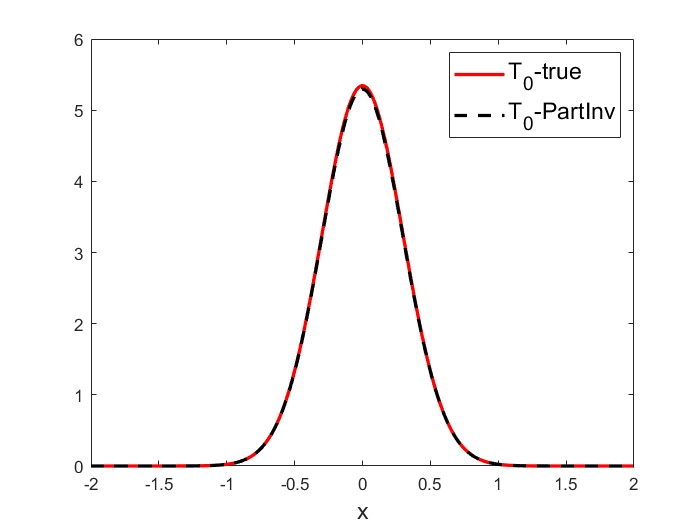}  
        \vspace{-2em}  
        \caption{Comparison of $T_0$  }
    \end{subfigure}
    \hfill
    \begin{subfigure}{0.48\textwidth}
        \centering
        \includegraphics[width=\linewidth]{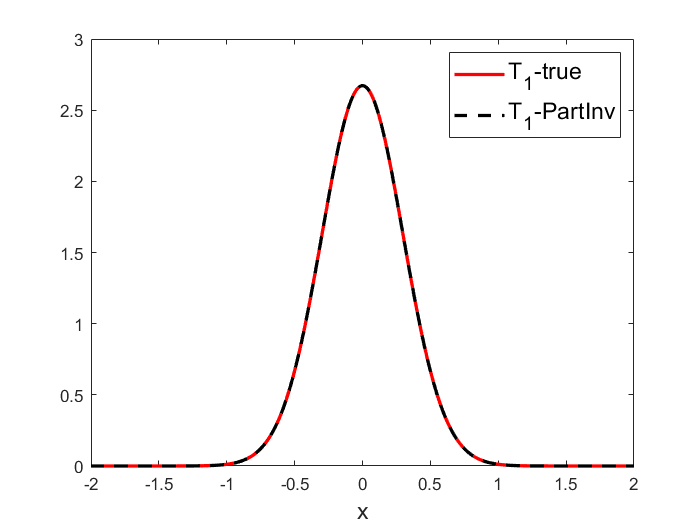}  
        \vspace{-2em}  
        \caption{Comparison of $T_1$}
    \end{subfigure}
    
    \begin{minipage}{\textwidth}
    \centering
    \caption{Recovery results for the turning-kernel components in Case 4 using PartInv with $5\%$ noise and $\epsilon=0.01$. The reconstructed profiles closely match the true functions despite the large spatial variation in their magnitudes. (a): Comparison between the true and reconstructed $T_0(x)$. (b): Comparison between the true and reconstructed $T_1(x)$.}
    \label{fig:T0T1_compare_case4}
    \end{minipage}
    
\end{figure}

\section{Conclusion and Discussion}
\label{sec:Discussion}
This work addresses the inverse problem of identifying parameters in multiscale chemotaxis models. Focusing on a one-dimensional two-flux kinetic model, we develop a variational framework to recover the turning kernel parameters \(T_0(x)\) and \(T_1(x)\) from the kinetic solution. The main contributions are: 1) A novel variational formulation for recovering \(T_0\) and \(T_1\) from solution data; 2)  A formal or conditional uniform stability argument in the whole spatial domain $\mathbb{R}$  showing that the error in density and flux is uniformly controlled by the loss functionals; 3) An effective sparsity-promoting algorithm (PartInv) that is accurate, robust, and efficient, even under noise.

The main theoretical contribution is a conditional scale-uniform loss-to-solution stability argument: under the stated regularity and data-informativeness assumptions, kernel estimates with small variational losses generate density and flux evolutions close to those of the true model across the kinetic-to-diffusive transition. This advantage is crucial when working with real experimental data, because the run-and-tumble process of  bacteria in different spatial domains can have different mean free paths and therefore different values of \(\epsilon\), depending on the external signal \(S(x)\). At the same time, experimental data is always noisy. However, how noisy data and the multiscale parameter affect the recovery of chemotaxis model parameters has rarely been studied.

In our current work, the loss function is designed to recover one parameter when the other is known: specifically, we can recover \(T_1\) given \(T_0\), or recover \(T_0\) given \(T_1\). If one needs to recover \(T_0\) and \(T_1\) simultaneously, a natural approach is to use alternating optimization — that is, iteratively fixing one parameter while optimizing the other. However, this strategy requires further investigation, as it may converge slowly or become trapped in local minima.  Moreover, the identifiability of $T_0$ and $T_1$ is inherently weighted by the observed solution fields. In regions where $J$ is close to zero, the loss functional provides little information about $T_0$, while $T_1$ is weakly constrained where $\rho$ is small. Therefore, reliable reconstruction may require sufficiently informative initial data, multiple experiments, or additional regularization, while we provide here only the loss functions that can handle the multiscale property of the model, and the experimental design is outside the scope of the current work. Moreover,  only the two-flux model is considered here. Extending this framework to higher dimensions would be an interesting direction for future work.

\subsubsection*{Acknowledgments}
All authors were partitially supported by a bilateral project of the Royal Society with the National Chinese Science Foundation NSFC 12411530067. 
  JAC was supported by the Advanced Grant Nonlocal-CPD (Nonlocal PDEs for Complex Particle Dynamics: Phase Transitions, Patterns and Synchronization) of the European Research Council Executive Agency (ERC) under the European Union's Horizon 2020 research and innovation programme (grant agreement No. 883363) and the EPSRC grant number EP/V051121/1.

\section*{Appendix}
\appendix
\section{Lemmas}\label{Appendix_A}

Throughout, let $J_0$ be the flux density corresponding to the limiting system \eqref{limiteqtion} and assume that all quantities below are well-defined. Recall that
\[
\mathcal{G}_0[\hat{T}_0]
=
\frac{1}{T}
\int_0^T\int_{\mathbb{R}}
|\hat{T}_0(x)-T_0(x)|^2J_0^2(x,t)\,dxdt.
\]

\begin{assumption}[Uniform boundedness and positivity]\label{ass:T_0}
There exist constants $0<T_{0,\min}\leq T_{0,\max}<\infty$ such that
\[
T_{0,\min}
\leq T_0(x),\hat{T}_0(x)
\leq T_{0,\max}
\qquad
\text{for almost every }x\in\mathbb{R}.
\]
Consequently,
\[
\left\|\frac{1}{T_0}\right\|_{L^\infty(\mathbb{R})},
\quad
\left\|\frac{1}{\hat{T}_0}\right\|_{L^\infty(\mathbb{R})}
\leq
\frac{1}{T_{0,\min}}.
\]
\end{assumption}

\begin{assumption}[Finite-dimensional admissible space]\label{ass:basis}
$T_0(x)$ and $\hat{T}_0$ are represented using the same prescribed basis functions:
\[
T_0(x)=\sum_{i=1}^{n}c_i\Psi_i(x),
\qquad
\hat{T}_0(x)=\sum_{i=1}^{n}\hat{c}_i\Psi_i(x),
\]
where $\Psi_i\in L^\infty(\mathbb{R})$. Consequently,
\[
T_0(x)-\hat{T}_0(x)
=
\sum_{i=1}^{n}(c_i-\hat{c}_i)\Psi_i(x).
\]
\end{assumption}

\begin{assumption}[Boundedness of the derivative-weighted Gram matrix]\label{ass:derivative_gram}
The limiting flux satisfies
\[
J_0,\partial_tJ_0
\in
L^2\bigl(\mathbb{R}\times(0,T)\bigr),
\]
and
\[
M_J
:=
\int_0^T\int_{\mathbb{R}}
\left(\sum_{i=1}^{n}|\Psi_i(x)|^2\right)
|\partial_tJ_0(x,t)|^2\,dxdt
<\infty.
\]
\end{assumption}
To interpret this condition, define the derivative-weighted Gram matrix
\[
(A_J)_{ij}
=
\int_0^T\int_{\mathbb{R}}
\Psi_i(x)\Psi_j(x)
|\partial_tJ_0(x,t)|^2\,dxdt.
\]
The matrix $A_J$ is symmetric and positive semidefinite, and
\[
\operatorname{trace}(A_J)
=
\sum_{i=1}^{n}(A_J)_{ii}
=
M_J.
\]
Consequently,
\[
\lambda_{\max}(A_J)\leq M_J.
\]
Thus, $M_J$ provides an upper bound for the derivative-weighted Gram matrix and measures the overall magnitude of the temporal variation of $J_0$ observed through the prescribed basis functions.

\begin{assumption}[Positive definiteness of the data-weighted Gram matrix]\label{ass:data_gram}
There exists a constant $\kappa>0$ such that, for every collection of real numbers $d_1,\ldots,d_n$,
\[
\int_0^T\int_{\mathbb{R}}
\left|
\sum_{i=1}^{n}d_i\Psi_i(x)
\right|^2
J_0^2(x,t)\,dxdt
\geq
\kappa\sum_{i=1}^{n}|d_i|^2.
\]
\end{assumption}

Indeed, define the data-weighted Gram matrix
\[
(B_J)_{ij}
=
\int_0^T\int_{\mathbb{R}}
\Psi_i(x)\Psi_j(x)J_0^2(x,t)\,dxdt.
\]
Then
\[
\sum_{i,j=1}^{n}d_i(B_J)_{ij}d_j
=
\int_0^T\int_{\mathbb{R}}
\left|
\sum_{i=1}^{n}d_i\Psi_i(x)
\right|^2
J_0^2(x,t)\,dxdt.
\]
Therefore, this assumption is equivalent to requiring that $B_J$ be positive definite, with
\[
\lambda_{\min}(B_J)\geq\kappa>0.
\]
It ensures that the observed flux $J_0$ contains sufficient information to distinguish every direction in the prescribed finite-dimensional admissible space.

\begin{rmk}
Assumption~\ref{ass:derivative_gram} gives an upper bound for the Gram matrix weighted by $|\partial_tJ_0|^2$, whereas Assumption~\ref{ass:data_gram} gives a positive lower bound for the Gram matrix weighted by $J_0^2$. Their combination allows the weighted quantity involving $\partial_tJ_0$ to be controlled by the loss functional involving $J_0$.
\end{rmk}

\begin{lem}\label{lem:gram_upper}
Under Assumptions~\ref{ass:basis} and \ref{ass:derivative_gram},
\[
\begin{aligned}
\int_0^T\int_{\mathbb{R}}
|T_0(x)-\hat{T}_0(x)|^2
|\partial_tJ_0(x,t)|^2\,dxdt
\leq
M_J\sum_{i=1}^{n}|c_i-\hat{c}_i|^2.
\end{aligned}
\]
\end{lem}

\begin{proof}
By Assumption~\ref{ass:basis},
\[
T_0(x)-\hat{T}_0(x)
=
\sum_{i=1}^{n}(c_i-\hat{c}_i)\Psi_i(x).
\]
The Cauchy--Schwarz inequality gives
\[
\begin{aligned}
|T_0(x)-\hat{T}_0(x)|^2
=
\left|
\sum_{i=1}^{n}(c_i-\hat{c}_i)\Psi_i(x)
\right|^2
\leq
\left(\sum_{i=1}^{n}|c_i-\hat{c}_i|^2\right)
\left(\sum_{i=1}^{n}|\Psi_i(x)|^2\right).
\end{aligned}
\]
Multiplying both sides by $|\partial_tJ_0(x,t)|^2$ and integrating over $\mathbb{R}\times(0,T)$, we obtain
\[
\begin{aligned}
\int_0^T\int_{\mathbb{R}}
|T_0-\hat{T}_0|^2|\partial_tJ_0|^2\,dxdt
\leq
\left(\sum_{i=1}^{n}|c_i-\hat{c}_i|^2\right)
\int_0^T\int_{\mathbb{R}}
\left(\sum_{i=1}^{n}|\Psi_i|^2\right)
|\partial_tJ_0|^2\,dxdt.
\end{aligned}
\]
The conclusion follows from the definition of $M_J$.
\end{proof}

\begin{lem}\label{main_lemma}
Under Assumptions~\ref{ass:basis}--\ref{ass:data_gram},
\begin{equation}\label{eq:main}
\begin{aligned}
\int_0^T\int_{\mathbb{R}}
|T_0(x)-\hat{T}_0(x)|^2
|\partial_tJ_0(x,t)|^2\,dxdt
\leq
\frac{M_J}{\kappa}
\int_0^T\int_{\mathbb{R}}
|T_0(x)-\hat{T}_0(x)|^2
J_0^2(x,t)\,dxdt.
\end{aligned}
\end{equation}
Equivalently,
\[
\int_0^T
\left\|
(T_0-\hat{T}_0)\partial_tJ_0(\cdot,t)
\right\|_{L_x^2}^2\,dt
\leq
\frac{M_J}{\kappa}
\int_0^T
\left\|
(T_0-\hat{T}_0)J_0(\cdot,t)
\right\|_{L_x^2}^2\,dt.
\]
\end{lem}

\begin{proof}
Applying Assumption~\ref{ass:data_gram} with
\[
d_i=c_i-\hat{c}_i,
\qquad i=1,\ldots,n,
\]
gives
\[
\begin{aligned}
\int_0^T\int_{\mathbb{R}}
|T_0(x)-\hat{T}_0(x)|^2J_0^2(x,t)\,dxdt
\geq
\kappa\sum_{i=1}^{n}|c_i-\hat{c}_i|^2.
\end{aligned}
\]
Therefore,
\[
\sum_{i=1}^{n}|c_i-\hat{c}_i|^2
\leq
\frac{1}{\kappa}
\int_0^T\int_{\mathbb{R}}
|T_0-\hat{T}_0|^2J_0^2\,dxdt.
\]
Combining this estimate with Lemma~\ref{lem:gram_upper}, we obtain
\[
\begin{aligned}
\int_0^T\int_{\mathbb{R}}
|T_0-\hat{T}_0|^2|\partial_tJ_0|^2\,dxdt
\leq
\frac{M_J}{\kappa}
\int_0^T\int_{\mathbb{R}}
|T_0-\hat{T}_0|^2J_0^2\,dxdt,
\end{aligned}
\]
which proves \eqref{eq:main}.
\end{proof}

\begin{lem}\label{lemma_forJ}
Under the assumptions of \textbf{Theorem 2 (\romannumeral2)} and Assumptions~\ref{ass:T_0}--\ref{ass:data_gram}, if $\hat{T}_1=T_1$, then
\[
\begin{aligned}
\|J_0(x,T)-\hat{J}_0(x,T)\|_{L_x^2}^2
\leq
C_4^0
\|J_0(x,0)-\hat{J}_0(x,0)\|_{L_x^2}^2
+
C_6^0\mathcal{G}_0[\hat{T}_0],
\end{aligned}
\]
where $C_4^0$ and $C_6^0$ are independent of $\epsilon$.
\end{lem}

\begin{proof}
Taking the derivative with respect to $t$ on both sides of Eq.~\eqref{limiteqtion_2}, we obtain
\[
v\partial_t\partial_x\rho_0
=
-2T_0(x)\partial_tJ_0
-2T_1(x)\partial_t\rho_0.
\]
By Eq.~\eqref{limiteqtion_1}, we have
\[
\partial_t\rho_0=-v\partial_xJ_0,
\]
then,
\[
-v^2\partial_x^2J_0
=
-2T_0(x)\partial_tJ_0
+
2vT_1(x)\partial_xJ_0.
\]
Equivalently,
\[
T_0(x)\partial_tJ_0
=
vT_1(x)\partial_xJ_0
+
\frac{v^2}{2}\partial_x^2J_0.
\]

Since $\hat{T}_1=T_1$, it follows that
\begin{align}
&
\int_0^T\int_{\mathbb{R}}
\left(
\left(
\frac{vT_1}{T_0}
-
\frac{v\hat{T}_1}{\hat{T}_0}
\right)\partial_xJ_0
+
\left(
\frac{v^2}{2T_0}
-
\frac{v^2}{2\hat{T}_0}
\right)\partial_x^2J_0
\right)^2
dxdt
\nonumber\\
=&
\int_0^T\int_{\mathbb{R}}
\left(
\left(
\frac{1}{T_0}
-
\frac{1}{\hat{T}_0}
\right)
\left(
vT_1\partial_xJ_0
+
\frac{v^2}{2}\partial_x^2J_0
\right)
\right)^2
dxdt
\nonumber\\
=&
\int_0^T\int_{\mathbb{R}}
\left(
\frac{\hat{T}_0-T_0}{\hat{T}_0}
\partial_tJ_0
\right)^2
dxdt
\nonumber\\
\leq&
\frac{1}{T_{0,\min}^2}
\int_0^T\int_{\mathbb{R}}
|\hat{T}_0-T_0|^2
|\partial_tJ_0|^2\,dxdt.
\label{eq:forcing_estimate}
\end{align}

Then, by \textup{\textbf{Theorem 2 (\romannumeral2)}} and \eqref{eq:forcing_estimate},
\begin{align}
\|J_0(x,T)-\hat{J}_0(x,T)\|_{L_x^2}^2
\leq{}&
C_4^0
\|J_0(x,0)-\hat{J}_0(x,0)\|_{L_x^2}^2 
+
\frac{C_5^0}{T_{0,\min}^2}
\int_0^T\int_{\mathbb{R}}
|\hat{T}_0-T_0|^2
|\partial_tJ_0|^2\,dxdt.
\label{eq:J0_before_gram}
\end{align}

Applying Lemma~\ref{main_lemma} to the last term gives
\[
\begin{aligned}
\int_0^T\int_{\mathbb{R}}
|\hat{T}_0-T_0|^2|\partial_tJ_0|^2\,dxdt
\leq
\frac{M_J}{\kappa}
\int_0^T\int_{\mathbb{R}}
|\hat{T}_0-T_0|^2J_0^2\,dxdt
=
\frac{TM_J}{\kappa}\mathcal{G}_0[\hat{T}_0].
\end{aligned}
\]
Substituting this estimate into \eqref{eq:J0_before_gram}, we obtain
\[
\begin{aligned}
\|J_0(x,T)-\hat{J}_0(x,T)\|_{L_x^2}^2
\leq
C_4^0
\|J_0(x,0)-\hat{J}_0(x,0)\|_{L_x^2}^2
+
C_6^0\mathcal{G}_0[\hat{T}_0],
\end{aligned}
\]
where
\[
C_6^0
=
\frac{C_5^0TM_J}{\kappa T_{0,\min}^2}.
\]
The constants $M_J$ and $\kappa$ depend only on the limiting flux $J_0$ and the prescribed basis functions. Therefore, they are independent of $\epsilon$, and hence $C_6^0$ is also independent of $\epsilon$.
\end{proof}

\section{AP scheme}\label{Appendix_B}
In this section, we provide the details of the numerical discretization for \eqref{twoflux_1},  which is based on an operator splitting method. Following \cite{carrillo2013asymptotic}, the system in \eqref{twoflux_1} is solved by first applying an implicit Euler discretization to the source part for one time step:
\begin{equation}\label{source_part}
    \begin{gathered}
       \partial_t\rho = 0,\\
      \partial_tJ = -\frac{2}{\epsilon^2}(T_0J+T_1 \rho) + v\left(1 - \frac{1}{\epsilon^2} \right)\partial_x\rho,
    \end{gathered}
\end{equation}
and then solving the transport part for one time step:
\begin{equation}\label{transport_part}
    \begin{gathered}
        \partial_t\rho + v\partial_x J = 0,\\
        \partial_t J + v\partial_x \rho = 0,
    \end{gathered}
\end{equation}
using an explicit upwind scheme. 

More specifically, let $\Delta t$ and $\Delta x$ be respectively the temporal and spatial mesh sizes; $\rho_m^l\approx\rho(x_m,t_l)$; $J_m^l\approx J(x_m,t_l)$.  
For the $l$th time step, the full discretization takes the form:

One first solves
\begin{equation*}
    \begin{gathered}
        \frac{\rho_m^* - \rho_m^l}{\Delta t} = 0,\nonumber\\
        \frac{J_m^* - J_m^l}{\Delta t} = -\frac{2}{\epsilon^2}(T_{0,m}J_m^*+T_{1,m} \rho_m^*) + v\left(1 - \frac{1}{\epsilon^2} \right)\frac{\rho_{m+1}^* - \rho_{m-1}^*}{2\Delta x}.\nonumber
    \end{gathered}
\end{equation*}
Let $\partial_x^{(c)}\rho_m^* := \dfrac{\rho_{m+1}^* - \rho_{m-1}^*}{2\Delta x}$. Then $\rho^*$ and $J^*$ can be computed explicitly from the known values at time step $l$ as
\begin{equation*}
    \begin{gathered}
        \rho_m^* = \rho_m^l,\nonumber\\
        J_m^* = \frac{\epsilon^2J_m^l-2\Delta t\, T_{1,m}\rho_m^l + v(\epsilon^2 - 1)\Delta t\partial_x^{(c)}\rho_m^l}{\epsilon^2+2\Delta tT_{0,m}}.\nonumber
    \end{gathered}
\end{equation*}
Then, we apply a first-order upwind scheme to Eq. \eqref{transport_part} to obtain $\rho_m^{l+1}$ and $J_m^{l+1}$:
\begin{equation}
    \begin{gathered}
        \frac{\rho_m^{l+1}-\rho_m^*}{\Delta t} + \frac{v}{2\Delta x} (J_{m+1}^* - J_{m-1}^*)=\frac{v}{2\Delta x} (\rho_{m+1}^* - 2\rho_{m}^* + \rho_{m-1}^*),\nonumber\\
        \frac{J_m^{l+1}-J_m^*}{\Delta t} + \frac{v}{2\Delta x} (\rho_{m+1}^* - \rho_{m-1}^*)=\frac{v}{2\Delta x} (J_{m+1}^* - 2J_{m}^* + J_{m-1}^*).
    \end{gathered}
\end{equation}
It has been shown in \cite{carrillo2013asymptotic} that this scheme is uniformly accurate and stable in $\epsilon$.

\bibliographystyle{abbrv} 
\bibliography{ref}

@article{willard2006signaling,
  title={Signaling pathways mediating chemotaxis in the social amoeba, Dictyostelium discoideum},
  author={Willard, Stacey S and Devreotes, Peter N},
  journal={European journal of cell biology},
  volume={85},
  number={9-10},
  pages={897--904},
  year={2006},
  publisher={Elsevier}
}

@article{de2016neutrophil,
  title={Neutrophil migration in infection and wound repair: going forward in reverse},
  author={De Oliveira, Sofia and Rosowski, Emily E and Huttenlocher, Anna},
  journal={Nature Reviews Immunology},
  volume={16},
  number={6},
  pages={378--391},
  year={2016},
  publisher={Nature Publishing Group UK London}
}

@article{keller1970initiation,
  title={Initiation of slime mold aggregation viewed as an instability},
  author={Keller, Evelyn F and Segel, Lee A},
  journal={Journal of theoretical biology},
  volume={26},
  number={3},
  pages={399--415},
  year={1970},
  publisher={Elsevier}
}

@article{stroock1974some,
  title={Some stochastic processes which arise from a model of the motion of a bacterium},
  author={Stroock, Daniel W},
  journal={Zeitschrift f{\"u}r Wahrscheinlichkeitstheorie und verwandte Gebiete},
  volume={28},
  number={4},
  pages={305--315},
  year={1974},
  publisher={Springer}
}

@article{alt1980biased,
  title={Biased random walk models for chemotaxis and related diffusion approximations},
  author={Alt, Wolgang},
  journal={Journal of mathematical biology},
  volume={9},
  number={2},
  pages={147--177},
  year={1980},
  publisher={Springer}
}

@article{othmer1988models,
  title={Models of dispersal in biological systems},
  author={Othmer, Hans G and Dunbar, Steven R and Alt, Wolfgang},
  journal={Journal of mathematical biology},
  volume={26},
  number={3},
  pages={263--298},
  year={1988},
  publisher={Springer}
}

@article{othmer2000diffusion,
  title={The diffusion limit of transport equations derived from velocity-jump processes},
  author={Othmer, Hans G and Hillen, Thomas},
  journal={SIAM Journal on Applied Mathematics},
  volume={61},
  number={3},
  pages={751--775},
  year={2000},
  publisher={SIAM}
}

@article{othmer2002diffusion,
  title={The diffusion limit of transport equations II: Chemotaxis equations},
  author={Othmer, Hans G and Hillen, Thomas},
  journal={SIAM Journal on Applied Mathematics},
  volume={62},
  number={4},
  pages={1222--1250},
  year={2002},
  publisher={SIAM}
}

@article{hellmuth2021multiscale,
  title={Multiscale convergence of the inverse problem for chemotaxis in the Bayesian setting},
  author={Hellmuth, Kathrin and Klingenberg, Christian and Li, Qin and Tang, Min},
  journal={Computation},
  volume={9},
  number={11},
  pages={119},
  year={2021},
  publisher={MDPI}
}

@article{saragosti2011directional,
  title={Directional persistence of chemotactic bacteria in a traveling concentration wave},
  author={Saragosti, J. and Calvez, V. and Bournaveas, N. and Perthame, B. and Buguin, A. and Silberzan, P.},
  journal={Proc. Natl. Acad. Sci. U.S.A.},
  volume={108},
  number={39},
  pages={16235--16240},
  year={2011},
  publisher={National Academy of Sciences}
}

@article{fister2008identification,
  title={Identification of a chemotactic sensitivity in a coupled system},
  author={Fister, K Renee and McCarthy, Maeve L},
  journal={Mathematical medicine and biology: a journal of the IMA},
  volume={25},
  number={3},
  pages={215--232},
  year={2008},
  publisher={OUP}
}

@article{carrillo2025sparse,
    AUTHOR = {Carrillo, Jos\'e{} A. and Estrada-Rodriguez, Gissell and
              Mikol\'as, L\'aszl\'o{} and Tang, Sui},
     TITLE = {Sparse identification of nonlocal interaction kernels in
              nonlinear gradient flow equations via partial inversion},
   JOURNAL = {Math. Models Methods Appl. Sci.},
  FJOURNAL = {Mathematical Models and Methods in Applied Sciences},
    VOLUME = {35},
      YEAR = {2025},
    NUMBER = {5},
     PAGES = {1073--1131},
      ISSN = {0218-2025,1793-6314},
   MRCLASS = {65M32 (35Q70 35R30 65F22 70-08 70F17)},
  MRNUMBER = {4896516},
       DOI = {10.1142/S0218202525500137},
       URL = {https://doi.org/10.1142/S0218202525500137},
}

@article{lang2023identifiability,
  title={Identifiability of interaction kernels in mean-field equations of interacting particles},
  author={Lang, Quanjun and Lu, Fei},
  journal={Foundations of Data Science},
  volume={5},
  number={4},
  pages={480--502},
  year={2023},
  publisher={Foundations of Data Science}
}

@article{lu2021learning,
  title={Learning interaction kernels in heterogeneous systems of agents from multiple trajectories},
  author={Lu, Fei and Maggioni, Mauro and Tang, Sui},
  journal={Journal of Machine Learning Research},
  volume={22},
  number={32},
  pages={1--67},
  year={2021}
}

@article{li2021identifiability,
  title={On the identifiability of interaction functions in systems of interacting particles},
  author={Li, Zhongyang and Lu, Fei and Maggioni, Mauro and Tang, Sui and Zhang, Cheng},
  journal={Stochastic Processes and their Applications},
  volume={132},
  pages={135--163},
  year={2021},
  publisher={Elsevier}
}

@article{tang2024identifiability,
  title={On the Identifiability of Nonlocal Interaction Kernels in First-Order Systems of Interacting Particles on Riemannian Manifolds},
  author={Tang, Sui and Tuerkoen, Malik and Zhou, Hanming},
  journal={SIAM Journal on Applied Mathematics},
  volume={84},
  number={5},
  pages={2067--2086},
  year={2024},
  publisher={SIAM}
}

@article{CTPortoErcole,
    AUTHOR = {Carrillo, J. A. and Toscani, G.},
     TITLE = {Contractive probability metrics and asymptotic behavior of
              dissipative kinetic equations},
   JOURNAL = {Riv. Mat. Univ. Parma (7)},
  FJOURNAL = {Rivista di Matematica della Universit\`a{} di Parma. Serie 7},
    VOLUME = {6},
      YEAR = {2007},
     PAGES = {75--198},
      ISSN = {0035-6298},
   MRCLASS = {82C40 (35B40 35F20 60B10 82-01)},
  MRNUMBER = {2355628},
MRREVIEWER = {Carlo\ Cercignani},
}

@article{CGT99,
    AUTHOR = {Carlen, E. A. and Gabetta, E. and Toscani, G.},
     TITLE = {Propagation of smoothness and the rate of exponential
              convergence to equilibrium for a spatially homogeneous
              {M}axwellian gas},
   JOURNAL = {Comm. Math. Phys.},
  FJOURNAL = {Communications in Mathematical Physics},
    VOLUME = {199},
      YEAR = {1999},
    NUMBER = {3},
     PAGES = {521--546},
      ISSN = {0010-3616,1432-0916},
   MRCLASS = {82C40 (76P05)},
  MRNUMBER = {1669689},
MRREVIEWER = {Carlo\ Cercignani},
       DOI = {10.1007/s002200050511},
       URL = {https://doi.org/10.1007/s002200050511},
}

@article{keller1971model,
  title={Model for chemotaxis},
  author={Keller, Evelyn F and Segel, Lee A},
  journal={Journal of theoretical biology},
  volume={30},
  number={2},
  pages={225--234},
  year={1971},
  publisher={Elsevier}
}

@article{keller1971traveling,
  title={Traveling bands of chemotactic bacteria: a theoretical analysis},
  author={Keller, Evelyn F and Segel, Lee A},
  journal={Journal of theoretical biology},
  volume={30},
  number={2},
  pages={235--248},
  year={1971},
  publisher={Elsevier}
}

@article{chalub2004kinetic,
  title={Kinetic models for chemotaxis and their drift-diffusion limits},
  author={Chalub, Fabio ACC and Markowich, Peter A and Perthame, Beno{\^\i}t and Schmeiser, Christian},
  journal={Monatshefte f{\"u}r Mathematik},
  volume={142},
  number={1},
  pages={123--141},
  year={2004},
  publisher={Springer}
}

@article{giometto2015generalized,
  title={Generalized receptor law governs phototaxis in the phytoplankton Euglena gracilis},
  author={Giometto, Andrea and Altermatt, Florian and Maritan, Amos and Stocker, Roman and Rinaldo, Andrea},
  journal={Proceedings of the National Academy of Sciences},
  volume={112},
  number={22},
  pages={7045--7050},
  year={2015},
  publisher={National Academy of Sciences}
}

@article{karalashvili2011identification,
  title={Identification of transport coefficient models in convection-diffusion equations},
  author={Karalashvili, Maka and Gro{\ss}, Sven and Marquardt, Wolfgang and Mhamdi, Adel and Reusken, Arnold},
  journal={SIAM Journal on Scientific Computing},
  volume={33},
  number={1},
  pages={303--327},
  year={2011},
  publisher={SIAM}
}

@article{lai2019inverse,
  title={Inverse problems for the stationary transport equation in the diffusion scaling},
  author={Lai, Ru-Yu and Li, Qin and Uhlmann, Gunther},
  journal={SIAM Journal on Applied Mathematics},
  volume={79},
  number={6},
  pages={2340--2358},
  year={2019},
  publisher={SIAM}
}

@article{bal2008inverse,
  title={Inverse transport with isotropic sources and angularly averaged measurements},
  author={Bal, Guillaume and Langmore, Ian and Monard, Fran{\c{c}}ois},
  journal={Inverse Probl. Imaging},
  volume={2},
  number={1},
  pages={23--42},
  year={2008}
}

@article{chen2018stability,
  title={Stability of inverse transport equation in diffusion scaling and Fokker--Planck limit},
  author={Chen, Ke and Li, Qin and Wang, Li},
  journal={SIAM Journal on Applied Mathematics},
  volume={78},
  number={5},
  pages={2626--2647},
  year={2018},
  publisher={SIAM}
}

@article{hellmuth2025reconstructing,
  title={Reconstructing the kinetic chemotaxis kernel using macroscopic data: well-posedness and ill-posedness},
  author={Hellmuth, Kathrin and Klingenberg, Christian and Li, Qin and Tang, Min},
  journal={SIAM Journal on Applied Mathematics},
  volume={85},
  number={2},
  pages={613--635},
  year={2025},
  publisher={SIAM}
}

@article{ahmed2008experimental,
  title={Experimental verification of the behavioral foundation of bacterial transport parameters using microfluidics},
  author={Ahmed, Tanvir and Stocker, Roman},
  journal={Biophysical journal},
  volume={95},
  number={9},
  pages={4481--4493},
  year={2008},
  publisher={Elsevier}
}

@article{perez2022microfluidic,
  title={Microfluidic devices for studying bacterial taxis, drug testing and biofilm formation},
  author={P{\'e}rez-Rodr{\'\i}guez, Sandra and Garc{\'\i}a-Aznar, Jos{\'e} Manuel and Gonzalo-Asensio, Jes{\'u}s},
  journal={Microbial Biotechnology},
  volume={15},
  number={2},
  pages={395--414},
  year={2022},
  publisher={Wiley Online Library}
}

@article{carrillo2013asymptotic,
  title={An asymptotic preserving scheme for the diffusive limit of kinetic systems for chemotaxis},
  author={Carrillo, Jos{\'e} A and Yan, Bokai},
  journal={Multiscale Modeling \& Simulation},
  volume={11},
  number={1},
  pages={336--361},
  year={2013},
  publisher={SIAM}
}

@article{kalwarczyk2012biologistics,
  title={Biologistics—diffusion coefficients for complete proteome of Escherichia coli},
  author={Kalwarczyk, Tomasz and Tabaka, Marcin and Holyst, Robert},
  journal={Bioinformatics},
  volume={28},
  number={22},
  pages={2971--2978},
  year={2012},
  publisher={Oxford University Press}
}

@book{ikeda2014stochastic,
  title={Stochastic differential equations and diffusion processes},
  author={Ikeda, Nobuyuki and Watanabe, Shinzo},
  volume={24},
  year={2014},
  publisher={Elsevier}
}

@book{carmona2016lectures,
  title={Lectures on BSDEs, stochastic control, and stochastic differential games with financial applications},
  author={Carmona, Ren{\'e}},
  year={2016},
  publisher={SIAM}
}

@article{berg1972chemotaxis,
  title={Chemotaxis in Escherichia coli analysed by three-dimensional tracking},
  author={Berg, Howard C and Brown, Douglas A},
  journal={nature},
  volume={239},
  number={5374},
  pages={500--504},
  year={1972},
  publisher={Nature Publishing Group UK London}
}

@article{bal2010inverse,
  title={Inverse diffusion theory of photoacoustics},
  author={Bal, Guillaume and Uhlmann, Gunther},
  journal={Inverse Problems},
  volume={26},
  number={8},
  pages={085010},
  year={2010}
}

@article{lai2022inverse,
  title={Inverse transport and diffusion problems in photoacoustic imaging with nonlinear absorption},
  author={Lai, Ru-Yu and Ren, Kui and Zhou, Ting},
  journal={SIAM Journal on Applied Mathematics},
  volume={82},
  number={2},
  pages={602--624},
  year={2022},
  publisher={SIAM}
}

@article{calvez2015confinement,
  title={Confinement by biased velocity jumps: aggregation of Escheria coli},
  author={Calvez, Vincent and Raoul, Gael and Schmeiser, Christian},
  journal={Kinetic and Related Models},
  volume={8},
  number={4},
  pages={651--666},
  year={2015},
  publisher={Kinetic and Related Models}
}

@article{si2014pathway,
  title={A pathway-based mean-field model for E. coli chemotaxis: Mathematical derivation and its hyperbolic and parabolic limits},
  author={Si, Guangwei and Tang, Min and Yang, Xu},
  journal={Multiscale Modeling \& Simulation},
  volume={12},
  number={2},
  pages={907--926},
  year={2014},
  publisher={SIAM}
}

@article{sun2017macroscopic,
  title={Macroscopic limits of pathway-based kinetic models for E. coli chemotaxis in large gradient environments},
  author={Sun, Weiran and Tang, Min},
  journal={Multiscale Modeling \& Simulation},
  volume={15},
  number={2},
  pages={797--826},
  year={2017},
  publisher={SIAM}
}

@article{chen2014guaranteed,
  title={Guaranteed sparse signal recovery with highly coherent sensing matrices},
  author={Chen, Guangliang and Divekar, Atul and Needell, Deanna},
  journal={Sampling Theory in Signal and Image Processing},
  volume={13},
  number={1},
  pages={91--109},
  year={2014},
  publisher={Springer}
}

@article{xue2021individual,
  title={Individual based models exhibiting l{\'e}vy-flight type movement induced by intracellular noise},
  author={Xue, Xiaoru and Tang, Min},
  journal={Journal of Mathematical Biology},
  volume={83},
  number={3},
  pages={27},
  year={2021},
  publisher={Springer}
}

@article{xue2025crossover,
  title={Crossover from ballistic transport to normal diffusion: a kinetic view},
  author={Xue, Zhe and Sun, Weiran and Zhou, Zhennan and Tang, Min},
  journal={arXiv preprint arXiv:2501.02240},
  year={2025}
}

@article{patteson2015running,
  title={Running and tumbling with E. coli in polymeric solutions},
  author={Patteson, Alison E and Gopinath, Arvind and Goulian, Mark and Arratia, Paulo E},
  journal={Scientific reports},
  volume={5},
  number={1},
  pages={15761},
  year={2015},
  publisher={Nature Publishing Group UK London}
}

\end{document}